\documentclass{amsart}[12pt]
\usepackage[curve,matrix,arrow,frame,tips]{xy}
\usepackage{mathrsfs}
\usepackage{enumerate}
\usepackage{amssymb}
\usepackage{stmaryrd}
\usepackage{manfnt}
\usepackage[pdfborder={0 0 0}]{hyperref}

\DeclareMathOperator{\colim}{colim}
\DeclareMathOperator{\Ho}{Ho}
\DeclareMathOperator{\gl}{gl}
\DeclareMathOperator{\Hom}{Hom}
\DeclareMathOperator{\End}{End}
\DeclareMathOperator{\Aut}{Aut}
\DeclareMathOperator{\map}{map}
\DeclareMathOperator{\Id}{Id}
\DeclareMathOperator{\sh}{sh}

\newcommand{\mC}{{\mathbb C}}
\newcommand{\mH}{{\mathbb H}}
\newcommand{\mK}{{\mathbb K}}
\newcommand{\mR}{{\mathbb R}}
\newcommand{\Ab}{{\mathcal A}b}
\newcommand{\Cc}{{\mathcal C}}
\newcommand{\Uc}{{\mathcal U}}
\newcommand{\Vc}{{\mathcal V}}
\newcommand{\bF}{\mathbf{F}}
\newcommand{\bGr}{\mathbf{Gr}}
\newcommand{\bL}{\mathbf{L}}
\newcommand{\bO}{\mathbf{O}}
\newcommand{\bSO}{{\mathbf{SO}}}
\newcommand{\bSp}{{\mathbf{Sp}}}
\newcommand{\bSU}{{\mathbf{SU}}}
\newcommand{\bT}{{\mathbf T}}
\newcommand{\bU}{{\mathbf U}}
\newcommand{\bBO}{{\mathbf{BO}}}
\newcommand{\bbO}{{\mathbf{bO}}}
\newcommand{\bMO}{{\mathbf{MO}}}
\newcommand{\bmO}{{\mathbf{mO}}}
\newcommand{\spec}{Sp}
\newcommand{\spc}{spc}
\newcommand{\iso}{\cong}
\newcommand{\sm}{\wedge}
\newcommand{\tensor}{\otimes}
\newcommand{\widebar}{\overline}
\newcommand{\xra}{\xrightarrow}
\newcommand{\xla}{\xleftarrow}
\newcommand{\GH}{{\mathcal{GH}}}
\newcommand{\upi}{{\underline \pi}}
\newcommand{\td}[1]{\langle #1\rangle}
\renewcommand{\to}{\longrightarrow}
\renewcommand{\Re}{\text{Re}}

\numberwithin{equation}{section}
\newtheorem{theorem}[equation]{Theorem}
\newtheorem{cor}[equation]{Corollary}
\newtheorem{prop}[equation]{Proposition}
 
\theoremstyle{definition}
\newtheorem{defn}[equation]{Definition}
\newtheorem{rk}[equation]{Remark}
\newtheorem{eg}[equation]{Example}
\newtheorem{construction}[equation]{Construction}
\newtheorem{notation}[equation]{Notation}

\newcommand{\Danger}{%
   \par\penalty-500\begingroup\parindent=0.0pt\clubpenalty=10000%
   \def\par{\endgraf\endgroup}\leavevmode
   \hangindent=2.4em\relax
   \hangafter=-2\relax
   \hbox to 0pt{\hss\relax
      \vbox to 0pt{\vskip-8pt\relax
        \hbox
          {\relax\dbend}
\vss}\kern.70em}}

\begin{document}

\title{Global stable splittings of Stiefel manifolds}

\date{\today; 2020 MSC 55N91, 55P91; keywords: global homotopy theory, stable splitting, Stiefel manifold}

\author{Stefan Schwede}
\address{Mathematisches Institut, Universit\"at Bonn, Germany}
\email{schwede@math.uni-bonn.de}

\begin{abstract}
  We prove global equivariant refinements of Miller's
  stable splittings of the infinite orthogonal, unitary and symplectic
  groups, and more generally of the spaces $O/O(m)$, $U/U(m)$ and $Sp/Sp(m)$.
  As such, our results encode compatible equivariant stable splittings,
  for all compact Lie groups, of specific equivariant refinements of these spaces.

  In the unitary and symplectic case, we also take the actions
  of the Galois groups into account. To properly formulate these Galois-global statements,
  we introduce a generalization of global stable homotopy theory in the presence
  of an extrinsic action of an additional topological group.
\end{abstract}

\maketitle

\setcounter{tocdepth}{1}
\tableofcontents

\section*{Introduction}

The purpose of this article is to establish global equivariant refinements
of Haynes Miller's stable splittings of the infinite orthogonal, unitary and symplectic
groups, and more generally of the spaces $O/O(m)$, $U/U(m)$ and $Sp/Sp(m)$.
In his classic paper \cite{miller}, Miller showed that a filtration by the size of
the `generalized +1 eigenspaces' on the finite-dimensional
 real, complex and quaternionic Stiefel manifolds
splits after a suitable suspension.
Miller then obtains stable splittings of the infinite dimensional Stiefel manifolds $O/O(m)$,
$U/U(m)$ and $Sp/Sp(m)$ by passing to colimits.
For $m=0$ and $m=1$, these include stable splittings
of the spaces underlying the groups $O$, $SO$, $U$, $S U$ and $Sp$.

Global equivariant homotopy theory is, informally speaking, equivariant homotopy theory
with simultaneous actions of all compact Lie groups.
Unstable global homotopy theory comes up in different incarnations,
for example as the homotopy theory of topological stacks and orbispaces \cite{gepner-henriques},
or spaces with an action of the `universal compact Lie group' \cite{schwede:universal_Lie}.
Some of the protagonists of Miller's splittings have very natural and interesting equivariant and
global refinements. For instance, the unitary groups of all hermitian inner product spaces
form a global space $\bU$ that refines the infinite unitary group,
and whose $G$-equivariant homotopy type is
that of the unitary group of a complete complex $G$-universe.
The global space $\bU$ features in a global refinement of Bott periodicity,
a global equivalence $\Omega^2\bU\simeq \bU$
that encodes equivariant Bott periodicity for all compact Lie groups at once,
see \cite[Theorem 2.5.41]{schwede:global}.
There are similarly  natural global refinements
$\bO$, $\bSO$, $\bSU$ and $\bSp$ of the spaces $O$, $SO$, $S U$ and $Sp$;
more generally, we introduce global refinements $\bO/m$, $\bU/m$ and $\bSp/m$
of the spaces $O/O(m)$, $U/U(m)$ and $Sp/Sp(m)$, see Example \ref{eg:L/m} below.

In \cite{schwede:global}, the author has developed a framework for stable global homotopy theory;
one of the upshots is the compactly generated and tensor-triangulated
global stable homotopy category \cite[Section 4.4]{schwede:global}
that forgets to the homotopy categories of genuine $G$-spectra for all compact Lie groups $G$;
this global stable homotopy category is the home for the splittings proved in this paper.
The global stable homotopy category
exhibits all the `genuine' (as opposed to `naive') features of equivariant stable homotopy theory,
it is different from the stabilization obtained by inverting ordinary suspension,
and its objects represent genuine cohomology theories on orbifolds \cite{juran}.
The following is the main result of this paper;
we prove it as Theorem \ref{thm:stably split L/m}, where we also specify the splitting morphisms.

\medskip

\noindent
{\bf Theorem.} For every $m\geq 0$ there is an isomorphism
\[
   \Sigma^\infty_+ \bO/m \ \iso \
   {\bigvee}_{k\geq 0}\ \Sigma^\infty  (\bGr_k^\mR)^{\nu(k,m)\oplus\mathfrak{ad}(k)}\]
in the global stable homotopy category, an isomorphism
\[  \Sigma^\infty_+ \bU/m \ \iso \
                          {\bigvee}_{k\geq 0}\  \Sigma^\infty  (\bGr_k^\mC)^{\nu(k,m)\oplus \mathfrak{ad}(k)} \]
in the $G(\mC)$-global stable homotopy category, and an isomorphism
\[  \Sigma^\infty_+ \bSp/m \ \iso \ {\bigvee}_{k\geq 0}\ \Sigma^\infty  (\bGr_k^\mH)^{\nu(k,m)\oplus \mathfrak{ad}(k)} \]
in the $G(\mH)$-global stable homotopy category.

\medskip

In the unitary and symplectic statements,
$G(\mC)$ and $G(\mH)$ denote the Galois groups of $\mC$ and $\mH$ over $\mR$.
The flavor of global homotopy theory that encodes an additional extrinsic symmetry group
transcends the theory of \cite{schwede:global};
we devote Appendix~\ref{sec:C-global} to developing
the basics of $C$-global homotopy theory, for any topological group $C$,
including the triangulated $C$-global stable homotopy category.
The Galois-global splittings are more highly structured,
and imply splittings in the global stable homotopy category of \cite{schwede:global}.
On the right hand side of the splitting, $(\bGr_k^\mR)^V$ denotes the global Thom
space of an $O(k)$-representation $V$,
a specific global refinement of the Thom space of the vector bundle over $B O(k)$
associated to $V$, see Example \ref{eg:global Thom}.
And similarly for $(\bGr_k^\mC)^V$ and $(\bGr_k^\mH)^V$, where now $V$
is a representation of the unitary group $U(k)$ or the symplectic group $S p(k)$,
respectively.
Moreover, $\nu(k,m)$ is an $m$-fold direct sum of copies of the tautological representation
of $O(k)$, $U(k)$ or $S p(k)$ on $\mR^k$, $\mC^k$ or $\mH^k$, respectively;
Miller \cite{miller} refers to these as the {\em canonical representations}.
And $\mathfrak{ad}(k)$ is the adjoint representation of $O(k)$, $U(k)$ or $Sp(k)$, respectively.

The special case $m=0$ yields stable global splittings of $\bO$, $\bU$ and $\bSp$,
see Theorems \ref{thm:stably split O} and \ref{thm:stably split U and Sp}.
And the special case $m=1$ yields stable global splittings of $\bSO$ and $\bSU$,
see Theorem  \ref{thm:stably split SO}.
The global stable homotopy category comes with a highly forgetful functor
to the non-equivariant stable homotopy category, see \cite[Section 4.5]{schwede:global};
applying this forgetful functor to our main result Theorem \ref{thm:stably split L/m}
returns Miller's non-equivariant splitting \cite[Corollary D]{miller}.

There are natural stabilization morphisms $\bO/m\to\bO/(m+1)$ obtained from the
preferred embeddings $\mR^m\to\mR^{m+1}$, and similarly in the unitary and symplectic cases;
we write $\bO/\infty$, $\bU/\infty$ and $\bSp/\infty$ for the global spaces
obtained in the colimit over $m$.
Non-equivariantly, these colimit spaces are contractible, so there is no incentive to study their
homotopy types any further. Equivariantly and globally, however, these limit objects
are interesting and highly non-trivial, see Remark \ref{rk:O/infty}.
In Theorem \ref{thm:stably split L/infty} we formulate global stable splittings
of the global spaces $\bO/\infty$, $\bU/\infty$ and $\bSp/\infty$,
ultimately obtained from the previous splittings by passing to homotopy colimits.

\smallskip

We end this introduction with some comments
intended to prevent possible misconceptions about the nature of our global splitting result.
Firstly,
the forgetful functor from the global stable homotopy category to
the non-equivariant stable homotopy category has a left adjoint,
an exact functor of triangulated categories that preserves infinite sums,
compare \cite[Theorem 4.5.1]{schwede:global}.
So one can apply this left adjoint to Miller's original non-equivariant splittings,
yielding global refinements in a formal way.
However, the global spaces $\bO/m$, $\bU/m$ and $\bSp/m$ are {\em not}
in the essential image of the left adjoint to the forgetful functor,
so our splitting  is not a formal consequence of Miller's.

Secondly, 
for the purposes of this paper, it is essential that the stable global homotopy theory of
\cite{schwede:global}, and more generally the $C$-global stable homotopy theory
of Appendix \ref{sec:C-global}, is not just the `naive' stabilization of unstable global homotopy theory.
Indeed, as we explain in Construction \ref{con:splitting morphism},
the global splitting morphisms make essential use of certain equivariant splittings
of the top cells of Stiefel manifolds that exist after smashing with non-trivial representations.
Consequently, those equivariant splittings require a genuine (as opposed to naive)
equivariant stabilization,
and our subsequent arguments would not work in the homotopy theory of orbispaces with only
the ordinary suspension inverted.

Thirdly,
Miller proves non-equivariant stable splittings of the finite-dimensional
real, complex and quaternionic Stiefel manifolds.
He then obtains the stable splitting of $O/O(m)$,  $U/U(m)$ and $Sp/Sp(m)$
by passing to colimits.
Crabb \cite{crabb:U(n) and Omega} and Ullman \cite{ullman:equivariant_miller}
have obtained certain equivariant refinements of some of Miller's splittings
for certain Stiefel manifolds of finite-dimensional representations of specific compact Lie groups.
Our results are in an entirely different direction.
In the global context, the finite-dimensional Stiefel mani\-folds are not underlying
any interesting global homotopy types (other than left or right induced).
So I cannot think of a meaningful global splitting that generalizes the
stable splittings of the finite-dimensional Stiefel manifolds due to Miller, Crabb and Ullman.
There is a certain tradeoff: the price for obtaining global results is to
work with infinite-dimensional objects.

\smallskip

{\bf Organization.}
We start in Section \ref{sec:warmup} by explaining a special case of our main result,
the global stable splitting of the orthogonal space $\bO$
made from the orthogonal groups. The short Section \ref{sec:warmup} is logically redundant,
and intended as a gentle introduction to the main ideas in an important
special case that is technically and notationally simpler.
In Section \ref{sec:filtration} we introduce 
the orthogonal spaces $\bO/m$, $\bU/m$ and $\bSp/m$
made from real, complex and quaternionic Stiefel manifolds, respectively.
We review their `eigenspace filtrations'
and recall Miller's identification of the open strata of the filtration
as total spaces of specific vector bundles over Grassmannians.
Section \ref{sec:filtration} does not contain any new mathematics;
its purpose is to recast the known facts about the eigenspace filtration
in a form tailored to our purposes, while making all inherent symmetries explicit.
Section \ref{sec:splitting} is the heart of the paper, culminating in the
statement and proof of our main result, Theorem \ref{thm:stably split L/m};
here we construct the morphisms in the
Galois-equivariant global stable homotopy categories
that stably split the eigenspace filtrations of $\bO/m$, $\bU/m$ and $\bSp/m$.
In Section \ref{sec:limit} we explain how to pass to the colimit over $m$,
and obtain splittings of the orthogonal spaces
$\bO/\infty$, $\bU/\infty$ and $\bSp/\infty$.

This paper contains two appendices. 
Appendix \ref{sec:C-global} is a brief introduction to $C$-global homotopy theory,
where $C$ is a topological group.
This appendix is needed to give content to the Galois-equivariant refinements
of the stable global splittings of $\bU/m$ and $\bSp/m$;
for these application we are interested
in the special case where $C$ is the Galois group $G(\mC)$ of $\mC$ over $\mR$
(a discrete group of order 2),
or where $C$ is the group $G(\mH)$ of $\mR$-algebra automorphism of the quaternions
(a compact Lie group isomorphic to $S O(3)$).
However, the basic theory works just as well over arbitrary topological groups,
so we develop it in that generality.
In Appendix \ref{sec:C-global}
we in particular set up the triangulated $C$-global stable homotopy category $\GH_C$,
and identify global Thom spaces over global classifying spaces as representing objects
for equivariant homotopy groups, see Theorem \ref{thm:Bgl alpha represents}.
Appendix \ref{sec:linear algebra} provides proofs of the linear algebra facts
used in the main part of the paper.
I make no claim to originality for anything in Appendix \ref{sec:linear algebra};
its purpose is to show that all relevant arguments from linear algebra over $\mR$ and $\mC$
can be adapted to the quaternion context.

\smallskip

{\bf Acknowledgments.}
The research in this paper was supported by the DFG Schwerpunktprogramm 1786
`Homotopy Theory and Algebraic Geometry' (project ID SCHW 860/1-1)
and by the Hausdorff Center for Mathematics
at the University of Bonn (DFG GZ 2047/1, project ID 390685813).
I would like to thank Tobias Lenz for valuable feedback on this paper,
and for several suggestions for improvements.

\section{Warm-up: the global stable splitting of \texorpdfstring{$\bO$}{O}}
\label{sec:warmup}

In this short section we sketch the global stable splitting of the orthogonal space $\bO$
made from the orthogonal groups. This section is logically redundant,
as the splitting of Theorem \ref{thm:stably split O} is a special case
of our main result, Theorem \ref{thm:stably split L/m}.
I am prepending this section because
it already exhibits all the key features of the later arguments in a simpler form,
without the two additional layers of complexity arising from
the extra parameter $m$ in the more general splittings of $\bO/m$, $\bU/m$ and $\bSp/m$,
and the Galois-global embellishment in the complex and quaternionic setting.
All necessary tools for the global splitting of $\bO$ are
already contained in \cite{schwede:global}, and there is no need
to appeal to the more general $C$-global homotopy theory from Appendix \ref{sec:C-global}.
Said differently, this brief section is intended as a gentle introduction to the main ideas in an important
special case that is technically and notationally simpler.

We will freely use the language and results from \cite{schwede:global}.
In particular, we model unstable global homotopy theory by
orthogonal spaces in the sense of  \cite[Definition 1.1.1]{schwede:global},
i.e., continuous functors to spaces from the category
$\bL$ of finite-dimensional euclidean inner product spaces and $\mR$-linear isometric embeddings.
The category $\bL$  is also denoted $\mathscr I$ or $\mathcal I$ by other authors,
and orthogonal spaces are also known as $\mathscr I$-functors,
$\mathscr I$-spaces or $\mathcal I$-spaces.
And we model stable global homotopy theory by orthogonal spectra
with respect to global equivalences as defined in \cite[Definition 4.1.3]{schwede:global},
i.e., morphisms that induce isomorphisms of equivariant homotopy groups for all compact Lie groups.

\begin{eg}
  The orthogonal space $\bO$ made from the orthogonal groups is
  a particular global refinement of the infinite orthogonal group $O=\bigcup_{n\geq 1}O(n)$.
  The value of the $\bO(V)$ at an inner product space $V$ is simply the orthogonal group of $V$.
  A linear isometric embedding $\psi:V\to W$ 
  is sent to the continuous group homomorphism $\bO(\psi):\bO(V)\to \bO(W)$ defined by
  \[ \bO(\psi)(f)( \psi(v) + w )\ = \ \psi(f(v)) + w\ ,\]
  where  $v\in V$ and $w\in W- \psi(V)$.
  So informally speaking, $\bO(\psi)$ is `conjugation by $\psi$' and direct sum
  with the identity on the orthogonal complement of the image of $\psi$.
  For more detailed information about the global homotopy type of $\bO$ we refer
  the reader to the discussion in \cite[Example 2.3.6]{schwede:global};
  suffice it to say here that the underlying $G$-homotopy type of $\bO$,
  for a compact Lie group $G$, is the orthogonal group of a complete $G$-universe
  (i.e., orthogonal automorphisms of the underlying inner product space with conjugation action
  by $G$).

  A filtration of $\bO$ by orthogonal subspaces $F_k\bO$ is defined by
  the size of the +1 eigenspaces.
  At an inner product space $V$, we set
  \[ (F_k\bO)(V)\ = \
    \{ f\in O(V)\ : \  \dim(\ker(f-\Id)^\perp)\leq  k\}\ , \]
  where $(-)^\perp$ denotes the orthogonal complement.
  For fixed $k$ and varying $V$, these spaces are closed under the structure maps;
  so they define an orthogonal subspace $F_k\bO$ of $\bO$.
  Altogether we obtain an  ascending sequence of orthogonal spaces $F_k\bO$  that exhausts $\bO$.
  This filtration happens to be the skeleton filtration of the orthogonal space $\bO$
  in the sense of \cite[Definition 1.2.2]{schwede:global},
  but we won't show this fact because it plays no role for our arguments.
\end{eg}

Before we attack the global splitting of the above eigenspace filtration of $\bO$,
we recall the identification of the $k$-th subquotient
with the global Thom space of the adjoint representation of the $k$-th orthogonal group.
This identification is classical, going back at least to Frankel \cite{frankel},
and versions of it appear in  \cite[Theorem B]{miller} and \cite[Proposition 1.7]{crabb:U(n) and Omega}.
We write  $\bGr_k$ for the $k$-th {\em Grassmannian},
the orthogonal space whose value $\bGr_k(V)$ is the
Grassmannian of $k$-planes in an inner product space $V$.
The structure maps take images under linear isometric embeddings.
The orthogonal space $\bGr_k$ is a global classifying space,
in the sense of \cite[Definition 1.1.27]{schwede:global},
of the orthogonal group $O(k)$.
We write
\[  \mathfrak{ad}(k)\ = \ \{ X\in M(k,k)\ : \  X^t= - X\}  \ ,\]
for the {\em adjoint representation} of $O(k)$,
i.e., the $\mR$-vector space of skew-symmetric real matrices of size $k\times k$,
with $O(k)$ acting by conjugation.
The {\em global Thom space} over $\bGr_k$ associated to $\mathfrak{ad}(k)$
is the based orthogonal space
\[  \bGr_k^{\mathfrak{ad}(k)}\ = \  \bL(\mR^k,-)_+\sm_{O(k)} S^{\mathfrak{ad}(k)}\ .  \]
The {\em Cayley transform} provides an open embedding 
\begin{equation}\label{eq:o(k)2O(k)}
  \mathfrak c\ : \ \mathfrak{ad}(k)\ \to \ O(k)  \ , \quad
  \mathfrak c(X)\ = \  (X/2-1) (X/2+1)^{-1}
\end{equation}
onto the subspace of $O(k)$ of those matrices that do not have +1 as an eigenvalue.
The embedding is $O(k)$-equivariant for the conjugation actions on both sides.
For varying inner product spaces $V$,
the collapse maps associated to the open embeddings
\[   \bL(\mR^k,V)\times_{O(k)} \mathfrak{ad}(k) \ \to \ (F_k \bO)(V)\ ,
  \quad  [\psi, X]\ \longmapsto \ \bO(\psi)(\mathfrak c(X))\ ,    \]
form a morphism of orthogonal spaces
\begin{equation}\label{eq:subquotient identify O}
 \Psi\ : \ F_k\bO \ \to \ \bGr_k^{\mathfrak{ad}(k)} 
\end{equation}
that factors over an isomorphism $F_k\bO/ F_{k-1}\bO \iso \bGr_k^{\mathfrak{ad}(k)}$.
This isomorphism is just a coordinate-free formulation of
a special case of Miller's \cite[Theorem B]{miller}.

On page 39 of \cite{crabb:U(n) and Omega},
Crabb gives a particularly elegant exposition
of Miller's method to stably split off the top cell of the unitary group $U(k)$
in a fully equivariant manner; we reproduce the argument, adapted to $O(k)$.
We write
\[\mathfrak{s a}(k)\ =\ \{ Z\in M(k,k)\ : \ Z^t = Z \}  \]
for the vector space of symmetric real matrices of size $k\times k$;
the notation anticipates the role of $\mathfrak{s a}(k)$ as the
self-adjoint endomorphisms of $\mR^k$ with respect to the standard inner product.
Then $M(k,k)=\mathfrak{ad}(k)\oplus \mathfrak{s a}(k)$.
A basic linear algebra fact is that the smooth map
\[  O(k)\times \mathfrak{s a}(k) \ \to \ M(k,k)\ , \quad
  (A,Z)\ \longmapsto \ A\cdot\exp(-Z) \ = \ A\cdot {\sum}_{k\geq 0} (-Z)^k/k!\]
is an open embedding with image the subspace $G L_k$ of invertible matrices.
The embedding is $O(k)$-equivariant
for the conjugation action on $O(k)$, $\mathfrak{s a}(k)$ and $M(k,k)$,
so it provides an $O(k)$-equivariant collapse map
\begin{equation}\label{collapse_orthogonal}
    t_k\ : \ S^{\mathfrak{ad}(k)\oplus\mathfrak{s a}(k)}\ = \ S^{M(k,k)} \ \to \  O(k)_+\sm S^{\mathfrak{s a}(k)}
\end{equation}
in the opposite direction.
The same reasoning as in \cite[page 39]{crabb:U(n) and Omega}
shows that the composite
\[ 
  S^{\mathfrak{ad}(k)\oplus\mathfrak{s a}(k)}\  \xra{\ t_k\ } \
  O(k)_+\sm S^{\mathfrak{s a}(k)}\ \xra{\mathfrak c^\flat\sm S^{\mathfrak{s a}(k)}} \
  S^{\mathfrak{ad}(k)\oplus\mathfrak{s a}(k)}
\]
is $O(k)$-equivariantly based homotopic to the identity,
where $\mathfrak c^\flat$ is the collapse map based on the open embedding \eqref{eq:o(k)2O(k)},
see also the proof of Theorem \ref{thm:s_k,m splits L/m} below.

Everything explained so far is a review or reformulation of ideas and results from
\cite{crabb:U(n) and Omega, frankel, miller}.
 Now comes the point where we leverage $O(k)$-equivariant information
into global information, using the technology developed in \cite{schwede:global}.
Specifically, we use the global representability theorem for equivariant homotopy groups
from \cite[Theorem 4.4.3 (i)]{schwede:global}, or rather a slight extension
from integer-graded to representation-graded homotopy groups.
The collapse map \eqref{collapse_orthogonal}
gives rise to  an $\mathfrak{ad}(k)$-graded $O(k)$-equivariant homotopy class
\[  \td{t_k} \ \in \  \pi_{\mathfrak{ad}(k)}^{O(k)}( \Sigma^\infty_+ F_k\bO ) \ ,\]
the class represented by the composite
\begin{align*}
  S^{\mathfrak{ad}(k)\oplus\mathfrak{s a}(k)\oplus\nu_k} \xra{t_k\sm S^{\nu_k}}\  O(k)_+\sm S^{\mathfrak{s a}(k)\oplus\nu_k}\
&  = \ (F_k\bO)(\nu_k)_+\sm  S^{\mathfrak{s a}(k)\oplus\nu_k}  \\
    \xra{(F_k\bO)(i_2)\sm \Id} \ &(F_k\bO)(\mathfrak{s a}(k)\oplus\nu_k)_+\sm  S^{\mathfrak{s a}(k)\oplus\nu_k}
\ = \  (\Sigma^\infty_+ F_k\bO)(\mathfrak{s a}(k)\oplus\nu_k) \ .
\end{align*}
Here $\nu_k$ is the tautological $O(k)$-representation on $\mR^k$,
and $i_2:\nu_k\to\mathfrak{s a}(k)\oplus\nu_k$ is the embedding as the second summand.
The representation-graded generalization of \cite[Theorem 4.4.3 (i)]{schwede:global}
-- which is also a special case of Theorem \ref{thm:Bgl alpha represents} (i) below --
provides a unique morphism
\[ s_k\ : \ \Sigma^\infty \bGr_k^{\mathfrak{ad}(k)}\ \to \ \Sigma^\infty_+ F_k\bO  \]
in the global stable homotopy category characterized by the equation
\[ (s_k)_*(e_{O(k),\mathfrak{ad}(k)})\ = \ \td{t_k} \ ,\]
where $e_{O(k),\mathfrak{ad}(k)}$ is the tautological homotopy class
in $\pi_{\mathfrak{ad}(k)}^{O(k)}( \Sigma^\infty\bGr_k^{\mathfrak{ad}(k)})$.
The fact that $t_k$ is a stable section to $\mathfrak c^\flat:O(k)\to S^{\mathfrak{ad}(k)}$
translates into the relation
\[ (\Sigma^\infty\Psi)_*\td{t_k} \ = \ e_{O(k),\mathfrak{ad}(k)}\ ,\]
where $\Psi:F_k\bO\to\bGr_k^{\mathfrak{ad}(k)}$
is the collapse morphism \eqref{eq:subquotient identify O}.
Since the morphism $\Sigma^\infty\Psi\circ s_k$
takes the tautological class $e_{O(k),\mathfrak{ad}(k)}$ to itself,
another application of the representability property shows that $s_k$
is a section, in the global stable homotopy category, to the morphism
\[ \Sigma^\infty\Psi\ : \     \Sigma^\infty_+ F_k\bO\  \to  \ \Sigma^\infty \bGr_k^{\mathfrak{ad}(k)} \ .\]
So $s_k$ splits the distinguished triangle
\[
  \Sigma^\infty_+ F_{k-1}\bO \to 
  \Sigma^\infty_+ F_k\bO\xra{\Sigma^\infty \Psi}
  \Sigma^\infty \bGr_k^{\mathfrak{ad}(k)} \to
  (\Sigma^\infty_+ F_{k-1}\bO)\sm S^1
\]
in the global stable homotopy category.
Induction on $k$ and passage to the colimit over the eigenspace filtration
then yields:

\begin{theorem}\label{thm:stably split O}
   The morphism
  \[ \sum s_k \ :\ \bigvee_{k\geq 0}  \Sigma^\infty \bGr_k^{\mathfrak{ad}(k)} \ \to \
    \Sigma^\infty_+ \bO\]
  is an isomorphism in the global stable homotopy category.  
\end{theorem}

This concludes the warm-up. In the next two sections
we promote the arguments leading up to Theorem \ref{thm:stably split O}
to Galois-equivariant global statements about the complex and quaterionic analogs $\bU$ and $\bSp$
of $\bO$, while simultaneously generalizing to $\bO/m$, $\bU/m$ and $\bSp/m$ for $m\geq 0$.

\section{The unstable filtration}
\label{sec:filtration}

In this section we introduce and study
the orthogonal spaces $\bO/m$, $\bU/m$ and $\bSp/m$
made from real, complex and quaternionic Stiefel manifolds, respectively.
These global objects are refinements of the spaces $O/O(m)$, $U/U(m)$
and $Sp/Sp(m)$, respectively.
We review the `eigenspace filtrations' of $\bO/m$, $\bU/m$ and $\bSp/m$
and recall Miller's identification of the open strata of the filtration
as total spaces of specific vector bundles over Grassmannians.
This section does not contain any new mathematics; we only recast known results
in a form that is particularly convenient for our purposes,
and that makes all inherent symmetries explicit.

\medskip

We start by fixing some notation and terminology.
We write $\mH$ for the skew-field of quaternions; it is the $\mR$-vector space
with basis $(1,i,j,k)$ and multiplication determined by the relations
\[ i^2 = j^2 =k^2 = i j k = -1 \ . \]
Quaternion conjugation is the anti-involution $\lambda\mapsto \bar \lambda$ of $\mH$ given by
\[ \widebar{a + b i + c j + d k} \ = \ a - b i - c j - d k \]
for $a,b,c,d\in\mR$.
The quaternions contain the complex numbers $\mC=\mR\{1,i\}$ as a subfield,
and quaternion conjugation restricts to complex conjugation on $\mC$.
The subfield $\mR$ coincides with the fixed points of conjugation.
For a uniform treatment we will also talk about conjugation on $\mR$, where it is just the identity.

In order to treat the real, complex and quaternion Stiefel manifolds
simultaneously and uniformly,
we will write $\mK$ for any one of the skew-fields $\mR$, $\mC$ or $\mH$.
The particular choice of $\mK$ is always in the background,
but it is not always reflected in the notation.

\begin{notation}[Galois groups]
Throughout our discussion we keep track of the symmetries of the skew-fields $\mR$, $\mC$ and $\mH$:
we write
\[ G(\mK)\ = \ \Aut_\mR(\mK) \]
for the `Galois group' of $\mK$ over $\mR$,
i.e., the compact Lie group of $\mR$-algebra automorphisms of $\mK$.
Obviously, $G(\mR)$ is a trivial group, and $G(\mC)$ is the usual Galois group of
the field extension $\mC$ over $\mR$.
The `Galois group' $G(\mH)$ of the quaternions is isomorphic to $S O(3)$.
More specifically, the tautological action of $G(\mH)$ on $\mH$
identifies it with the special orthogonal group
of the 3-dimensional subspace $\mR\{i,j,k\}$
with respect to the euclidean inner product $\td{x,y}=\Re(\bar x y)$.
Also, every $\mR$-automorphism of $\mH$ is inner, i.e., given by conjugation
by an element of $Sp(1)=\{x\in \mH\colon \bar x\cdot x=1\}$;
since the center of $S p(1)$ consists of $\pm 1$, the map 
\[ \text{inner}\ : \   S p(1)/\{\pm 1\} \ \xra{\ \iso \ } \ G(\mH)   \]
that sends a unit quaternion to the associated inner automorphism is an isomorphism.
\end{notation}

\begin{defn}[Inner product spaces]\label{def:mK-inner}
  We let $\mK$ be one of the skew-fields $\mR$, $\mC$ or $\mH$.
  A {\em $\mK$-inner product space} is a finite-dimensional right $\mK$-module $W$
  equipped with an $\mR$-bilinear $\mK$-valued inner product
  \[ [-,-]\ : \ W\times W \ \to \ \mK \]
  that is sesquilinear and hermitian in the sense of the relations
  \[ [x\cdot\lambda,y\cdot\mu]\ = \ \bar\lambda \cdot [x,y]\cdot \mu
    \text{\qquad and\qquad} [y,x]\ =  \ \widebar{[x,y]}  \]
  for all $x,y\in W$ and $\lambda,\mu\in\mK$, and that is positive-definite,
  i.e., the real number $[x,x]$ is positive unless $x=0$.
  Then {\em length} of a vector $x\in W$  is the real number $|x|=\sqrt{[x,x]}$.
\end{defn}

When $\mK$ is $\mR$ or $\mC$, inner product spaces are usually called {\em euclidean vector spaces}
or {\em hermitian vector spaces}, respectively.
I do not know of standard terminology for $\mH$-inner product spaces.

\medskip

\Danger Since $\mR$ and $\mC$ are commutative, one can write the
factors $\bar\lambda$,  $[x,y]$ and $\mu$ in the sesquilinearity relation
in any other order, and the more common definitions display the scalars on the same side
of the inner product.
Multiplication in the quaternions is not commutative, so for $\mK=\mH$,
one has to beware of the order of the factors $\bar\lambda$,  $[x,y]$ and $\mu$.

\medskip

We will write $\mK^k$ for the standard $k$-dimensional $\mK$-vector space
endowed with the inner product 
\[ [x,y]\ = \  \bar x_1\cdot y_1 + \dots + \bar x_k\cdot y_k \ . \]
Given a $\mK$-subspace $V$ of a $\mK$-inner product space $W$, we write
\[  W - V \ = \ \{w\in W \ : \ [v,w]=0 \text{ for all $v\in V$}\} \]
for the {\em orthogonal complement}, another $\mK$-subspace such that $W=V\oplus(W-V)$.
If the ambient space $W$ is clear from the context, we might also use the notation $V^\perp$ for $W-V$.

Given two $\mK$-inner product spaces $V$ and $W$, we write
$\bL^\mK(V,W)$ for the Stiefel manifold of $\mK$-linear isometric embeddings from $V$ to $W$,
i.e., $\mK$-linear maps that respect the inner products, and are thus necessarily injective.
For $\mK=\mR$ we sometimes omit the superscript, i.e., $\bL(V,W)=\bL^\mR(V,W)$.

\begin{construction}
We recall the filtration of $\bL^\mK(W,W\oplus \mK^m)$ whose non-equivariant stable splitability
is the subject of Miller's paper \cite{miller}.
For $k\geq 0$ we set
\begin{equation}\label{eq:eigenspace filtration}
  \bF_k(W;m)\ = \
  \{ f\in \bL^\mK(W,W\oplus\mK^m)\ : \ \dim_\mK(\ker(f - i_1)^\perp)\leq  k\}\ ,   
\end{equation}
where $i_1:W\to W\oplus\mK^m$ is the embedding of the first summand.
For a linear isometric embedding $f:W\to W\oplus\mK^m$, the kernel of $f-i_1$
coincides with the kernel of the endomorphism $f_1-1:W\to W$, i.e.,
the $+1$ eigenspace of the first component of $f$.
We will therefore refer to this filtration as the {\em eigenspace filtration}. 

For a $\mK$-linear isometric embedding $\psi:V\to W$ we define `conjugation by $\psi$'
\begin{equation}\label{eq:conjugate by varphi}
  {^\psi(-)}\ : \ \bL^\mK(V,V\oplus\mK^m)\ \to \ \bL^\mK(W,W\oplus\mK^m)
\end{equation}
by
\[ (^\psi f)( \psi(v) + w )\ = \ (\psi\oplus\mK^m)(f(v)) + (w,0)\ ,\]
where $v\in V$ and $w\in W- \psi(V)$.
This continuous map is a closed embedding that preserves the eigenspace filtration,
i.e., it sends $\bF_k(V;m)$ into $\bF_k(W;m)$.
\end{construction}

The subquotients of the eigenspace filtration
can be described explicitly as Thom spaces over Grassmannians,
see Theorem \ref{thm:strata identification} below.
The key ideas originate from Frankel's paper \cite{frankel},
and the non-equivariant version of the formulation we need is due to Miller \cite[Theorem B]{miller}.
Given a finite-dimensional $\mK$-inner product space $W$,
Miller \cite[Theorem A]{miller} uses Morse theory to identify the stratum
\[ \bF_k(W;m)\setminus \bF_{k-1}(W;m)\]
with the total space of the vector bundle over the Grassmannian $Gr_k^\mK(W)$
associated to the representation $\Hom_\mK(\mK^k,\mK^m)\oplus\mathfrak{ad}(k)$
of $O(k)$, $U(k)$ or $S p(k)$.
We now recall a different identification,
based on explicit formulas due to Crabb \cite{crabb:U(n) and Omega},
see \eqref{eq:Millers bundle iso} below.

\begin{notation}
  We write $I(k)=\bL^\mK(\mK^k,\mK^k)$ for the isometry group
  of the standard $k$-dimensional $\mK$-inner product space.
  And we write
  \[  \mathfrak{ad}(k)\ = \ \{ X\in \End_\mK(\mK^k)\ : \  X^*= - X\}  \]
  for the $\mR$-vector spaces of skew-adjoint endomorphisms of $\mK^k$,
  where $X^*$ denotes the adjoint endomorphism, compare Remark \ref{rk:adjoint}.
  If we express linear endomorphism of $\mK^k$ as matrices, the group $I(k)$
  becomes the orthogonal group $O(k)$,
  the unitary group $U(k)$, or the symplectic group $S p(k)$, respectively,
  and the adjoint operator corresponds to the conjugate-transpose matrix.
  The isometry group $I(k)$ acts on $\mathfrak{ad}(k)$ by conjugation.
  The exponential map \eqref{eq:exponential} restricts to a map
  \[ \exp\ : \ \mathfrak{ad}(k) \ \to \ I(k) \]
  from the skew-adjoint endomorphisms to the isometry group of $\mK^k$.
  This map is a local diffeomorphism around the origin,
  and it exhibits $\mathfrak{ad}(k)$
  as the adjoint representation of the compact Lie group $I(k)$,
  whence the notation.
\end{notation}

\begin{construction}
  Crabb states in \cite[Lemma 1.13]{crabb:U(n) and Omega} that for all $k,m\geq 0$, the map
  \begin{equation}\label{eq:Cayley}
    \mathfrak c\ : \ \Hom_{\mK}(\mK^k,\mK^m)\oplus \mathfrak{ad}(k) \ \to \ \bL^\mK(\mK^k,\mK^{k+m})  \ ,\quad
    \mathfrak c(Y,X)(w)  = \ (g(w),h(w))     
  \end{equation}
  with
  \begin{alignat*}{3}
    g \ &= \ (X/2 + Y^*Y/4-1)(X/2 + Y^*Y/4 +1)^{-1} &&: \ \mK^k \ \to \ \mK^k\\
    h \ &= \ Y (X/2+ Y^*Y/4+1)^{-1} &&: \ \mK^k\ \to \ \mK^m
  \end{alignat*}
  is an open embedding onto the complement of
  $\bF_{k-1}(\mK^k;m)$ in $\bL^\mK(\mK^k,\mK^{k+m})$.
  Crabb works in the complex setting and refrains from giving a proof,
  so we provide an argument in Proposition \ref{prop:Cayley invariant}.
  The open embedding $\mathfrak c$ is clearly equivariant for the action of the isometry group $I(k)$,
  acting tautologically on all instances of $\mK^k$ and trivially on all instances of $\mK^m$.
  More precisely, the relation
  \[ \mathfrak c(Y A^*, A X A^* )\ = \ (A\oplus\mK^m)\cdot \mathfrak c(Y,X)\cdot  A^* \]
  holds for all $A\in I(k)$.
  For $m=0$, the map  specializes to the Cayley transform 
  $\mathfrak c:\mathfrak{ad}(k) \to I(k)$,
  see also \eqref{eq:o(k)2O(k)} above.

  A generalization of the open embedding \eqref{eq:Cayley} is the map
  \begin{equation}\label{eq:Millers bundle iso}
    \bL^\mK(\mK^k,W)\times_{I(k)} ( \Hom_\mK(\mK^k,\mK^m)\oplus\mathfrak{ad}(k))\
    \to \ \bF_k (W;m) \ ,\quad  [\psi,Y,X]\ \longmapsto \ {^\psi (\mathfrak c(Y,X))}\ .
  \end{equation}
  This map is an open embedding onto the complement of $\bF_{k-1}(W;m)$,
  i.e., we recover Miller's \cite[Theorem A]{miller}.
  For the convenience of the reader,
  we also recall a proof of this fact in Proposition \ref{prop:Stiefel relative homeo}.

  Now and in the following, we shall abbreviate
  \[ \nu(k,m)\ = \ \Hom_\mK(\mK^k,\mK^m)\ . \]
  A reader so inclined may identify this with the space of
  $\mK$-valued matrices of size $k\times m$, or with the $m$-fold sum of
  the tautological representation of $I(k)$ on $\mK^k$, via the isomorphism
  \[ \Hom_\mK(\mK^k,\mK^m)\ \iso \ \mK^k\oplus\dots\oplus \mK^k\ , \quad
    Y\ \longmapsto \ (Y^*(e_1),\dots Y^*(e_m)) \]
  that evaluates the adjoint at the coordinate basis.
  The collapse map associated to the embedding \eqref{eq:Millers bundle iso}
  then becomes a continuous map
  \[ \Psi(W)\ : \ \bF_k(W;m)\ \to \
    \bL^\mK(\mK^k,W)_+\sm_{I(k)} S^{\nu(k,m)\oplus \mathfrak{ad}(k) }  \ .\]
\end{construction}

In the rest of this section we organize the Stiefel manifolds, their eigenspace
filtrations and the identifications of the strata and subquotients
into unstable global information.
We continue to use the orthogonal space model of \cite[Section 1]{schwede:global}
to represent unstable global homotopy types.
Since we want to keep track of the actions of the Galois group $G(\mC)$ and $G(\mH)$,
we actually need an extension
to the context of `$C$-global homotopy theory', i.e., a global homotopy theory
that incorporates an additional action of a topological group $C$.
We develop the necessary theory in Appendix \ref{sec:C-global},
modeled by orthogonal $C$-spaces,
i.e., continuous functors from the linear isometries category $\bL$ to
the category of $C$-spaces.

\begin{eg}[The orthogonal spaces $\bO/m$, $\bU/m$ and $\bSp/m$]\label{eg:L/m}
  We let $\mK$ be one of the skew fields $\mR$, $\mC$ or $\mH$.
  For $m\geq 0$, we define an orthogonal $G(\mK)$-space $\bL/m$ as follows.
  The value of $\bL/m$ at a euclidean inner product space $V$ is
  \[ (\bL/m)(V)\ = \   \bL^\mK(V_\mK,V_\mK\oplus\mK^m)\ , \]
  the Stiefel manifold of $\mK$-linear isometric embeddings of $V_\mK$ into $V_\mK\oplus\mK^m$.
  Here $V_\mK=V\tensor_\mR\mK$ is the scalar extension from $\mR$ to $\mK$,
  endowed with the $\mK$-inner product 
  \[  [x\tensor\lambda,y\tensor\mu]\ = \ \bar\lambda \cdot \td{x,y}\cdot \mu       \]
  for $x,y\in V$ and $\lambda,\mu\in\mK$.
  The space $(\bL/m)(V)$ is homeomorphic to the homogeneous space
  $O(k+m)/O(m)$,  $U(k+m)/U(m)$ or $Sp(k+m)/Sp(m)$, where $k=\dim_\mR(V)$.
  An $\mR$-linear isometric embedding $\varphi:V\to W$ is sent to the continuous map
  \[ (\bL/m)(\varphi)\ = \ {^{\varphi_\mK}(-)}\ : \
    \bL^\mK(V_\mK,V_\mK\oplus\mK^m)\ \to\ \bL^\mK(W_\mK,W_\mK\oplus\mK^m) \]
  given by conjugation by $\varphi_\mK:V_\mK\to W_\mK$ as defined in \eqref{eq:conjugate by varphi}.
  The Galois group $G(\mK)=\Aut_\mR(\mK)$ acts on $(\bL/m)(V)$ as a similar `conjugation':
  for $\tau\in G(\mK)$ we define
  \begin{align*}
    {^\tau}(-) \ : \ &\bL^\mK(V_\mK,V_\mK\oplus\mK^m)\ \to \
                       \bL^\mK(V_\mK,V_\mK\oplus\mK^m)\text{\quad by}\\ 
    {^\tau f}\ &= \ ((V\tensor\tau)\oplus\tau^m)\circ f \circ(V\tensor\tau)^{-1}\ .
  \end{align*}
  As $\tau$ varies, these maps define a continuous $G(\mK)$-action on $(\bL/m)(V)$.
  As $V$ varies, these actions assemble into a continuous $G(\mK)$-action
  on the orthogonal space $\bL/m$.
  So we can -- and will -- view $\bL/m$ as an orthogonal $G(\mK)$-space.
  Whenever the skew-field $\mK$ is specific, we replace the generic notation $\bL/m$
  by the specific notation $\bO/m$ (for $\mK=\mR$),
  $\bU/m$ (for $\mK=\mC$) or $\bSp/m$ (for $\mK=\mH$).
    
  The eigenspace filtration \eqref{eq:eigenspace filtration} provides an exhausting filtration
  \[ \ast \ = \ F_0(\bL/m)\ \subseteq \ F_1(\bL/m)\ \subseteq \ F_2(\bL/m)\ \subseteq \
    \dots\ \subseteq \ F_k(\bL/m)\ \subseteq \ \dots
  \]
  of $\bL/m$ by $G(\mK)$-invariant orthogonal subspaces with terms
  \[   F_k(\bL/m)(V)\ = \ \bF_k(V_\mK;m)\ = \
    \{ f\in \bL^\mK(V_\mK,V_\mK\oplus\mK^m)\ : \   \dim_\mK(\ker(f - i_1)^\perp)\leq  k\}\ . \]
\end{eg}

\begin{eg}[The orthogonal spaces $\bO$, $\bSO$, $\bU$, $\bSU$ and $\bSp$]
  The special cases $m=0$ and $m=1$ of Example \ref{eg:L/m}
  provide particularly interesting global refinements of the infinite orthogonal,
  special orthogonal, unitary, special unitary and symplectic groups.  
  Indeed, $(\bL/0)(V)=\bL^\mK(V_\mK,V_\mK)$
  is the isometry group of $V_\mK$; in this special case, $\bL/0$ also supports
  an ultra-commutative multiplication in the sense of \cite[Chapter 2]{schwede:global}
  by orthogonal direct sum of isometric embeddings.
  In the real case, $\bL/0$ is the ultra-commutative monoid $\bO$ of orthogonal groups
  discussed in detail in \cite[Example 2.3.6]{schwede:global}.
  In the complex case, we obtain the ultra-commutative monoid $\bU$ of unitary groups
  discussed in \cite[Example 2.3.7]{schwede:global},
  with additional action by the Galois group of $\mC$ over $\mR$.
  For $\mK=\mH$, the construction becomes the ultra-commutative monoid $\bSp$
  made from symplectic groups, compare \cite[Example 2.3.9]{schwede:global},
  with an additional action by the compact Lie group $G(\mH)$.

  In the real and complex cases, the special case $m=1$ also deserves explicit mentioning.
  Indeed, in the real case, $\bL/1$ is globally equivalent to the
  orthogonal space $\bSO$ of special orthogonal groups, the 
  ultra-commutative submonoid of $\bO$ consisting of the special orthogonal groups,
  compare Proposition \ref{prop:SO_versus_O/1}.
  Similarly, in the complex case, $\bL/1$ is $G(\mC)$-globally equivalent to the
  orthogonal space $\bSU$ of special unitary groups, see again Proposition \ref{prop:SO_versus_O/1}.
  We refer the reader to the extensive discussion in \cite[Chapter 2]{schwede:global}
  for further information about equivariant and
  global homotopical properties of $\bO$, $\bSO$, $\bU$, $\bSU$ and $\bSp$.
\end{eg}

The main result of this paper, Theorem \ref{thm:stably split L/m} below,
says that the filtration of $\bL/m$ by the orthogonal $G(\mK)$-subspaces $F_k(\bL/m)$
splits once we pass to the $G(\mK)$-global stable homotopy category.
As we record now, the individual identifications of the subquotients of the eigenspace
filtration are natural enough to assemble into isomorphisms of
$G(\mK)$-spaces.

\begin{eg}[Global Thom spaces]\label{eg:global Thom}
  We organize the various Grassmannians 
  and the relevant Thom spaces over them into orthogonal $G(\mK)$-spaces.
  We write $\bGr_k^\mK$ for the orthogonal space whose value
  at a euclidean inner product space $V$ is 
  the Grassmannian of $k$-dimensional $\mK$-subspaces in $V_\mK=V\tensor_\mR\mK$.
  The structure map $\bGr_k^\mK(\varphi)$
  associated to an $\mR$-linear isometric embedding $\varphi:V\to W$
  takes the image under the scalar extension $\varphi_\mK:V_\mK\to W_\mK$.
  A Galois automorphism $\tau\in G(\mK)$ acts on $\bGr_k^\mK(V)$ by
  taking images under $V\tensor\tau:V_\mK\to V_\mK$;
  this operation does take $\mK$-subspaces to
  $\mK$-subspaces, despite the fact that $V\tensor\tau$ need not be $\mK$-linear.
  Altogether, it makes $\bGr_k^\mK$ into an orthogonal $G(\mK)$-space.

  The Galois group acts continuously on the isometry group $I(k)=\bL^\mK(\mK^k,\mK^k)$
  by conjugation, i.e., by
  \[ {^\tau f}\ = \ \tau^k \circ f\circ (\tau^k)^{-1}  \]
  for $\tau\in G(\mK)$ and $f\in I(k)$.
  The {\em extended isometry group} is 
  \[ \tilde I(k)\ = \ I(k)\rtimes G(\mK)\ , \]
  the semidirect product for this action.
  By Theorem \ref{thm:zeta_is_A-global}, $\bGr_k^\mK$ is a global classifying $G(\mK)$-space
  for the augmentation $\epsilon_k:\tilde I(k)\to G(\mK)$ of the extended isometry group.
  In particular, the orthogonal space of $\bGr_k^\mR$ is a global classifying space
  for the orthogonal group $O(k)$, and 
  the underlying orthogonal spaces of $\bGr_k^\mC$ and $\bGr_k^\mH$ are
  global classifying spaces for the unitary group $U(k)$ and symplectic group $S p(k)$, respectively.
  
  Now we let $U$ be a real representation of the extended isometry group $\tilde I(k)$.
  The Stiefel manifold $\bL^\mK(\mK^k,V_\mK)$ comes with a continuous
  $\tilde I(k)$-action by 
  \[ {^{(A,\tau)}  \varphi} \ = \ (V\tensor\tau)\circ \varphi\circ (A\circ \tau^k)^{-1} \ .\]
  The normal subgroup $I(k)$ of $\tilde I(k)$ acts freely,
  and the orthogonal orbit space $\bL^\mK(\mK^k,V_\mK)/I(k)$
  is isomorphic to the Grassmannian $\bGr_k^\mK(V)$, by sending a linear isometric embedding to its image.
  We endow the space $\bL^\mK(\mK^k,V_\mK)\times_{I(k)} U$
  with a $G(\mK)$-action by
  \[ {^\tau [\varphi,u]}\ = \ [^{(1,\tau)}\varphi,(1,\tau) u] \ ;\]
  it is the total space of a $G(\mK)$-vector bundle over $\bGr_k^\mK(V)$, with Thom $G(\mK)$-space
 $\bL^\mK(\mK^k,V_\mK)_+\sm_{I(k)} S^U$.
  We write
  \[ (\bGr_k^\mK)^U\ = \  \bL^\mK(\mK^k,(-)_\mK)_+\sm_{I(k)} S^U \ ,      \]
  and refer to this as the {\em $G(\mK)$-global Thom space} associated to $U$.
\end{eg}

If we specialize the collapse map associated to the open embedding \eqref{eq:Millers bundle iso}
to inner product spaces of the form $W=V_\mK$,
for euclidean inner product spaces $V$, it becomes a continuous map
\[ 
  \Psi(V_\mK)\ : \ F_k(\bL/m)(V)\ = \ \bF_k(V_\mK;m)\ \to \
  \bL^\mK(\mK^k,V_\mK)_+\sm_{I(k)} S^{\nu(k,m)\oplus \mathfrak{ad}(k)}  \ = \
  (\bGr_k^\mK)^{\nu(k,m)\oplus \mathfrak{ad}(k) }(V)\ .     
\]
We have thus proved the following theorem,
which summarizes work of Miller \cite[Theorem B]{miller}
and Crabb \cite[Lemma 1.14]{crabb:U(n) and Omega} in a coordinate-free
and Galois-enhanced fashion.

\begin{theorem}\label{thm:strata identification}
  For varying euclidean inner product spaces $V$,
  the maps $\Psi(V_\mK)$ constitute a morphism of orthogonal $G(\mK)$-spaces
  \[  \Psi\ : \ F_k(\bL/m)\ \to \ (\bGr_k^\mK)^{\nu(k,m)\oplus \mathfrak{ad}(k) }      \]
  that factors through an isomorphism
    \[  F_k(\bL/m) / F_{k-1}(\bL/m)\ \iso \ (\bGr_k^\mK)^{\nu(k,m)\oplus \mathfrak{ad}(k) }   \ .     \]
\end{theorem}  

We will also need to know that the pointset level eigenspace filtration of the orthogonal space
$\bL/m$ is homotopically meaningful. This is guaranteed by the following proposition,
showing that the pairs of successive filtration stages have the homotopy extension
property internal to orthogonal $G(\mK)$-spaces.

\begin{prop}\label{prop:skeleton h-cofibration}
  For all $k,m\geq 0,$ the inclusion $F_{k-1}(\bL/m)\to F_k(\bL/m)$
  has the homotopy extension property in the category of orthogonal $G(\mK)$-spaces.
\end{prop}
\begin{proof}
  The key ingredient is the pushout square of $\tilde I(k)$-spaces
  \[\xymatrix{
      S(\nu(k,m)\oplus\mathfrak{a d}(k))\ar[r]\ar[d] &
      D(\nu(k,m)\oplus\mathfrak{a d}(k))\ar[d] \\
      \bF_{k-1}(\mK^k;m)\ar[r] & \bL^\mK(\mK^k,\mK^{k+m}) }\]
  proved by Miller in \cite[Lemma 2.4]{miller},
  using a Morse theoretic construction that goes back to Frankel \cite{frankel}.
  The upper row is the inclusion of the unit sphere into the unit disc of the representation
  $\nu(k,m)\oplus\mathfrak{a d}(k)$ of the extended isometry group $\tilde I(k)$,
  with respect to some $\tilde I(k)$-invariant inner product.
  There is one caveat, namely that Miller states the equivariance
  only for the group $I(k)$ (for which he writes $G_k$), whereas
  we need equivariance for the larger group $\tilde I(k)$.
  But this is easy to fix: as explained in \cite[\S 4]{miller},
  the vertical maps in the pushout square are derived from an
  explicit Morse-Bott function and an explicit Riemannian metric on $\bL^\mK(\mK^k,\mK^{k+m})$
  specified on page 417 of \cite{miller};
  both the Morse-Bott function and the metric are invariant
  even for the extended isometry group $\tilde I(k)$.
  Hence also the gradient vector field, the associated flow and the vertical attaching map
  in the pushout square derived from this data are equivariant
  for the larger group $\tilde I(k)$.

  The inclusion of unit sphere into unit disc of any orthogonal representation of a compact Lie
  group has the equivariant homotopy extension property. Since the homotopy extension property
  is also stable under cobase change, we conclude that
  the inclusion $\bF_{k-1}(\mK^k;m)\to \bL^\mK(\mK^k,\mK^{k+m})$
  has the $\tilde I(k)$-equivariant homotopy extension property.  

  For every $\mK$-inner product space $W$, the map
  \[  \bL^\mK(\mK^k,W)\times_{I(k)} \bL^\mK(\mK^k,\mK^{k+m})  \ \to \ \bF_k(W;m) \ ,\quad
    [\psi,f]\ \longmapsto \ {^\psi f} \]
  is a relative homeomorphism from the pair
  \[ \left(
      \bL^\mK(\mK^k,W)\times_{I(k)} \bL^\mK(\mK^k,\mK^{k+m}),\
      \bL^\mK(\mK^k,W)\times_{I(k)}  \bF_{k-1}(\mK^k;m)  \right)\]
  to the pair $(\bF_k(W;m),\bF_{k-1}(W;m))$,
  compare Proposition \ref{prop:Stiefel relative homeo}.
  Taking $W=V_\mK$ and letting $V$ vary over all euclidean inner product spaces
  thus yields a pushout square of orthogonal $G(\mK)$-spaces
  \[\xymatrix{
      \bL^\mK(\mK^k,(-)_\mK)\times_{I(k)} \bF_{k-1}(\mK^k;m) \ar[r]\ar[d] &
      \bL^\mK(\mK^k,(-)_\mK)\times_{I(k)} \bL^\mK(\mK^k,\mK^{k+m})
      \ar[d] \\
      F_{k-1}(\bL/m)\ar[r] & F_k(\bL/m) }\]
  Since  the inclusion $\bF_{k-1}(\mK^k;m)\to \bL^\mK(\mK^k,\mK^{k+m})$
  has the $\tilde I(k)$-equivariant homotopy extension property,
  the upper horizontal morphism has the homotopy extension property
  internal to orthogonal $G(\mK)$-spaces.
  The homotopy extension property is stable under cobase change, so this proves the claim.  
\end{proof}

\section{The stable splitting}
\label{sec:splitting}

In this section we construct the morphisms in the $G(\mK)$-global stable homotopy category
that stably split the eigenspace filtration of $\bL/m$,
and we prove our main result Theorem \ref{thm:stably split L/m},
the  $G(\mK)$-global splitting of the unreduced suspension spectrum of $\bL/m$.
The construction is based on a particularly natural way to split off 
the top cell of the Stiefel manifold $\bL^\mK(\mK^k,\mK^{k+m})$
that was exhibited by Crabb \cite[Theorem 1.16]{crabb:U(n) and Omega}.
Compared to the statement in Miller's original paper on the subject \cite{miller},
the main refinement is the additional equivariance with respect to the extended
isometry groups;
this extra equivariance is crucial for our arguments, as it allows us to leverage
the information to the $G(\mK)$-global context.

Crabb actually states the analogous splitting in the unitary situation,
i.e., for the complex Stiefel manifold $\bL^\mC(\mC^k,\mC^{k+m})$
(which he denotes $U(\mC^k;\mC^m)$). 
For the convenience of the reader, we recall the construction in a form
that works simultaneously over $\mR$, $\mC$ and $\mH$.
The upshot is the map $t_{k,m}$ in \eqref{collapse_symplectic} below;
needless to say that I claim no originality for the construction of this map.

\begin{construction}
  We write 
  \[   \mathfrak{s a}(k)\ =\ \{ Z\in \End_\mK(\mK^k)\ : \ Z^* = Z \}  \]
  for the $\mR$-vector spaces of self-adjoint
  endomorphisms of $\mK^k$.
  A basic linear algebra fact is that for all $k,m\geq 0$, the smooth map
  \begin{equation}\label{eq:split_embedding}
    \bL^\mK(\mK^k,\mK^{k+m})\times \mathfrak{s a}(k) \ \to \ \Hom_\mK(\mK^k,\mK^{k+m})\ , \quad
    (A,Z)\ \longmapsto \ A\circ\exp(-Z)   
  \end{equation}
  is an open embedding with image the subspace of $\mK$-linear monomorphisms.
  For the convenience of the reader, we give a proof in Proposition \ref{prop: inverse emb mod m over H}.
  The extended isometry group $\tilde I(k)$ acts on $\Hom_\mK(\mK^k,\mK^{k+m})$
  by conjugation, i.e., by
  \[ {^{(A,\tau)}  X} \ = \ ( (A\circ\tau^k)\oplus\tau^m)\cdot X \cdot (A\circ\tau^k)^{-1} \]
  for $(A,\tau)\in\tilde I(k)=I(k)\rtimes G(\mK)$ and $X:\mK^k\to\mK^{k+m}$.
  This conjugation action leaves the subspace $\bL^\mK(\mK^k,\mK^{k+m})$ invariant;
  and in the case $m=0$, the action leaves the subspace $\mathfrak{s a}(k)$ invariant.
  The open embedding \eqref{eq:split_embedding}
  is $\tilde I(k)$-equivariant, so it provides an $\tilde I(k)$-equivariant collapse map
  \[ S^{\Hom_\mK(\mK^k,\mK^{k+m})} \ \to \ \bL^\mK(\mK^k,\mK^{k+m})_+\sm S^{\mathfrak{s a}(k)}\ . \]
   We recall that $\nu(k,m)=\Hom_\mK(\mK^k,\mK^m)$; we use the equivariant direct sum decomposition
  \[   \nu(k,m)\oplus \mathfrak{ad}(k)\oplus\mathfrak{s a}(k) \ \xra{\ \iso \ }\
    \Hom_\mK(\mK^k,\mK^k\oplus \mK^m) \ , \quad (f,X,Y)\ \longmapsto \ (X+Y,f)\ ,\]
 to interpret the previous collapse map as a continuous $\tilde I(k)$-equivariant map
 \begin{equation}\label{collapse_symplectic}
   t_{k,m}\ : \ S^{\nu(k,m)\oplus\mathfrak{ad}(k)}\sm S^{\mathfrak{s a}(k)}\ \to \
    \bL^\mK(\mK^k,\mK^{k+m})_+\sm S^{\mathfrak{s a}(k)}\ .
  \end{equation}
  As is implicit in \cite{crabb:U(n) and Omega}, and as we recall in the proof of
  Theorem \ref{thm:s_k,m splits L/m}, the map $t_{k,m}$ represents a section,
  in the $\tilde I(k)$-equivariant stable homotopy category,
  of the collapse map $\bL^\mK(\mK^k,\mK^{k+m})\to S^{\nu(k,m)\oplus\mathfrak{ad}(k)}$
  associated to the open embedding \eqref{eq:Cayley}.
\end{construction}

Now comes the point where we leverage $\tilde I(k)$-equivariant information
into $G(\mK)$-global information.
We model $G(\mK)$-global stable homotopy theory by orthogonal $G(\mK)$-spectra, i.e.,
orthogonal spectra equipped with a continuous action of $G(\mK)$.
We refer to Appendix \ref{sec:C-global} for a detailed discussion of this model,
the notion of  $G(\mK)$-global equivalence
(see Definition \ref{def:stable C-global equivalence}),
and the $G(\mK)$-global stable homotopy category \eqref{eq:GH_C}. 
The next construction turns the collapse map $t_{k,m}$ into a morphism
$s_{k,m} :\Sigma^\infty (\bGr_k^\mK)^{\nu(k,m)\oplus\mathfrak{ad}(k)}\to \Sigma^\infty_+ F_k(\bL/m)$
in the $G(\mK)$-global stable homotopy category.

\begin{construction}\label{con:splitting morphism}
  The unreduced suspension spectrum of the orthogonal $G(\mK)$-space $F_k(\bL/m)$
  is an orthogonal $G(\mK)$-spectrum $\Sigma^\infty_+ F_k(\bL/m)$.
  Restricting actions along the augmentation
  $\epsilon_k:\tilde I(k)\to G(\mK)$ of the extended isometry group
  provides an orthogonal $\tilde I(k)$-spectrum $\epsilon_k^*(\Sigma^\infty_+ F_k(\bL/m))$.
  We will now promote the collapse map \eqref{collapse_symplectic} to 
  an equivariant  homotopy class
  \[  \td{t_{k,m}} \ \in \
    \pi_{\nu(k,m)\oplus\mathfrak{ad}(k)}^{\tilde I(k)}\left( \epsilon_k^*( \Sigma^\infty_+ F_k(\bL/m))\right) \ , \]
  graded by the $\tilde I(k)$-representation $\nu(k,m)\oplus\mathfrak{ad}(k)$;
  see \eqref{eq:repgraded_homotopy_group} for equivariant homotopy groups
  graded by a representation.

  We write $\nu_k$ for the tautological orthogonal representation
  of $\tilde I(k)$ on $\mK^k$, i.e., the underlying $\mR$-vector space $u\mK^k$
  endowed with the euclidean inner product $\td{x,y}=\Re[x,y]$, the real part of
  the $\mK$-inner product.
  The $\mK$-linear isometric embedding
  \[   
    \zeta \ : \ \mK^k \ \to \ (u \mK^k)\tensor_\mR\mK \ = \  (\nu_k)_\mK 
  \]
  is defined in \eqref{eq:define zeta}. We write
  $i_2 :  (\nu_k)_\mK  \to ( \mathfrak{s a}(k)\oplus \nu_k)_\mK$
  for the embedding as the second summand. Conjugation by $i_2\zeta$
  as defined in \eqref{eq:conjugate by varphi} is then
  a continuous $\tilde I(k)$-equivariant map
  \[ {^{i_2 \zeta}(-)}\ : \ \bL^\mK(\mK^k,\mK^{k+m})= \bF_k(\mK^k;m)\ \to \
    \bF_k((\mathfrak{s a}(k)\oplus\nu_k)_\mK;m) = 
    F_k(\bL/m)(\mathfrak{s a}(k)\oplus \nu_k)\ .
  \]
We define $\td{t_{k,m}}$ as the homotopy class of the following composite
\begin{align*}
  S^{\nu(k,m)\oplus\mathfrak{ad}(k)\oplus \mathfrak{s a}(k)\oplus\nu_k} \
  \xra{t_{k,m}\sm S^{\nu_k}} \  &\bL^\mK(\mK^k,\mK^{k+m})_+\sm S^{\mathfrak{s a}(k)\oplus\nu_k}\\
  \xra{{^{i_2\zeta}(-)}_+\sm S^{\mathfrak{s a}(k)\oplus \nu_k}} \ 
      &F_k(\bL/m)(\mathfrak{s a}(k)\oplus\nu_k)_+\sm  S^{\mathfrak{s a}(k)\oplus\nu_k} 
          = (\Sigma^\infty F_k(\bL/m))(\mathfrak{s a}(k)\oplus\nu_k) \ .
\end{align*}
\end{construction}

The tautological homotopy class $e_{k,\nu(k,m)\oplus\mathfrak{ad}(k)}$ is defined in \eqref{eq:define_e_k,U}.
The representability result in Corollary \ref{cor:J_is_A-global} provides a unique morphism
\begin{equation}\label{eq:define s_k,m}
  s_{k,m}\ : \ \Sigma^\infty (\bGr_k^\mK)^{\nu(k,m)\oplus\mathfrak{ad}(k)}\ \to \ \Sigma^\infty_+ F_k(\bL/m)
\end{equation}
in the $G(\mK)$-global stable homotopy category characterized by the equation
\begin{equation}\label{eq:sJe_sigmat}
   (s_{k,m})_*(e_{k,\nu(k,m)\oplus\mathfrak{ad}(k)})\ = \ \td{t_{k,m}} \ .
\end{equation}
The collapse morphism of orthogonal $G(\mK)$-spaces
\[  \Psi\ : \ F_k(\bL/m)\ \to \ (\bGr_k^\mK)^{\nu(k,m)\oplus \mathfrak{ad}(k) }      \]
was defined in  Theorem \ref{thm:strata identification}.

\begin{theorem}\label{thm:s_k,m splits L/m}
  For all $k,m\geq 0$, the relation
  \[ (\Sigma^\infty\Psi)_*\td{t_{k,m}}  \ =\ e_{k,\nu(k,m)\oplus\mathfrak{ad}(k)}  \]  
  holds in the group $\pi^{\tilde I(k)}_{\nu(k,m)\oplus\mathfrak{ad}(k)}(\epsilon_k^*(\Sigma^\infty (\bGr_k^\mK)^{\nu(k,m)\oplus\mathfrak{ad}(k)}))$. 
  Hence the composite
  \[  \Sigma^\infty (\bGr_k^\mK)^{\nu(k,m)\oplus\mathfrak{ad}(k)}\
    \xra{s_{k,m}} \ \Sigma^\infty_+ F_k(\bL/m)\
    \xra{\Sigma^\infty \Psi} \
    \Sigma^\infty (\bGr_k^\mK)^{\nu(k,m)\oplus\mathfrak{ad}(k)}\]
  is the identity in the $G(\mK)$-global stable homotopy category.
\end{theorem}
\begin{proof}
  The theorem hinges on the fact that the collapse map $t_{k,m}$ defined in \eqref{collapse_symplectic}
  indeed splits off the top cell of the Stiefel manifold $\bL^\mK(\mK^k,\mK^{k+m})$
  in the genuine $\tilde I(k)$-equivariant stable homotopy category.
  More precisely, the composite
\[ 
    S^{\nu(k,m)\oplus \mathfrak{ad}(k)}\sm S^{\mathfrak{s a}(k)}\  \xra{\ t_{k,m}\ } \
    \bL^\mK(\mK^k,\mK^{k+m})_+\sm S^{\mathfrak{s a}(k)}\ \xra{\mathfrak c^\flat\sm S^{\mathfrak{s a}(k)}} \
   S^{\nu(k,m)\oplus\mathfrak{ad}(k)}\sm S^{\mathfrak{s a}(k)}     
  \]
  is $\tilde I(k)$-equivariantly based homotopic to the identity,
  where $\mathfrak c^\flat :\bL^\mK(\mK^k,\mK^{k+m})\to S^{\nu(k,m)\oplus\mathfrak{ad}(k)}$
  is the collapse map based on the open embedding \eqref{eq:Cayley}.
  The argument given by Crabb \cite{crabb:U(n) and Omega} in the unitary situation 
  works just as well over $\mR$ and over $\mH$, as follows.
  Collapsing along open embeddings is transitive, so the  composite
  $(\mathfrak c^\flat\sm S^{\mathfrak{s a}(k)})\circ t_{k,m}$ is the collapse map of the composite open embedding:
  \begin{align*}
    F \ : \      \Hom_\mK(\mK^k,\mK^{k+m}) =  \nu(k,m)\times\mathfrak{ad}(k)&\times
\mathfrak{s a}(k)\
    \to\ \Hom_\mK(\mK^k,\mK^{k+m})\\
    F(Y,X,Z)\quad &= \quad \mathfrak c(Y,X)\cdot \exp(-Z)
  \end{align*}
  Inspection of the explicit formulas shows that the open embedding $F$ is smooth
  and its differential at the origin is the identity.
  So the associated collapse map is 
  $\tilde I(k)$-equivariantly based homotopic to the identity.
  This is the desired equivariant homotopy
  from  $(\mathfrak c^\flat\sm S^{\mathfrak{s a}(k)})\circ t_{k,m}$ to the identity of $S^{\Hom(\mK^k,\mK^{k+m})}$.

  Now we can determine $(\Sigma^\infty\Psi)_*\td{t_{k,m}}$.
  Expanding definitions shows that $(\Sigma^\infty\Psi)_*\td{t_{k,m}}$
  is represented by the following composite, 
  smashed with the identity of $S^{\nu_k}$:
  \begin{align*}
    S^{\nu(k,m)\oplus\mathfrak{ad}(k)\oplus \mathfrak{s a}(k)} \
    \xra{t_{k,m}\sm S^{\mathfrak{s a}(k)}} \  &\bL^\mK(\mK^k,\mK^{k+m})_+\sm S^{\mathfrak{s a}(k)}\\
    \xra{c^\flat\sm S^{\mathfrak{s a}(k)}} \  & S^{\nu(k,m)\oplus\mathfrak{ad}(k)}\sm S^{\mathfrak{s a}(k)}\ =\
                                             \left(\bL^\mK(\mK^k,\mK^k)_+\sm_{I(k)}  S^{\nu(k,m)\oplus\mathfrak{ad}(k)}\right)\sm  S^{\mathfrak{s a}(k)}  \\
    \xra{\zeta_*\sm S^{\mathfrak{s a}(k)}} \ & \left(\bL^\mK(\mK^k,(\nu_k)_\mK)_+\sm_{I(k)}  S^{\nu(k,m)\oplus\mathfrak{ad}(k)}\right)\sm  S^{\mathfrak{s a}(k)}  \\
= \qquad   &(\bGr_k^\mK)^{\nu(k,m)\oplus\mathfrak{ad}(k)}(\nu_k)\sm S^{\mathfrak{s a}(k)} \\
                   \xra{(i_2)_*\sm S^{\mathfrak{s a}(k)}}\  &(\bGr_k^\mK)^{\nu(k,m)\oplus\mathfrak{ad}(k)}(\mathfrak{s a}(k)\oplus\nu_k)\sm S^{\mathfrak{s a}(k)} \ .
  \end{align*}
  Since the composite of the first two maps is equivariantly
  homotopic to the identity, we conclude that
  $(\Sigma^\infty\Psi)_* \td{t_{k,m}} = e_{k,\nu(k,m)\oplus\mathfrak{ad}(k)}$.
  Together with the defining property \eqref{eq:sJe_sigmat} of $s_{k,m}$, this shows that
  \[  ((\Sigma^\infty\Psi)\circ s_{k,m})_*(e_{k,\nu(k,m)\oplus\mathfrak{ad}(k)}) \
    = \     (\Sigma^\infty\Psi)_* \td{t_{k,m}} \ = \ 
    e_{k,\nu(k,m)\oplus\mathfrak{ad}(k)} \ . \]
  The representability property of Corollary \ref{cor:J_is_A-global}
  thus shows that the composite $(\Sigma^\infty\Psi)\circ s_{k,m}$ is the identity.
\end{proof}

We write
\[ s_{k,m}^\mK\ : \ \Sigma^\infty (\bGr_k^\mK)^{\nu(k,m)\oplus\mathfrak{ad}(k)}\ \to \ \Sigma^\infty_+ \bL/m   \]
for the composite of the morphism $s_{k,m}$ of \eqref{eq:define s_k,m}
and the morphism of suspension spectra induced by the inclusion $F_k(\bL/m)\to\bL/m$.
The following theorem is the main result of this paper.

\begin{theorem}\label{thm:stably split L/m}
  Let $\mK$ be one of the skew-fields $\mR$, $\mC$ or $\mH$.
  For every $m\geq 0$, the morphism
  \[ \sum s_{k,m}^\mK \ :\ \bigvee_{k\geq 0}  \Sigma^\infty (\bGr_k^\mK)^{\nu(k,m)\oplus \mathfrak{ad}(k)} \ \to \
    \Sigma^\infty_+ \bL/m\]
  is an isomorphism in the $G(\mK)$-global stable homotopy category.  
\end{theorem}
\begin{proof}
  In a first step we prove the analogous splitting result for $F_n(\bL/m)$ by induction on $n\geq 0$,
  namely that the morphism
  \begin{equation}\label{eq:sum02n}
 \sum s_{k,m} \ :\ \bigvee_{0\leq k\leq n}  \Sigma^\infty \bGr_k^{\nu(k,m)\oplus \mathfrak{ad}(k)} \ \to \
    \Sigma^\infty_+ F_n(\bL/m)    
  \end{equation}
  is an isomorphism in the $G(\mK)$-global stable homotopy category.  
  There is nothing to show for $n=0$, so we assume $n\geq 1$.
  The inclusion $F_{n-1}(\bL/m)\to F_n(\bL/m)$ has the homotopy extension property
  in the category of orthogonal $G(\mK)$-spaces
  by Proposition \ref{prop:skeleton h-cofibration}.
  So the induced morphism $\Sigma^\infty_+ F_{n-1}(\bL/m)\to \Sigma^\infty_+ F_n(\bL/m)$ is an h-cofibration
  of orthogonal $G(\mK)$-spectra.
  Hence there is a distinguished triangle
  \[
    \Sigma^\infty_+ F_{n-1}(\bL/m) \to 
    \Sigma^\infty_+ F_n(\bL/m)\xra{\Sigma^\infty \Psi}
    \Sigma^\infty (\bGr_n^\mK)^{\nu(n,m)\oplus \mathfrak{ad}(n)} \to  (\Sigma^\infty_+ F_{n-1}(\bL/m))\sm S^1    
  \]
  in the $G(\mK)$-global stable homotopy category;
  we have used Theorem \ref{thm:strata identification} to substitute
  $F_n(\bL/m)/F_{n-1}(\bL/m)$ by $(\bGr_n^\mK)^{\nu(n,m)\oplus \mathfrak{ad}(n)}$.
  Theorem \ref{thm:s_k,m splits L/m} splits the distinguished triangle, so the morphism
    \[
    \text{incl} + s_{n,m} \ : \ \Sigma^\infty_+  F_{n-1}(\bL/m)\ \vee\ \Sigma^\infty  (\bGr_n^\mK)^{\nu(n,m)\oplus\mathfrak{ad}(n)} \ \to \
    \Sigma^\infty_+  F_n(\bL/m)
  \]
  is an isomorphism in $\GH_{G(\mK)}$.
  Induction on $n$ then proves that the morphism \eqref{eq:sum02n}
  is an isomorphism in the $G(\mK)$-global stable homotopy category.

  Passage to the colimit over $n$ then yields the claim, as follows.
  For every continuous homomorphism $\alpha:G\to G(\mK)$, the $G$-equivariant homotopy groups
  of source and target of the morphism in question are the colimit over $n$ of the truncated versions
  \eqref{eq:sum02n}. So the morphism of the theorem induces isomorphisms of
  $G$-equivariant homotopy groups for all such homomorphisms $\alpha$, and is thus
  an isomorphism in the $G(\mK)$-global stable homotopy category.
\end{proof}

We take the time to make some particularly interesting special cases of Theorem \ref{thm:stably split L/m}
explicit.
For $m=0$, the orthogonal space $\bL/0$ specializes to the ultra-commutative monoids $\bO$,
$\bU$ and $\bSp$ of orthogonal, unitary and symplectic groups
that we discuss in detail in Examples 2.3.6, 2.3.7 and 2.3.9 of \cite{schwede:global}.
In the real case, we already presented the global stable splitting of $\bO$
in Theorem \ref{thm:stably split O}.
In the complex and quaternionic case, Theorem \ref{thm:stably split L/m} specializes to:

\begin{theorem}\label{thm:stably split U and Sp}
The morphism
\[ \sum s^\mC_{k,0}\ : \ 
  \bigvee_{k\geq 0} \Sigma^\infty (\bGr_k^\mC)^{\mathfrak{ad}(k)} \ \to \ \Sigma^\infty_+ \bU \] 
is an isomorphism in the $G(\mC)$-global stable homotopy category.
The morphism
  \[ \sum s^\mH_{k,0}\ : \ 
  \bigvee_{k\geq 0} \Sigma^\infty (\bGr_k^\mH)^{\mathfrak{ad}(k)} \ \to \ \Sigma^\infty_+ \bSp\] 
is an isomorphism in the $G(\mH)$-global stable homotopy category.
\end{theorem}

In the real and complex situation, the special case $m=1$
of Theorem \ref{thm:stably split L/m} is also worth making explicit
because $\bO/1$ and $\bU/1$ can be replaced by other global spaces
made from special orthogonal and special unitary groups.
The closed embeddings
\[ i_1\circ - \ : \ O(V) \ \to \ \bL(V,V\oplus \mR)\text{\qquad and\qquad}
i_1\circ - \ : \  U(V_\mC) \ \to \ \bL^\mC(V_\mC,V_\mC\oplus\mC)\]
form morphisms of orthogonal spaces and orthogonal $G(\mC)$-spaces
\[ \iota\ : \ \bO \ \to\ \bO/1 \text{\qquad and\qquad} \iota\ : \ \bU \ \to\ \bU/1\ , \]
respectively.
We denote by $\bSO$ and $\bSU$ the orthogonal subspaces of $\bO$ and $\bU$ with respective values 
\[ \bSO(V)\ =\ S O(V) \text{\qquad and\qquad} \bSU\ = \ S U(V_\mC)\ ,\]
the isometries of determinant 1.

\begin{prop}\label{prop:SO_versus_O/1}
  The morphisms 
  \[  \iota|_{\bSO}\ :\ \bSO\ \to\ \bO/1\text{\qquad and\qquad}
    \iota|_{\bSU}\ :\ \bSU\ \to\ \bU/1 \]
  are a global equivalence of orthogonal spaces,
  and a $G(\mC)$-global equivalence of orthogonal $G(\mC)$-spaces, respectively. 
\end{prop}
\begin{proof}
  By \cite[Theorem 1.1.10]{schwede:global}, the closed embeddings
  \[ S O(V)\ \to \ S O(V\oplus\mR) \ , \quad A \ \longmapsto \ A\oplus\mR\]
  form a global equivalence $\bSO \to \sh(\bSO)$,
  where $\sh$ denotes the additive shift of an orthogonal space
  in the sense of \cite[Example 1.1.11]{schwede:global}.
  The restriction homeomorphisms
  \[ -\circ i_1\ : \ S O(V\oplus\mR) \ \xra{\ \iso \ } \ \bL(V,V\oplus\mR)\]
  form an isomorphism of orthogonal spaces $\sh(\bSO)\iso \bO/1$.
  The restriction of $\iota:\bO\to\bO/1$ to $\bSO$ coincides with
  the composite of the former global equivalence and
  the latter isomorphism; so $\iota|_{\bSO}$ is a global equivalence, too.
  The argument in the complex case is analogous, with one caveat: one must convince oneself that
  \cite[Theorem 1.1.10]{schwede:global} generalizes to the $G(\mC)$-global context,
  with the same proof.
\end{proof}

Because the morphism $\iota|_{\bSO}:\bSO\to\bO/1$ is a global equivalence of orthogonal spaces,
the morphism
\[ \Sigma^\infty_+ \iota|_{\bSO}\ : \ \Sigma^\infty_+ \bSO \ \to\ \Sigma^\infty_+ \bO/1\]
is a global equivalence of orthogonal spectra, see \cite[Corollary 4.1.9]{schwede:global}.
Similarly, Proposition \ref{prop:suspension is homotopical} shows
that the morphism of orthogonal $G(\mC)$-spectra
$\Sigma^\infty_+\iota|_{\bSU}: \Sigma^\infty_+\bSU\to \Sigma^\infty_+\bU/1$
is a $G(\mC)$-global equivalence.
So the special case $m=1$ of Theorem \ref{thm:stably split L/m} yields the following result.

\begin{theorem}\label{thm:stably split SO}
  The morphism
  \[  
    \sum (\Sigma^\infty_+\iota|_{\bSO})^{-1}\circ s^\mR_{k,1}\ : \ 
    \bigvee_{k\geq 0} \Sigma^\infty (\bGr_k^\mR)^{\nu_k\oplus \mathfrak{ad}(k)} \ \to \ \Sigma^\infty_+ \bSO
  \]
  is an isomorphism in the global stable homotopy category, and the morphism 
  \[  
    \sum (\Sigma^\infty_+\iota|_{\bSU})^{-1}\circ s^\mC_{k,1}\ : \ 
    \bigvee_{k\geq 0} \Sigma^\infty (\bGr_k^\mC)^{\nu_k\oplus \mathfrak{ad}(k)} \ \to \ \Sigma^\infty_+ \bSU
  \]
  is an isomorphism in the $G(\mC)$-global stable homotopy category.
\end{theorem}

\section{The limit case}
\label{sec:limit}

The stable global splittings of $\bO/m$, $\bU/m$ and $\bSp/m$ are sufficiently compatible
so that we can pass to the colimit in $m$, and also obtain
stable global splittings of the $E_\infty$-global spaces $\bO/\infty$, $\bU/\infty$ and $\bSp/\infty$.
Curiously, this is a purely equivariant phenomenon: the underlying non-equivariant
spaces of  $\bO/\infty$, $\bU/\infty$ and $\bSp/\infty$ are contractible,
so from a non-equivariant perspective, this colimit splitting has no content.
Equivariantly and globally, however, $\bO/\infty$, $\bU/\infty$ and $\bSp/\infty$
are non-trivial and very interesting homotopy types.

\begin{rk}[The $E_\infty$-global space $\bO/\infty$]\label{rk:O/infty}
  To advertise the global spaces $\bO/\infty$, $\bU/\infty$ and $\bSp/\infty$,
  we recall some background on how these measure the difference between
  specific global refinements of the spaces $B O$, $B U$ and $B S p$.
  Since these remarks are mostly motivational, we don't give complete details and
  we restrict to the real case $\bO/\infty$;
  analogous remarks apply to $\bU/\infty$ and $\bSp/\infty$, with additional Galois equivariance.

  Postcomposition with the embedding $i:\mR^m\to\mR^{m+1}$ by $i(x)=(x,0)$ is a closed embedding
  \[ (V\oplus i)\circ \ - \ : \ \bL(V,V\oplus\mR^m)\ \to \ \bL(V,V\oplus\mR^{m+1})\ .  \]
  In the colimit over $m$ we obtain the infinite Stiefel manifold $\bL(V,V\oplus\mR^\infty)$;
  the eigenspace filtration and the conjugation maps \eqref{eq:conjugate by varphi} make just as much sense
  in the limiting case. We can thus define an orthogonal space $\bO/\infty$ with values
  \[ (\bO/\infty)(V)\ = \ \bL(V,V\oplus\mR^\infty)\ , \]
  and an exhausting eigenspace filtration by orthogonal subspaces $F_k(\bO/\infty)$.
  Since the orthogonal space $\bO/\infty$ is objectwise contractible, its underlying
  non-equivariant homotopy type is boring. However, $\bO/\infty$ represents a very interesting
  and non-trivial global homotopy type, as we shall now explain.

  The classifying space $B O$
  of the infinite orthogonal group has two particularly interesting global refinements
  $\bbO$ and $\bBO$, see Examples 2.4.1 and 2.4.18 of \cite{schwede:global}.
  Their values at an inner product space $V$ are given by
  \[ \bbO(V)\ = \ Gr_{|V|}(V\oplus \mR^\infty)\text{\qquad and\qquad}
     \bBO(V)\ = \ Gr_{|V|}(V\oplus V\oplus \mR^\infty)\ ,
  \]
  the Grassmannians of $\dim(V)$-planes in
  $V\oplus \mR^\infty$ and in $V\oplus V\oplus \mR^\infty$, respectively.
  The structure maps are given by 
  \begin{alignat*}{3}
    \bbO(\varphi)(L)\ &= \ \ (\varphi\oplus \mR^\infty)(L)\ &&+ \ \ ((W-\varphi(V))\oplus 0) \text{\quad and}\\
    \bBO(\varphi)(L)\ &= \ (\varphi\oplus\varphi\oplus \mR^\infty)(L)\ &&+ \ ((W-\varphi(V))\oplus 0\oplus 0) \ .
  \end{alignat*}
  To be completely honest, the orthogonal space $\bBO$ of \cite{schwede:global} is not literally
  the one we use here, as it is built from Grassmannians in $V\oplus V$ as opposed
  to $V\oplus V\oplus\mR^\infty$;
  but the two are globally equivalent by \cite[Proposition 2.4.28]{schwede:global}.
  As we explain in detail in Section 2.4  of \cite{schwede:global}, the small
  difference in the definitions of $\bbO$ and $\bBO$ has a drastic effect on the global homotopy type.
  A morphism $i:\bbO\to\bBO$ is given at $V$ by the effect of the embedding
  \[ V\oplus \mR^\infty \ \to \ V\oplus V\oplus\mR^\infty\ , \quad (v,x)\ \longmapsto\ (0,v,x)\ ;
  \]
  the morphism $i:\bbO\to\bBO$ is a non-equivariant equivalence, but {\em not} a global equivalence.
  Each of $\bbO$ and $\bBO$ comes with an associated global Thom spectrum $\bmO$ and
  $\bMO$, and these represent equivariant bordism and stable equivariant bordism,
  respectively, see Sections 6.1 and 6.2 of \cite{schwede:global}.

  Now we connect the orthogonal space $\bO/\infty$ to the two global forms of $B O$.
  The maps
  \[ p(V)\ : \
    (\bO/\infty)(V)\ = \ \bL(V,V\oplus\mR^\infty)\ \to \ G r_{|V|}(V\oplus\mR^\infty)\ = \ \bbO(V)  \]
  that send a linear isometric embedding to its image form a morphism of 
  orthogonal spaces $p:\bO/\infty\to\bbO$.
  The composite $i\circ p:\bO/\infty\to \bBO$ is null-homotopic, as witnessed by the
  system of compatible homotopies
  \begin{align*}
     H \ &: \  [0,1]\times \bL(V,V\oplus\mR^\infty)\ \to \ G r_{|V|}(V\oplus V\oplus\mR^\infty) \\
    H(t,f)\ &= \ \text{image}\left( v\mapsto ( t v, \sqrt{1-t^2}\cdot f(v))\right)\ . 
  \end{align*}
  We write $P(\bBO)$ for the path object of $\bBO$,
  i.e., $P(\bBO)(V)$ is the space of paths in $\bBO(V)$
  that end in the subspace $V\oplus 0\oplus 0$.
  The null-homotopy $H$ is adjoint to a morphism
  \[ \tilde H \ : \ \bO/\infty \ \to \ P(\bBO) \]
  that participates in a commutative diagram of orthogonal spaces:
  \begin{equation}\begin{aligned}\label{eq:bO2BO}
   \xymatrix{
      \bO\ar[r]\ar[d] &  \bO/\infty \ar[r]^-{\tilde H} \ar[d]_p & P(\bBO) \ar[d] \\
      \ast \ar[r] & \bbO \ar[r]_i & \bBO }
  \end{aligned}
  \end{equation}
  We claim without proof that both squares  are globally homotopy cartesian.
  Since the path object $P(\bBO)$ is globally trivial,
  the right square expresses $\bO/\infty$ as the global homotopy fiber of the morphism $i:\bbO\to\bBO$.
  In this sense,  $\bO/\infty$ measures the difference between $\bbO$ and $\bBO$.

  The situation becomes even more interesting by the additional multiplicative structure present.
  Indeed, the orthogonal space
  $\bBO$ can be refined to a globally group-like ultra-commutative monoid,
  see \cite[Example 2.4.1]{schwede:global}.
  The orthogonal space $\bbO$ supports a global $E_\infty$-multiplication,
  implemented by an action of the linear isometries operad,
  compare \cite[Remark 2.4.25]{schwede:global}.
  By \cite[Proposition 2.4.29]{schwede:global},
  the $E_\infty$-structure on $\bbO$ is {\em not} group-like
  and cannot be refined to an ultra-commutative multiplication.
  In contrast to $\bO/m$ for finite $m$, the limiting object $\bO/\infty$ supports
  an $E_\infty$-structure in a very similar way as $\bbO$,
  and all the morphisms in the diagram \eqref{eq:bO2BO} are morphisms of $E_\infty$-orthogonal spaces.
\end{rk}

As we explain in the next construction, the identification of the strata and subquotients
of the eigenspace filtration of $\bL/m$ are compatible with increasing $m$,
and thus yield analogous identifications for the eigenspace filtration of $\bL/\infty$.

\begin{construction}
  The open embeddings \eqref{eq:Cayley}
  are compatible with increasing $m$ along the embedding $i:\mK^m\to\mK^{m+1}$,  $i(x)=(x,0)$.
  The same is true for the open embeddings \eqref{eq:Millers bundle iso},
  so the associated collapse maps participate in a commutative square of orthogonal spaces:
  \[ \xymatrix@C=15mm{
      F_k(\bL/m)\ar[r]^-\Psi \ar[d]_{F_k(\bL/i)} &
      (\bGr_k^\mK)^{\nu(k,m)\oplus\mathfrak{ad}(k)}\ar[d]^{(\bGr_k^\mK)^{\nu(k,i)\oplus\mathfrak{ad}(k)}}\\
      F_k(\bL/(m+1))\ar[r]_-\Psi & (\bGr_k^\mK)^{\nu(k,m+1)\oplus\mathfrak{ad}(k)} 
    } \]
  We now pass to colimits over $m$ in the vertical direction;
  we denote the colimit of the right vertical sequence by $(\bGr_k^\mK)^{\nu(k,\infty)\oplus\mathfrak{ad}(k)}$.
  Theorem \ref{thm:strata identification} then implies that
  the morphism of orthogonal $G(\mK)$-spaces
  $\Psi: F_k(\bL/\infty) \to (\bGr_k^\mK)^{\nu(k,\infty)\oplus \mathfrak{ad}(k) } $
  factors through an isomorphism
  \[  F_k(\bL/\infty) / F_{k-1}(\bL/\infty)\ \iso \ (\bGr_k^\mK)^{\nu(k,\infty)\oplus \mathfrak{ad}(k) }   \ .     \]
\end{construction}

Now we explain that also the stable splitting morphisms
$s_{k,m}:\Sigma^\infty (\bGr_k^\mK)^{\nu(k,m)\oplus\mathfrak{ad}(k)}\to \Sigma^\infty_+ F_k(\bL/m)$  
of the eigenspace filtrations are sufficiently compatible for varying $m$.

\begin{construction}
  The open embeddings \eqref{eq:split_embedding}
  are compatible with increasing $m$ along the embedding $i:\mK^m\to\mK^{m+1}$,
  $i(x)=(x,0)$.
  So the associated collapse maps \eqref{collapse_symplectic} form a commutative square:
  \[ \xymatrix@C=15mm{
      S^{\nu(k,m)\oplus\mathfrak{ad}(k)}\sm S^{\mathfrak{s a}(k)} \ar[r]^-{t_{k,m}}
      \ar[d]_{S^{\nu(k,i)\oplus\mathfrak{ad}(k)}\sm S^{\mathfrak{s a}(k)}} &
      \bL^\mK(\mK^k,\mK^{k+m})_+\sm S^{\mathfrak{s a}(k)}  \ar[d]^{\bL^\mK(\mK^k,\mK^k\oplus i)_+\sm S^{\mathfrak{s a}(k)}}\\
      S^{\nu(k,m+1)\oplus\mathfrak{ad}(k)}\sm S^{\mathfrak{s a}(k)} \ar[r]_-{t_{k,m+1}} & \bL^\mK(\mK^k,\mK^{k+m+1})_+\sm S^{\mathfrak{s a}(k)} 
    } \]
  The associated homotopy classes thus satisfy the relation
  \[ (\Sigma^\infty_+ F_k(\bL/i))_*\td{t_{k,m}}\ = \ (\nu(k,i)\oplus\mathfrak{ad}(k))^*\td{t_{k,m+1}} \]
  in the representation-graded equivariant homotopy group 
  \[ \pi_{\nu(k,m)\oplus\mathfrak{ad}(k)}^{\tilde I(k)}\left( \epsilon_k^*( \Sigma^\infty_+ F_k(\bL/(m+1)))\right) \ . \]
  Here $(\nu(k,i)\oplus\mathfrak{ad}(k))^*$ is the grading-changing homomorphism
  given by precomposition with
  $S^{\nu(k,i)\oplus\mathfrak{ad}(k)}:S^{\nu(k,m)\oplus\mathfrak{ad}(k)}\to S^{\nu(k,m+1)\oplus\mathfrak{ad}(k)}$.
  The representability isomorphism of Corollary \ref{cor:J_is_A-global} turns this relation
  into a commutative square in the $G(\mK)$-global homotopy category:
  \[ \xymatrix@C=15mm{
      \Sigma^\infty (\bGr_k^\mK)^{\nu(k,m)\oplus\mathfrak{ad}(k)}\ar[r]^-{s_{k,m}}
      \ar[d]_{(\bGr_k^\mK)^{\nu(k,i)\oplus\mathfrak{ad}(k)}} &
      \Sigma^\infty_+ F_k(\bL/m)\ar[d]^{\Sigma^\infty_+ F_k(\bL/i)} \\
      \Sigma^\infty (\bGr_k^\mK)^{\nu(k,m+1)\oplus\mathfrak{ad}(k)}\ar[r]_-{s_{k,m+1}} &\Sigma^\infty_+ F_k(\bL/(m+1))
    } \]
  Now we pass to homotopy colimits, in the sense of triangulated categories,
  in the vertical direction; since the vertical morphisms
  in the previous square are represented by actual morphisms of orthogonal $G(\mK)$-spaces
  that are objectwise closed embeddings, the vertical homotopy colimits are modeled
  by the suspension spectra of the corresponding colimits of orthogonal $G(\mK)$-spaces.
  So there exists a morphism in  the $G(\mK)$-global stable homotopy category
  \[  s_{k,\infty}\ : \ \Sigma^\infty (\bGr_k^\mK)^{\nu(k,\infty)\oplus\mathfrak{ad}(k)}\ \to \
    \Sigma^\infty_+ F_k(\bL/\infty)  \]
  such that all the squares
  \[ \xymatrix@C=15mm{
      \Sigma^\infty (\bGr_k^\mK)^{\nu(k,m)\oplus\mathfrak{ad}(k)}\ar[r]^-{s_{k,m}}
      \ar[d] &      \Sigma^\infty_+ F_k(\bL/m)\ar[d]\\
      \Sigma^\infty (\bGr_k^\mK)^{\nu(k,\infty)\oplus\mathfrak{ad}(k)}\ar[r]_-{s_{k,\infty}} &\Sigma^\infty_+ F_k(\bL/\infty)
    } \]
  commute in $\GH_{G(\mK)}$.
  Sequential homotopy colimits in triangulated categories are only weak colimits,
  so the morphism $s_{k,\infty}$ is only determined by this property
  up a potential ambiguity measured by a $\lim^1$-term.
  We write
  \[ s_{k,\infty}^\mK\ : \ \Sigma^\infty (\bGr_k^\mK)^{\nu(k,\infty)\oplus\mathfrak{ad}(k)}\ \to \
    \Sigma^\infty_+ \bL/\infty   \]
  for the composite of the morphism $s_{k,\infty}$ 
  and the morphism of suspension spectra induced by the inclusion $F_k(\bL/\infty)\to\bL/\infty$.
\end{construction}

\begin{theorem}\label{thm:stably split L/infty}
  Let $\mK$ be one of the skew-fields $\mR$, $\mC$ or $\mH$.
  The morphism
  \[ \sum s_{k,\infty}^\mK \ :\ \bigvee_{k\geq 0}\  \Sigma^\infty (\bGr_k^\mK)^{\nu(k,\infty)\oplus \mathfrak{ad}(k)}
    \ \to \  \Sigma^\infty_+ \bL/\infty \]
  is an isomorphism in the $G(\mK)$-global stable homotopy category.  
\end{theorem}
\begin{proof}
  We let $\alpha:G\to G(\mK)$ be any continuous homomorphism from a compact Lie group.
  Both vertical maps in the commutative square of graded abelian groups
  \[ \xymatrix@C=25mm{
     \bigoplus_{k\geq 0}
     \colim_m \ \pi_*^G(\alpha^*(\Sigma^\infty (\bGr_k^\mK)^{\nu(k,m)\oplus\mathfrak{ad}(k)}))
      \ar[r]^-{\sum_k\colim_m(s^\mK_{k,m})_*}
      \ar[d] &  \colim_m\ \pi_*^G(\alpha^*( \Sigma^\infty_+ \bL/m))\ar[d]\\
       \bigoplus_{k\geq 0}\ \pi_*^G(\alpha^*(\Sigma^\infty (\bGr_k^\mK)^{\nu(k,\infty)\oplus\mathfrak{ad}(k)}))
      \ar[r]_-{\sum_k(s^\mK_{k,\infty})_*} & \pi_*^G(\alpha^*(\Sigma^\infty_+ \bL/\infty))
    } \]
  are isomorphisms, and the upper horizontal map is an isomorphism by Theorem \ref{thm:stably split L/m}.
  So the lower vertical map is an isomorphism. Since equivariant homotopy groups
  takes wedges to direct sums, this proves the claim.
\end{proof}

\begin{appendix}

\section{A glimpse of \texorpdfstring{$C$}{C}-global homotopy theory}
\label{sec:C-global}

In this section we give a brief introduction to $C$-global homotopy theory,
where $C$ is a topological group.
In this refinement of global homotopy theory,
all objects are equipped with an external action of $C$.
For the application to the global stable splitting of Stiefel manifolds we are mostly interested
in the special case where $C$ is
the Galois group $G(\mC)$ of $\mC$ over $\mR$ (a discrete group of order 2),
or where $C$ is the group $G(\mH)$ of $\mR$-algebra automorphism of the quaternions
(a compact Lie group isomorphic to $S O(3)$).
However, the basic theory works just as well over arbitrary topological groups,
so we develop it in that generality.

The philosophy behind $C$-global homotopy theory
is to merge the `global' direction of \cite{schwede:global}
with the `proper stable homotopy theory' in the spirit of \cite{dhlps}.
So while we allow arbitrary topological groups to act, all homotopical information
is probed by restriction along continuous homomorphisms $\alpha:G\to C$
whose source $G$ is a {\em compact Lie group}.
We will introduce unstable and stable $C$-global homotopy theory via particular models.
For the unstable theory, we will use orthogonal $C$-spaces;
for the stable theory, we will use orthogonal $C$-spectra. 

In order to keep the length of this paper within a reasonable bound,
we introduce just enough formalism to be able to formulate and prove 
the $G(\mC)$-global stable splitting of $\bU/m$ and
the $G(\mH)$-global stable splitting of $\bSp/m$ in Theorem \ref{thm:stably split L/m}.
To this end we set up the triangulated $C$-global stable homotopy category $\GH_C$,
identify global Thom spaces over global classifying spaces as representing objects
for equivariant homotopy groups (see Theorem \ref{thm:Bgl alpha represents}),
and prove that $\GH_C$ is compactly generated (see Corollary \ref{cor:compactly generated}).
The global stable splitting of $\bO/m$ does not involve any extrinsic group actions,
so it can be formulated entirely in the framework of \cite{schwede:global}
and does not need the tools from this appendix.

In the special case of discrete groups $C$,
and when probing through homomorphisms from finite groups (as opposed to compact Lie groups),
Lenz \cite{lenz:G-global} introduces several models for
unstable and stable $C$-global homotopy theory;
among these are models based on $I$-spaces and symmetric spectra,
the discrete analogs of orthogonal spaces and orthogonal spectra.
In his context, Lenz develops substantially more theory
and provides many more tools than we do here.
In the special case of compact Lie groups $C$,
Barrero \cite[Theorem A.20]{barrero:operads}
exhibits the $C$-global equivalences as the weak equivalence in the
{\em $C$-global model structure} on the category of orthogonal $C$-spaces.

\subsection{Unstable \texorpdfstring{$C$}{C}-global homotopy theory}
We start with an introduction to unstable $C$-global homotopy theory,
where $C$ is any topological group.
Without the external $C$-action, there are several models for unstable global homotopy theory available;
historically, the first model were the orbispaces of Gepner-Henriques \cite{gepner-henriques}
that model a suitable homotopy theory of topological stacks.
More recently, the author developed models in terms
of orthogonal spaces \cite[Section 1]{schwede:global}
and spaces with an action of the `universal compact Lie group' \cite{schwede:universal_Lie}.
We will generalize the orthogonal space model now,
and work with orthogonal $C$-spaces,
i.e., orthogonal spaces equipped with a continuous $C$-action;
the relevant homotopy theory is encoded in
the class of {\em $C$-global equivalences}, see Definition \ref{def:unstable C-global equivalence}.

\medskip

For the purposes of this paper, a {\em space} is a compactly generated space in the
sense of \cite{mccord}, i.e., a $k$-space (also called Kelley space) that
satisfies the weak Hausdorff condition.
We write $\bT$ for the category of compactly generated spaces and continuous maps.
A {\em topological group} is a group object internal to
the category $\bT$ of compactly generated spaces.
A {\em $C$-space} is then a $C$-object internal to $\bT$, i.e., a compactly generated space $X$ 
equipped with an associative and unital action 
$C  \times X \to  X$
that is continuous with respect to the compactly generated product topology.
We write $C\bT$ for the category of $C$-spaces and continuous $C$-maps.

As in \cite{schwede:global}, we denote by $\bL$ the category with objects the 
finite-dimensional euclidean inner product spaces
and morphisms the $\mR$-linear isometric embeddings.
The morphism spaces of the category $\bL$ come with a preferred topology
as Stiefel manifolds; this makes $\bL$ into a topological category.

\begin{defn}
Let $C$ be a topological group. An {\em orthogonal $C$-space} is a continuous functor from the
linear isometries category $\bL$ to the category of $C$-spaces.
A {\em morphism} of orthogonal $C$-spaces is a natural transformation of functors.
We denote the category of orthogonal $C$-spaces by $\spc_C$.
\end{defn}

The use of continuous functors from the category $\bL$ to spaces has a long history
in homotopy theory. 
The category $\bL$ (or its extension that
also contains countably infinite-dimensional inner product spaces)
is denoted $\mathscr I$ by Boardman and Vogt \cite{boardman-vogt:homotopy everything},
and this notation is also used in \cite{may-quinn-ray};
other sources \cite{lind:diagram} use the symbol $\mathcal I$.
Accordingly, orthogonal spaces are sometimes referred to as $\mathscr I$-functors,
$\mathscr I$-spaces or $\mathcal I$-spaces.

\medskip

Orthogonal $C$-spaces admit a refinement of global homotopy theory
that takes the $C$-action into account, and that generalizes the unstable global homotopy theory
as developed in \cite[Chapter 1]{schwede:global}.
The additional homotopical information is located at compact Lie groups `augmented over $C$',
i.e., compact Lie groups $G$ equipped with a continuous homomorphism $\alpha:G\to C$.

For us, a {\em representation} of a compact Lie group $G$ is an inner product space $V$
equipped with a continuous $G$-action through linear isometries.
Such an action can also be packaged as a continuous homomorphism $\rho:G\to  O(V)$
to the orthogonal group of $V$.
For every orthogonal $C$-space $Y$, every continuous homomorphism $\alpha:G\to C$ and
$G$-representation $V$, the value $Y(V)$
comes with a continuous $(G\times C)$-action
from the $G$-action on $V$ and the $C$-action on $Y$.
For a $G$-equivariant linear isometric embedding $\varphi:V\to W$,
the induced map $Y(\varphi):Y(V)\to Y(W)$ is $(G\times C)$-equivariant.
We write
\[  \Gamma(\alpha)\ =\ \{(g,\alpha(g))\ |\  g\in G\} \]
for the graph of $\alpha$, a closed subgroup of $G\times C$.
We denote by
\[ D^k \ = \ \{x\in\mR^k \ : \ |x|\leq 1\} \text{\qquad and\qquad}
 \partial D^k \ = \ \{x\in\mR^k \ : \ |x|= 1\} \]
the unit disc in $\mR^k$ and its boundary, a sphere of dimension $k-1$,
respectively.
In particular, $D^0=\{0\}$ is a one-point space and $\partial D^0=\emptyset$
is empty.

\begin{defn}\label{def:unstable C-global equivalence}
  Let $C$ be a topological group.
  A morphism $f:X\to Y$ of orthogonal $C$-spaces is a {\em $C$-global equivalence}
  if the following condition holds:
  for every compact Lie group $G$,
  every continuous homomorphism $\alpha:G\to C$,
  every $G$-representation $V$,
  every $k\geq 0$ and all continuous maps $a:\partial D^k\to X(V)^{\Gamma(\alpha)}$
  and $b:D^k\to Y(V)^{\Gamma(\alpha)}$ such that $b|_{\partial D^k}=f(V)^{\Gamma(\alpha)}\circ a$,
  there is a $G$-representation $W$,
  a $G$-equivariant linear isometric embedding $\varphi:V\to W$
  and a continuous map $\lambda:D^k\to X(W)^{\Gamma(\alpha)}$
  such that $\lambda|_{\partial D^k}=X(\varphi)^{\Gamma(\alpha)}\circ a$ and
  such that $f(W)^{\Gamma(\alpha)}\circ \lambda$
  is homotopic, relative to $\partial D^k$, to $Y(\varphi)^{\Gamma(\alpha)}\circ b$.
\end{defn}

If the group $C$ is trivial, then the notion of $C$-global equivalence specializes
to the global equivalences of \cite[Definition 1.1.2]{schwede:global}.
We recall from \cite[Definition 1.1.16]{schwede:global} that an orthogonal space $Y$ is {\em closed}
if it takes every linear isometric embedding $\varphi:V\to W$ of inner product spaces
to a closed embedding $Y(\varphi):Y(V)\to Y(W)$.
Most orthogonal spaces that occur naturally are closed,
and for morphisms between closed orthogonal $C$-spaces,
the next proposition  provides a useful criterion for
detecting $C$-global equivalences.

We let $G$ be a compact Lie group. We recall that a {\em complete $G$-universe}
is an orthogonal $G$-representation of countably infinite dimension into which
every finite-dimensional $G$-representation embeds,
by an equivariant $\mR$-linear isometric embedding.
Complete $G$-universe exists, they are unique up to equivariant linear
isomorphism, and the space of equivariant linear isometric self-embeddings
of a complete $G$-universe is contractible.
If $H$ is a closed subgroup of a compact Lie group $G$,
then the underlying $H$-representation of a complete
$G$-universe is a complete $H$-universe.
In the following, for every compact Lie group $G$ we fix
a complete $G$-universe $\Uc_G$.
We let $s(\Uc_G)$ denote the poset, under inclusion,
of finite-dimensional $G$-subrepresentations of $\Uc_G$.
The {\em underlying $G$-space} of an orthogonal space $Y$ is
\[ Y(\Uc_G)\ = \ \colim_{V\in s(\Uc_G)}\, Y(V)\ , \]
the colimit of the $G$-spaces $Y(V)$. 
If $Y$ is an orthogonal $C$-space for some topological group $C$,
then $Y(\Uc_G)$ becomes a $(G\times C)$-space.
The next proposition is a generalization of \cite[Proposition 1.1.17]{schwede:global};
the proof is almost verbatim the same, and we omit it. 

\begin{prop}\label{prop:global eq for closed} 
  Let $C$ be a topological group.
  Let $f:X\to Y$ be a morphism between orthogonal $C$-spaces
  whose underlying orthogonal spaces are closed.
  Then $f$ is a $C$-global equivalence if and only if
  for every continuous homomorphism  $\alpha:G\to C$ from a compact Lie group, the map
  \[ f(\Uc_G)^{\Gamma(\alpha)}\ : \ X(\Uc_G)^{\Gamma(\alpha)} \ \to \  Y(\Uc_G)^{\Gamma(\alpha)}\]
  is a weak equivalence.
\end{prop}

The basic building blocks of `classical' global homotopy theory
(i.e., global homotopy theory without an external action
of any additional group) are the global classifying spaces of compact Lie groups.
For example, these global classifying spaces generate the $\infty$-category
of global spaces under colimits, and their suspension spectra
generated the global stable homotopy category
as a triangulated category \cite[Theorem 4.4.3]{schwede:global}.
In the orthogonal space model, the global classifying space $B_{\gl}G$ of
a compact Lie group is  introduced in \cite[Definition 1.1.27]{schwede:global};
the name is motivated by the fact that $B_{\gl}G$ globally classifies $G$-equivariant principal bundles,
compare \cite[Proposition 1.1.30]{schwede:global}.
The counterpart of $B_{\gl}G$ in the world of topological stacks
is thus the stack of principal $G$-bundles.

Global classifying spaces also exist in $C$-global homotopy theory,
and we will introduce them now.
In the $C$-global context, these object are not associated to compact Lie groups,
but rather to continuous homomorphisms $\alpha:G\to C$ from compact Lie groups.
As in the `classical' case, their global classifying spaces $B_{\gl}\alpha$
are the basic building blocks of $C$-global homotopy theory.
A rigorous statement to this effect is our Theorem \ref{thm:Bgl alpha represents}
below, saying that $C$-global stable homotopy category is compactly generated
by the suspension spectra of the global classifying spaces
of all continuous homomorphisms from compact Lie groups to $C$.

\begin{construction}[Global classifying spaces]\label{con:Bgl of alpha}
  We let $C$ be a topological group, $G$ a compact Lie group, and $\alpha:G\to C$
  a continuous homomorphism.
  We let $V$ be a faithful $G$-representation.
  Then the assignment
  \[ B_{\gl}\alpha \ = \ C\times_\alpha \bL(V,-)\ : \ \bL \ \to \ C\bT \]
  is an orthogonal $C$-space, the {\em global classifying space} of $\alpha$.
  The value of $B_{\gl}\alpha$ at an inner product space $W$ is thus the
  orbit $C$-space of the $G$-action on $C\times \bL(V,W)$ by
  \[ (c,\varphi)\cdot g \ = \ (c\alpha(g), \varphi\circ l_g)\ , \]
  where $l_g:V\to V$ is left multiplication by the group element $g$.
\end{construction}

When the group $C$ is trivial, $B_{\gl}\alpha$ specializes to
a global classifying space $B_{\gl}G$ for the compact Lie group $G$ as defined in
\cite[Definition 1.1.27]{schwede:global}.
The following Proposition \ref{prop:Bgl independence}
in particular shows that the $C$-global homotopy type of $B_{\gl}\alpha$
is independent of the choice of faithful $G$-representation;
the proposition generalizes \cite[Proposition 1.1.26]{schwede:global}.
Given two inner product spaces $V$ and $W$,
the restriction homomorphism of represented orthogonal spaces
\[ \rho_{V,W}\ : \ \bL(V\oplus W,-)\ \to \  \bL(V,-) \]
restricts a linear isometric embedding from $V\oplus W$ to the first summand.

\begin{prop}\label{prop:Bgl independence}
  Let $\alpha:G\to C$ be a continuous homomorphism
  from a compact Lie group to a topological group.
  Let $V$ and $W$ be $G$-representations such that $G$ acts faithfully on $V$, and let $A$ be a $G$-space.
  Then the restriction morphism
  \[ C\times_\alpha(\rho_{V,W}\times A)\ : \ C\times_\alpha (\bL(V\oplus W,-)\times A)\ \to
    C\times_\alpha (\bL(V,-)\times A)\]
  is a $C$-global equivalence.
\end{prop}
\begin{proof}
  We claim that the underlying orthogonal space of $C\times_\alpha(\bL(V,-)\times A)$ is closed;
  the same is then also true for $C\times_\alpha(\bL(V\oplus W,-)\times A)$.
  Indeed, if $\varphi:U\to U'$
  is a linear isometric embedding, then the induced map of Stiefel manifolds
  $\bL(V,\varphi):\bL(V,U)\to\bL(V,U')$ is a closed embedding, hence so
  is the map $C\times \bL(V,\varphi)\times A$.
  Taking orbits by a continuous action of a compact topological group
  preserves closed embeddings, see \cite[Proposition B.13 (iii)]{schwede:global},
  so the map $C\times_\alpha(\bL(V,\varphi)\times A)$ is a closed embedding, too.

  Now we let $K$ be another compact Lie group, and we let $\Uc_K$ a complete $K$-universe.
  The map
  \[ \rho_{V,W}(\Uc_K)\ : \  \bL(V\oplus W,\Uc_K)\ \to \  \bL(V,\Uc_K)\]
  is a $(K\times G)$-homotopy equivalence by \cite[Proposition 1.1.26]{schwede:global}.
  So the map
  \[ C\times_\alpha (\rho_{V,W}(\Uc_K)\times A)\ : \
    C\times_\alpha (\bL(V\oplus W,\Uc_K)\times A)\ \to \
    C\times_\alpha (\bL(V,\Uc_K)\times A) \]
  is a $(K\times C)$-homotopy equivalence.
  Hence for every continuous homomorphism $\beta:K\to C$,
  the map of $\Gamma(\beta)$-fixed points
  $\left(C\times_\alpha (\rho_{V,W}(\Uc_K)\times A)\right)^{\Gamma(\beta)}$
  is a weak equivalence.
  Proposition \ref{prop:global eq for closed} then applies
  to show that the morphism $C\times_\alpha(\rho_{V,W}\times A)$ is a $C$-global equivalence.
\end{proof}

\subsection{Stable \texorpdfstring{$C$}{C}-global homotopy theory}
We continue to let $C$ be any topological group. An {\em orthogonal $C$-spectrum}
is an orthogonal spectrum equipped with a continuous $C$-action by automorphisms
of orthogonal spectra.
A {\em morphism} of orthogonal $C$-spectra is a $C$-equivariant morphism of orthogonal spectra.
If $\alpha:G\to C$ is a continuous homomorphism between topological groups
and $Y$ is an orthogonal $C$-spectrum, we write $\alpha^* Y$ for the
orthogonal $G$-spectrum with the same underlying orthogonal spectrum,
and with $G$-action through the homomorphism $\alpha$.

For a compact Lie group $G$, the $k$-th {\em equivariant homotopy group} $\pi_k^G(Y)$
of an orthogonal $G$-spectrum,
based on a complete $G$-universe, is defined, for example, in \cite[(3.1.11)]{schwede:global}.
We emphasize that while we allow for actions of arbitrary topological groups,
equivariant homotopy groups only show up for actions of {\em compact Lie groups}.

\begin{defn}\label{def:stable C-global equivalence}
Let $C$ be a topological group. A morphism $f:X\to Y$ of orthogonal $C$-spectra 
is a {\em $C$-global equivalence}
if for every compact Lie group $G$, every continuous homomorphism $\alpha:G\to C$
and all integers~$k$,
the map $\pi_k^G(\alpha^* f):\pi_k^G(\alpha^* X) \to \pi_k^G(\alpha^* Y)$ is an isomorphism.
\end{defn}

An equivalent way to recast the previous is definition is as follows.
A morphism $f:X\to Y$ of orthogonal $C$-spectra 
is a $C$-global equivalence if and only if for
every continuous homomorphism $\alpha:G\to C$ from a compact Lie group,
the morphism of orthogonal $G$-spectra $\alpha^* f:\alpha^* X\to \alpha^* Y$
is a $\upi_*$-isomorphism in the sense of \cite[Definition 3.1.12]{schwede:global}.

A morphism $f:A\to B$ of orthogonal $C$-spectra is an {\em h-cofibration}
if it has the homotopy extension property, 
i.e., given a morphism of orthogonal $C$-spectra $\varphi:B\to X$ and a homotopy
$H:A\sm [0,1]_+\to X$ starting with $\varphi f$,
there is a homotopy $\bar H:B\sm [0,1]_+\to X$ starting with $\varphi$
such that $\bar H\circ(f\sm [0,1]_+)=H$.

\begin{prop}\label{prop:h-cof for C}
  Let $C$ be a topological group.
  \begin{enumerate}[\em (i)]
  \item Let $\{f_i:X_i\to Y_i\}_{i\in I}$ be a family of 
    $C$-global equivalences of orthogonal $C$-spectra.
    Then the coproduct $\bigvee_{i\in I}f_i:\bigvee_{i\in I}X_i\to\bigvee_{i\in I}Y_i$
    is a $C$-global equivalence.
  \item
    Let 
    \[ \xymatrix{ A \ar[r]^-f\ar[d]_g & B\ar[d]\\ C \ar[r]_-k & D} \]
    be a pushout square of orthogonal $C$-spectra such that $f$ is a $C$-global equivalence.
    If, in addition, $f$ or $g$ is an h-cofibration, then also the morphism $k$ is a $C$-global equivalence.
  \item
    Let $f_n:Y_n\to Y_{n+1}$ be h-cofibrations of orthogonal $C$-spectra
    that are also $C$-global equivalences, for $n\geq 0$.
    Then the canonical morphism $f_\infty:Y_0\to \colim_n Y_n$ to any colimit of the sequence
    is a $C$-global equivalence.
  \end{enumerate}
\end{prop}
\begin{proof}
  The fact that the restriction functor $\alpha^*:\spec_C\to\spec_G$
  along a continuous homomorphism $\alpha:G\to C$ preserves h-cofibrations and colimits
  reduces all three claims to the special case of compact Lie groups.
  In that context, all three statements can be found in \cite[III Theorem 3.5]{mandell-may},
  and proofs can be found in Corollary 3.1.37, Corollary 3.1.39 and Proposition 3.1.41
  of \cite{schwede:global}.
\end{proof}

We will now argue that the classes of $C$-global equivalences and h-cofibrations
make the category of orthogonal $C$-spectra into a {\em cofibration category}.
A cofibration category is a category equipped with two classes of morphisms,
the `weak equivalences' and the `cofibrations', that satisfy a specific list of axioms.
The original notion and the basic theory goes back to Brown \cite{brown:abstract}
who had introduced the dual notion under the name `category of fibrant objects'.
After Brown, several authors have proposed and studied variations of the concept
that differ in details;
a nice summary and an extensive list of references can be found in the introduction
of \cite{szumilo:cofibration cats}.
For our purposes, the formulation of Szumi{\l}o
is particularly convenient, and we will use the axioms as stated in
\cite[Section 1]{szumilo:cofibration cats}.

The {\em homotopy category} of a cofibration category $\Cc$ is any localization at the class
of weak equivalences, i.e., a functor $\gamma:\Cc\to\Ho(\Cc)$ that is initial among functors
from $\Cc$ that take weak equivalences to isomorphisms.
We will often refer to the category $\Ho(\Cc)$ alone as the homotopy category, leaving the
localization functor $\gamma$ implicit.
The cofibration structure gives rise to an abstract notion of
homotopy and a `calculus of left fractions', through which the homotopy category becomes manageable.
This calculus includes the following two facts, where `acyclic cofibration'
refers to a morphism that is simultaneously a cofibration and a weak equivalence.
\begin{enumerate}[(i)]
\item  Every morphism in $\Ho(\Cc)$ is a fraction of the form $\gamma(s)^{-1}\circ\gamma(f)$,
  where $f$ and $s$ are $\Cc$-morphisms with the same target, and $s$ is an acyclic cofibration.
\item Given two morphisms $f,g : A\to B$ in $\Cc$, then $\gamma(f) = \gamma(g)$ in $\Ho(\Cc)$
  if and only if there is an acyclic cofibration $s:B\to\bar B$  such that $s f$ and $s g$ are homotopic.
\end{enumerate}
A proof of these facts (or rather the dual statements for `categories of fibrant objects')
can be found in \cite[I.2 Theorem I]{brown:abstract} and Remark 2 immediately thereafter.  

In our applications we will want to know that the $C$-global stable homotopy category
has arbitrary coproducts, and that these are modeled by wedges of orthogonal $C$-spectra.
This fact is a special case of a general property of cofibration categories
with well-behaved set-indexed coproducts, see the following proposition.
The proposition should not be surprising, 
but I am not aware of a reference, so I include a proof.
For {\em finite} coproducts, the following proposition is a special case
of the statement dual to \cite[Corollaire 2.9]{cisinski:categories_derivables}.
Following \cite{szumilo:cofibration cats} we call a cofibration category {\em cocomplete}
if it has set-indexed coproducts, and if the classes of cofibrations and
acyclic cofibrations are stable under coproducts.

\begin{prop}\label{prop:infinite coproducts}
  Let $\Cc$ be a cocomplete cofibration category.
  Then the localization functor $\gamma:\Cc\to \Ho(\Cc)$ preserves
  coproducts. In particular, the homotopy category admits coproducts.
\end{prop}
\begin{proof}
  Since the localization functor $\gamma$ can be arranged to be the identity on objects,
  we drop $\gamma$ in front of objects to simplify the notation.
  Now we consider an index set $I$,
  and an $I$-indexed family $\{X_i\}_{i\in I}$ of $\Cc$-objects.
  We denote a coproduct of the family by $\amalg_{i\in I}X_i$, and we write
  $\kappa_j:X_j\to\amalg_{i\in I} X_i$ for the universal morphisms.
  We must show that for every $\Cc$-object $Y$, the map
  \begin{equation}\label{eq:universal_test}
    \Ho(\Cc)(\amalg_{i\in I}X_i,Y)\ \to \ {\prod}_{j\in I}\Ho(\Cc)(X_j,Y)    \ , \quad
    \psi \longmapsto \ (\psi\circ\gamma(\kappa_j))_{j\in I}
  \end{equation}
  is bijective.
  For surjectivity we let $(\psi_j:X_j\to Y)$ be any $I$-indexed family of morphisms in $\Ho(\Cc)$.
  By the calculus of left fractions, we can write
  \[ \psi_j \ = \ \gamma(s_j)^{-1}\circ \gamma(f_j)  \]
  for some families of $\Cc$-morphisms $f_j:X_j\to W_j$ and $s_j:Y\to W_j$
  such that the morphisms $s_j$ are acyclic cofibrations.
  We choose a coproduct of the family $\{W_i\}_{i\in I}$
  and a coproduct of the constant family $\{Y\}_{i\in I}$ of copies of $Y$.
  Then we form the $\Cc$-morphisms
  \[ \amalg_{i\in I} X_i\ \xra{\amalg f_i} \ \amalg_{i\in I} W_i\ \xla[\sim]{\amalg s_i} \
    \amalg_{i\in I}Y \ \xra{\ \nabla\ } \ Y \ ,\]
  where $\nabla$ denotes the fold morphism.
  Since coproducts of acyclic cofibrations
  are acyclic cofibrations, the middle morphism is an acyclic cofibration.
  So we can form the morphism
  \[  \gamma(\nabla)\circ\gamma(\amalg s_i)^{-1}\circ \gamma(\amalg f_i)\ : \ \amalg_{i\in I} X_i \ \to \ Y\]
  in the homotopy category.
  The map \eqref{eq:universal_test} sends this morphism
  to the original family $(\psi_j)_{j\in I}$;
  so the map \eqref{eq:universal_test} is surjective.

  For injectivity we consider two morphisms $\psi,\psi':\amalg_{i\in I} X_i\to Y$ in $\Ho(\Cc)$
  such that $\psi\circ\gamma(\kappa_j)=\psi'\circ\gamma(\kappa_j)$ for all $j\in I$.
  We start with the special case where $\psi=\gamma(f)$ and $\psi'=\gamma(f')$
  for two $\Cc$-morphisms $f,f':\amalg_{i\in I}X_i\to Y$.
  Because
  \[     \gamma(f\kappa_j)\ = \ 
    \psi\circ\gamma(\kappa_j)\ =\ \psi'\circ\gamma(\kappa_j) = \gamma(f'\kappa_j)\ , \]
  the calculus of left fractions provides acyclic cofibrations $t_j: Y\to \bar Y_j$
  such that $t_j f\kappa_j:X_i\to \bar Y_j$ is homotopic to
  $t_j f'\kappa_j:X_j\to \bar Y_j$ for every $j\in I$.
  We choose a pushout:
  \[ \xymatrix{
      \amalg_{i\in I} Y \ar[r]^-\nabla\ar[d]_{\amalg t_i}^\sim & Y \ar[d]_\sim^t\\
      \amalg_{i\in I} \bar Y_i \ar[r]_-{\nabla'} &  Y'
    } \]
  Since coproducts of acyclic cofibrations are acyclic cofibrations,
  the left vertical morphism is an acyclic cofibration, and hence so is
  the right vertical morphism $t:Y\to Y'$.

  For each $j\in I$, we choose a cylinder object $Z_j$ of $X_j$ and a homotopy $H_j:Z_j\to Y_j$
  from  $t_j f\kappa_j$ to $t_j f'\kappa_j$.
  Since coproducts preserve cofibrations and acyclic cofibrations,
  the coproduct $\amalg_{i\in I}Z_i$ is a cylinder object for $\amalg_{i\in I} X_i$,
  where we leave the additional data of a cylinder object implicit.
  Moreover, the composite
  \[ \amalg_{i\in I}Z_i\ \xra{\amalg H_i}\ \amalg_{i\in I} \bar Y_i\ \xra{\ \nabla'\ } \ Y' \]
  is then a homotopy from $t f$ to $t f'$.
  We conclude that $\gamma(t f)=\gamma(t f')$ in $\Ho(\Cc)$.
  Since $t$ is a weak equivalence, $\gamma(t)$ is an isomorphism in $\Ho(\Cc)$,
  and so $\gamma(f)=\gamma(f')$. This proves injectivity in the special case.

  Now we treat the general case, and we let $\psi,\psi':\amalg_{i\in I}X_i\to Y$
  be arbitrary morphisms in $\Ho(\Cc)$ such that $\psi\circ\gamma(\kappa_j)=\psi'\circ\gamma(\kappa_j)$
  for all $j\in I$.
  The calculus of left fractions provides $\Cc$-morphisms $f:\amalg_{i\in I}X_i\to W$,
  $f':\amalg_{i\in I}X_i\to W'$, $s:Y\to W$ and $s:Y\to W'$
  such that $s$ and $s'$ are acyclic cofibrations and such that
  \[ \psi \ = \ \gamma(s)^{-1}\circ\gamma(f) \text{\qquad and\qquad}
     \psi' \ = \ \gamma(s')^{-1}\circ\gamma(f')\ .  \]
   We choose a pushout:
   \[ \xymatrix{
       Y \ar[r]_-\sim^-s\ar[d]_{s'}^\sim & W \ar[d]_\sim^t\\
       W' \ar[r]^\sim_-{t'} &  V
     } \]
   Then $t$ and $t'$ are acyclic cofibrations because $s$ and $s'$ are.
   We now  obtain the relation
   \[   \gamma(t f)\circ\gamma(\kappa_j)\
     = \ \gamma(t)\circ\gamma(s)\circ\psi\circ\gamma(\kappa_j)\
     = \ \gamma(t')\circ\gamma(s')\circ\psi'\circ\gamma(\kappa_j)\ =\
       \gamma(t' f')\circ\gamma(\kappa_j) \]
   for every $j\in I$.
   The special case treated above lets us conclude that $\gamma(t f)=\gamma(t' f')$.
   Thus
   \[ \gamma(t s)\circ \psi \ = \ \gamma(t f)\ = \ \gamma(t' f')\ = \ \gamma(t' s')\circ \psi' \ .\]
   Because the morphism
   $\gamma(t s)=\gamma(t' s')$ is an isomorphism,
   also $\psi=\psi'$. This completes the proof.
\end{proof}

A cofibration category is {\em pointed} if it has a zero object,
i.e., if every initial object is also terminal.
The homotopy category of a pointed cofibration category supports a specific
{\em suspension functor} $\Sigma:\Ho(\Cc)\to\Ho(\Cc)$,
see the dual to \cite[I.4 Theorem 3]{brown:abstract},
or \cite[Proposition A.4]{schwede:topological}.
A cofibration category is {\em stable} if it is pointed and the
suspension functor is an autoequivalence of the homotopy category.
The homotopy category of a stable cofibration category supports the structure
of a triangulated category, see \cite[Theorem A.12]{schwede:topological}.
A candidate triangle in $\Ho(\Cc)$ is distinguished if and only if it is isomorphic
to the triangle
\[ A \ \xra{\gamma(j) }\ B \ \xra{\gamma(\text{proj})} \ B/A \ \xra{\delta(j)}\ \Sigma A \]
arising from some cofibration $j:A\to B$;
here $\delta(j):B/A\to \Sigma A$ is a specific `connecting morphism' in $\Ho(\Cc)$,
defined in \cite[(A.10)]{schwede:topological}.

We write
\begin{equation}\label{eq:GH_C} 
  \GH_C\ = \ \spec_C[C\text{-global equivalences}^{-1}]
\end{equation}
for the localization of the category of orthogonal $C$-spectra at the class of $C$-global equivalences,
and we refer to it as the {\em $C$-global stable homotopy category}.
The suspension functor  $-\sm S^1:\spec_C\to\spec_C$ preserves $C$-global equivalences,
so it descends to a functor $\Ho(-\sm S^1):\GH_C\to \GH_C$ by the universal property
of localizations.
For every orthogonal $C$-spectrum $X$,
the $C$-spectrum $X\sm S^1$ is a cokernel
of the `cone inclusion' $-\sm 1:X\to X\sm[0,1]=C X$, which is always an h-cofibration
to an object that is homotopy equivalent, and hence $C$-globally equivalent, to the zero object.
So we can -- and will -- choose the suspension functor on the $C$-global stable
homotopy category as $\Sigma=\Ho(-\sm S^1)$.

\begin{theorem}\label{thm:cofibration structure}
  Let $C$ be a topological group.
  \begin{enumerate}[\em (i)]
  \item 
    The $C$-global equivalences and the h-cofibrations make the category of orthogonal $C$-spectra
    into a cocomplete stable cofibration category.
  \item
    The localization functor $\gamma:\spec_C\to \GH_C$ preserves coproducts. In particular,
    the $C$-global stable homotopy category admits coproducts.
  \end{enumerate}
\end{theorem}
\begin{proof}
  (i) We verify the axioms (C0) -- (C6) and (C7-$\kappa$)
  for any regular cardinal $\kappa$, as stated in \cite[Definition 1.1]{szumilo:frames}.
  Most of the axioms are straightforward from the definitions:
  the $C$-global equivalences satisfy the 2-out-of-6 property (C0);
  every isomorphism is a $C$-global equivalence and an h-cofibration (C1);
  the trivial orthogonal $C$-spectrum is a zero object (C2); the unique morphism $\ast\to X$
  from a trivial orthogonal $C$-spectrum to any orthogonal $C$-spectrum is an h-cofibration (C3).
  Since h-cofibrations are precisely the morphisms with the left lifting property against
  the class of morphisms $X^{[0,1]}\to X$ that evaluate a path at 0, the class of
  h-cofibrations is stable under pushouts along arbitrary morphisms (half of C4),
  under sequential colimits (half of C6) and under arbitrary coproducts (half of C7-$\kappa$).
  The other halves of axioms (C4), (C6) and (C7-$\kappa$)
  demand that the class of h-cofibrations that are simultaneously
  $C$-global equivalences be stable under cobase change, under sequential composites,
  and under $\kappa$-small coproducts;
  Proposition \ref{prop:h-cof for C} takes care of these requirements.
  Finally, a morphism $X\to Y$ of orthogonal $C$-spectra
  factors as the composite of the mapping cylinder inclusion
  \[ X \ \to \ X\sm[0,1]_+\cup_f Y\ ,  \]
  which is an h-cofibration, followed by the projection 
  $X\sm[0,1]_+\cup_f Y\to Y$ to the `end' of the cylinder.
  This projection is a homotopy equivalence of orthogonal $C$-spectra,
  and hence a $C$-global equivalence. This verifies the factorization axiom (C5).
    
  As we explained before stating the theorem, the abstract suspension functor
  is here given by $\Ho(-\sm S^1):\GH_C\to \GH_C$.
  The loop functor $\Omega:\spec_C\to\spec_C$ also preserves $C$-global equivalences;
  so it, too, descends to a functor on the $C$-global stable homotopy category.
  The unit $X\to \Omega(X\sm S^1)$ and counit $(\Omega X)\sm S^1\to X$
  of the adjunction $(-\sm S^1,\Omega)$ are $C$-global equivalences
  by \cite[Proposition 3.1.25]{schwede:global};
  so $\Ho(-\sm S^1)$ and $\Ho(\Omega)$ are inverse autoequivalences of $\GH_C$.
  In particular, the cofibration structure  on the category of orthogonal $C$-spectra is stable.
  Part (ii) is a special case of Proposition \ref{prop:infinite coproducts}.
\end{proof}

The unreduced suspension spectrum of an orthogonal space is defined in
Construction 4.1.7 of \cite{schwede:global}.
The suspension spectrum functor is continuous,
so it extends, by functoriality, to a functor
\[       \Sigma^\infty_+ \ : \ \spc_C \ \to \ \spec_C  \]
from orthogonal $C$-spaces to orthogonal $C$-spectra.
The next proposition is the immediate generalization of \cite[Corollary 4.1.9]{schwede:global}
from global to $C$-global homotopy theory. Essentially the same proof as there
also works in our more general context, mutatis mutandis; we omit the details.

\begin{prop}\label{prop:suspension is homotopical}
  Let $C$ be a topological group. The unreduced suspension spectrum functor takes
  $C$-global equivalences of orthogonal $C$-spaces to $C$-global equivalences of
  orthogonal $C$-spectra.
\end{prop}

We let $U$ be a representation of a compact Lie group $G$.
For an orthogonal $G$-spectrum, we write $\map_*(S^U,Y)$ for the orthogonal $G$-spectrum
obtained by applying the based mapping space from the representation sphere levelwise.
We define the {\em $U$-th $G$-equivariant homotopy group} of an orthogonal $G$-spectrum $Y$ as
\begin{equation}\label{eq:repgraded_homotopy_group}
 \pi_U^G(Y)\ = \  \pi_0^G(\map_*(S^U,Y)) \ .   
\end{equation}
So for $U=\mR^n$ with trivial $G$-action, this specializes to
the $n$-th homotopy group $\pi_n^G(Y)$.
Looping by a representation sphere preserves equivariant equivalences,
compare \cite[Proposition 3.1.40]{schwede:global}.
So for every continuous homomorphism $\alpha:G\to C$ to a topological group,
the functor $\pi_U^G(\alpha^*(-))$ takes $C$-global equivalences to isomorphisms.
Hence the universal property of a localization provides a unique factorization
\[ \pi_U^G\circ \alpha^* \ : \ \GH_C \ \to \ \Ab\]
through the $C$-global stable homotopy category, for which we will
use the same notation.
Our next major goal is to show that this functor is representable by
the suspension spectrum of a specific based orthogonal $C$-space,
the `global Thom space' over the global classifying space of $\alpha:G\to C$
associated to the given representation.

\begin{construction}[$C$-global Thom spaces]
  The $C$-global classifying space of a continuous homomorphism
  $\alpha:G\to C$ from a compact Lie group to a topological group was introduced in
  Construction \ref{con:Bgl of alpha}.
  The definition involves a choice of faithful $G$-representation $V$,
  suppressed from the notation. 
  Now we also consider another $G$-representation $U$, and we define an associated `$C$-global Thom space'
  over $B_{\gl}\alpha$, namely the based orthogonal $C$-space
  \[ (B_{\gl}\alpha)^U \ = \ C_+\sm_\alpha( \bL(V,-)_+\sm S^U)\ : \ \bL \ \to \ C\bT \ .\]
  For example, if $U=\mR^n$ with trivial $G$-action, then
  $(B_{\gl}\alpha)^U$ is isomorphic to $(B_{\gl}\alpha)_+\sm S^n$.
  The next theorem (with $Z$ being the unit sphere of the representation $U$)
  shows that the resulting suspension spectrum is independent, up to $C$-global equivalence,
  of the choice of faithful $G$-representation;
  this is the justification for omitting the representation $V$ from the notation.
\end{construction}

\begin{theorem}\label{thm:rho2C-global}
  Let $\alpha:G\to C$ be a continuous homomorphism from a compact Lie group to
  a topological group. Let $V$ and $W$ be $G$-representations such that $V$ is faithful,
  and let $Z$ be a $G$-space.
  Then the morphism of orthogonal $C$-spectra
  \[ \Sigma^\infty C_+\sm_{\alpha}((\rho_{V,W})_+\sm Z^\diamond)\ : \
    \Sigma^\infty C_+\sm_\alpha (\bL(V\oplus W,-)_+\sm Z^\diamond)\ \to \
    \Sigma^\infty C_+\sm_\alpha (\bL(V,-)_+\sm Z^\diamond) \]
  is a $C$-global equivalence, where $Z^\diamond$ denotes the unreduced suspension of $Z$.
\end{theorem}
\begin{proof}
  Proposition \ref{prop:Bgl independence} shows that the two morphisms of
  orthogonal $C$-spaces
  \begin{align*}
    C\times_\alpha\rho_{V,W}\quad &: \ C\times_\alpha \bL(V\oplus W,-)\ \to\ C\times_\alpha \bL(V,-)
                                    \text{\qquad and}\\
                                    C\times_\alpha(\rho_{V,W}\times Z)\ &: \ C\times_\alpha (\bL(V\oplus W,-)\times Z)\ \to
    C\times_\alpha (\bL(V,-)\times Z)
  \end{align*}
   are $C$-global equivalences.
   So the induced morphisms of unreduced suspension spectra
   \[ 
     \Sigma^\infty_+ C\times_\alpha \rho_{V,W} \text{\qquad and\qquad}
     \Sigma^\infty_+  C\times_\alpha(\rho_{V,W}\times Z)
   \]
   are $C$-global equivalences of orthogonal $C$-spectra by Proposition \ref{prop:suspension is homotopical}.
   The spectrum $\Sigma^\infty C_+\sm_\alpha (\bL(V,-)_+\sm Z^\diamond)$ is isomorphic
   to the mapping cone of the morphism
  \[  \Sigma^\infty_+  C\times_\alpha(\bL(V,-)\times Z)\ \to \ 
     \Sigma^\infty_+ C\times_\alpha \bL(V,-) \]
   that collapses $Z$ to a point, and similarly for $V\oplus W$ instead of $V$.

   Now we let $\beta:K\to C$ be a continuous homomorphism from another compact Lie group,
   and we restrict actions along $\beta$.
   The two mapping cones give rise to long exact sequences
   of $K$-equivariant homotopy groups,
   see for example \cite[Proposition 3.1.36]{schwede:global}.
   The various morphisms derived from $\rho_{V,W}$ feature in
   a commutative diagram relating the two long exact sequences.
   The five lemma thus shows that the morphism 
   \[ \beta^*(\Sigma^\infty C_+\sm_{\alpha}((\rho_{V,W})_+\sm Z^\diamond))\]    
   is a $\upi_*$-isomorphism of orthogonal $K$-spectra. This finishes the proof.
\end{proof}
  
The reduced suspension spectrum of the $C$-global Thom space $(B_{\gl}\alpha)^U$
is an orthogonal $C$-spectrum,
and it comes with a tautological $G$-equivariant homotopy class
\begin{equation}\label{eq:tautological_class}
  e_{\alpha,U,V} \ \in \ \pi^G_U\left(\alpha^*(\Sigma^\infty (B_{\gl}\alpha)^U)\right) \ ,  
\end{equation}
defined as the class of the $G$-map
\begin{align*}
  S^{V\oplus U} \
  &\to \ S^V\sm \alpha^*(C_+\sm_\alpha(\bL(V,V)_+\sm S^U))\ = \
    \left(\alpha^*(\Sigma^\infty_+ (B_{\gl}\alpha)^U )\right)(V) \\
  (v,u)\ &\longmapsto \qquad v\sm [1\sm \Id_V\sm u]\ .
\end{align*}
The following representability theorem is a $C$-global generalization of 
\cite[Theorem 4.4.3]{schwede:global},
which is the special case where the group $C$ and the representation $U$ are trivial.

\begin{theorem}\label{thm:Bgl alpha represents}
  Let $\alpha:G\to C$ be a continuous homomorphism from a compact Lie group
  to a topological group, and let $U$ be a $G$-representation.
  \begin{enumerate}[\em (i)]
  \item
    For every orthogonal $C$-spectrum $Y$, the evaluation homomorphism
    \[  \GH_C(\Sigma^\infty (B_{\gl}\alpha)^U ,Y) \   \to  \ \pi_U^G(\alpha^* Y) \ , \quad
      f \ \longmapsto \ f_*(e_{\alpha,U,V})\]
    is an isomorphism.
  \item The orthogonal $C$-spectrum $\Sigma^\infty_+ B_{\gl}\alpha$
    is a compact object in the triangulated category $\GH_C$.
  \end{enumerate}
\end{theorem}
\begin{proof}
  (i)
  To show surjectivity we represent any given class $y\in\pi_U^G(\alpha^*Y)$
  by a continuous based $G$-map $f:S^{V\oplus W\oplus U}\to (\alpha^*Y)(V\oplus W)$,
  for some $G$-representation $W$.
  Adjoint to $f$ is a morphism of orthogonal $C$-spectra
  \[ f^\sharp \ : \ \Sigma^\infty C_+\sm_\alpha (\bL(V\oplus W,-)_+\sm S^U)\ \to \ Y\ . \]
  This morphism satisfies
  \[  f^\sharp_*(e_{\alpha,U,V\oplus W}) \ = \ y \ ,    \]
  by design.
  The morphism $\Sigma^\infty C_+\sm_\alpha((\rho_{V,W})_+\sm S^U)$ is a $C$-global equivalence
  by Theorem \ref{thm:rho2C-global}, so it becomes invertible in the $C$-global stable homotopy category.
  So we obtain a morphism in $\GH_C$
  \[ \gamma(f^\sharp)\circ \gamma(\Sigma^\infty C_+\sm_\alpha((\rho_{V,W})_+\sm S^U))^{-1}\ : \
    \Sigma^\infty (B_{\gl}\alpha)^U = \Sigma^\infty C_+\sm_\alpha (\bL(V,-)_+\sm S^U) \ \to \ Y \ . \]
  The morphism $\Sigma^\infty C_+\sm_\alpha((\rho_{V,W})_+\sm S^U)$ sends
  the tautological class $e_{\alpha,U,V\oplus W}$ to the tautological class $e_{\alpha,U,V}$.
  So we deduce the relation
\[ 
 \left( \gamma(f^\sharp)\circ \gamma(\Sigma^\infty C_+\sm_\alpha((\rho_{V,W})_+\sm S^U))^{-1}\right)_*(e_{\alpha,U,V})\
   = \ \gamma(f^\sharp)_*(e_{\alpha,U,V\oplus W})\ = \ y\ .
 \]
  This proves surjectivity of the evaluation homomorphism.

  For injectivity we consider a morphism $\omega:\Sigma^\infty (B_{\gl}\alpha)^U\to Y$
  in $\GH_C$ such that $\omega_*(e_{\alpha,U,V})=0$.
  The calculus of fractions lets us write $\omega = \gamma(s)^{-1}\circ\gamma(f)$
  for two morphisms of orthogonal $C$-spectra $f:\Sigma^\infty (B_{\gl}\alpha)^U\to Z$
  and $s:Y\to Z$ such that $s$ is a $C$-global equivalence.
  Then $\gamma(f)_*(e_{\alpha,U,V})=\gamma(s)_*(\omega_*(e_{\alpha,U,V}))=0$.
  So we can assume without loss of generality that the original morphism is
  of the form $\omega=\gamma(f)$ for a morphism of orthogonal $C$-spectra
  $f:\Sigma^\infty (B_{\gl}\alpha)^U\to Y$. 
  The class $\gamma(f)_*(e_{\alpha,U,V})$ is represented by the $G$-map
  \[  S^{V\oplus U} \ \to \ Y(V)\ , \quad (v,u)\ \longmapsto \ f(V)(v\sm[1\sm\Id_V\sm u])\ .\]
  Since this class is trivial, there is a $G$-representation $W$ such that the stabilization
  \begin{equation}\label{eq:some other composite}    
    S^{V\oplus W\oplus U} \ \to \  Y(V\oplus W)\ , \quad
    (v,w,u)\ \longmapsto \ f(V\oplus W)((v,w)\sm[1\sm i_V\sm u])
     \end{equation}
  is $G$-equivariantly based null-homotopic.
  We choose a null-homotopy that witnesses this fact 
  and adjoint it to a morphism of orthogonal $C$-spectra
  \[  H \ : \ 
        \Sigma^\infty C_+\sm_\alpha (\bL(V\oplus W,-)_+\sm S^U)\sm [0,1]
 \ \to \ Y \ ; \]
  here the unit interval $[0,1]$ is based at 0, and the restriction of $H$ to the point 1
  is adjoint to \eqref{eq:some other composite}.
  We arrive at a commutative diagram in $\spec_C$:
  \[ \xymatrix@C=15mm{ \Sigma^\infty C_+\sm_\alpha (\bL(V,-)_+\sm S^U) \ar@/^1pc/[dr]^-f & \\
      \Sigma^\infty C_+\sm_\alpha (\bL(V\oplus W,-)_+\sm S^U) \ar[r] \ar[d]_{-\sm 1}
      \ar[u]^{\Sigma^\infty C_+\sm_\alpha((\rho_{V,W})_+\sm S^U)}_\sim & Y\\
       \Sigma^\infty C_+\sm_\alpha (\bL(V,-)_+\sm S^U)\sm [0,1] \ar@/_1pc/[ur]_-H& }\]
  Since the upper left morphism becomes invertible in the $C$-global stable homotopy category,
  and because the lower left spectrum becomes a zero object in $\GH_C$, this proves that the
  image of $f$ in $\GH_C$ is the zero morphism.

  (ii)
  By Theorem \ref{thm:cofibration structure} (ii),
  the wedge of any family of orthogonal $C$-spectra is a coproduct in $\GH_C$. 
  The vertical maps in the commutative square
  \[\xymatrix{  {\bigoplus_{i\in I}} \GH_C( \Sigma^\infty_+ B_{\gl}\alpha,\, X_i ) \ar[r]\ar[d] &
      \GH_C( \Sigma^\infty_+ B_{\gl}\alpha,\, \bigoplus_{i\in I}X_i ) \ar[d] \\
      {\bigoplus_{i\in I}} \pi^G_0(\alpha^*(X_i)) \ar[r] &
      \pi_0^G\left( \bigvee_{i\in I}\alpha^*(X_i)\right) }\]
  are evaluation at the tautological class,
  and hence isomorphisms by part (i).
  The lower horizontal map is an isomorphism, see for example \cite[Corollary 3.1.37]{schwede:global};
  so the upper horizontal map is an isomorphism, too.
  This shows that $\Sigma^\infty_+ B_{\gl}\alpha$ is compact 
  as an object of the triangulated category $\GH_C$.
\end{proof}

In the special case where the $G$-representation $U$ is trivial,
Theorem \ref{thm:Bgl alpha represents} (i) says that
the unreduced suspension spectrum $\Sigma^\infty_+ B_{\gl}\alpha$
represents the functor $\pi_0^G\circ\alpha^*:\GH_C\to\Ab$.
So if $Y$ is an orthogonal $C$-spectrum such that the group
$\GH_C( \Sigma^\infty_+ (B_{\gl}\alpha)[k],\, Y)$ is trivial for every
continuous homomorphism $\alpha:G\to C$ from a compact Lie group $G$
and all integers $k$, then $Y$ is $C$-globally equivalent
to the trivial orthogonal $C$-spectrum.
So $Y$ is a zero object in $\GH_C$. This proves the following result.

\begin{cor}\label{cor:compactly generated}
  Let $C$ be a topological group.
  As $\alpha:G\to C$ varies over a set of representatives of the isomorphism classes
  of continuous homomorphisms from compact Lie groups to $C$,
  the suspension spectra $\Sigma^\infty_+ B_{\gl}\alpha$ form a set of compact weak generators 
  for the triangulated $C$-global stable homotopy category $\GH_C$.
  In particular, the $C$-global stable homotopy category is compactly generated.
\end{cor}

\begin{rk}[The $\infty$-category of $C$-global spectra]
  We let $C$ be a topological group.
  The triangulated $C$-global stable homotopy category is only the shadow of a more
  refined structure, namely an underlying compactly generated stable $\infty$-category.
  We define the {\em $\infty$-category of $C$-global spectra} as the $\infty$-categorical localization of
  the 1-category of orthogonal $C$-spectra at the class of $C$-global equivalences.
  For example, the quasicategory of frames
  in the sense of Szumi{\l}o \cite[Section 2]{szumilo:frames} associated to the cofibration structure
  of Theorem \ref{thm:cofibration structure} on the category of orthogonal $C$-spectra
  is a particular construction.  
  This quasicategory is cocomplete by \cite[Theorem 2.3]{szumilo:frames}.
  Stability of an $\infty$-category is detected by the homotopy category,
  see \cite[Corollary 1.4.2.27]{lurie:higher_algebra};
  so the $\infty$-category of $C$-global spectra is stable.
  Similarly, the property of a stable $\infty$-category to be compactly generated
  is detected by the homotopy category, compare \cite[Remark 1.4.4.3]{lurie:higher_algebra}.
  So the stable $\infty$-category of $C$-global spectra is compactly generated;
  in particular, this stable $\infty$-category is also presentable and complete.

  Suppose that $C$ is a {\em Lie group} (not necessarily compact);
  for example, $C$ could be an infinite discrete group.
  Then the $C$-global equivalences also take part in several model category structures on
  the category of orthogonal $C$-spectra.
  The essential ingredients for the construction of the model structure
  are a synthesis of the arguments needed in the special case of the trivial group, i.e.,
  the global model structure of orthogonal spectra \cite[Theorem 4.3.18]{schwede:global},
  and the arguments used in \cite[Section 1.2]{dhlps} to set up the {\em proper stable homotopy theory}
  for non-compact Lie groups. Since we don't need any model structures
  for the purposes of this paper, I won't dwell on this any further.
\end{rk}

\subsection{Global classifying spaces by complex and quaternion isometric embeddings}

As we explain in detail in Construction \ref{con:Bgl of alpha},
the global classifying $C$-space of a continuous homomorphism $\alpha:G\to C$
from a compact Lie group uses real Stiefel manifolds, i.e., spaces of $\mR$-linear isometric
embeddings from a faithful orthogonal $G$-representation.
In our applications to stable splittings of $\bU/m$ and $\bSp/m$,
the naturally occurring objects are {\em complex} and {\em quaternionic} Stiefel manifolds.
We will thus need to know that we can also use spaces of $\mC$-linear or $\mH$-linear isometric
embeddings from a faithful unitary or symplectic $G$-representation
to define global classifying spaces.
Moreover, we want to keep track of the natural symmetries, parameterized by the groups
of $\mR$-algebra automorphisms of $\mC$ and $\mH$.
In this subsection we explain the connection in detail.
A precursor of the results in this subsection already occurs in \cite[Proposition 1.3.11]{schwede:global},
which treats the complex case without any mentioning of the Galois group $G(\mC)$.
The key ingredients are already present in the proof of the precursor;
our main work here is to carefully adapt the arguments from $\mC$ to $\mH$,
while incorporating an augmentation to the Galois group.

\medskip

In the rest of this section, we let $\mK$ be one of the skew-fields $\mC$ or $\mH$.
The arguments also work for $\mK=\mR$, but then they are either tautological, or already well-known
and explicitly stated in \cite{schwede:global}.
As before we write $G(\mK)=\Aut_\mR(\mK)$ for the `Galois group' of $\mK$,
i.e., the compact Lie group of $\mR$-algebra automorphisms.
Then $G(\mC)$ is discrete of order 2, and $G(\mH)$ is abstractly isomorphic to $S O(3)$.

\begin{defn}\label{def:G,epsilon-rep}
  We let $G$ be a compact Lie group and $\epsilon:G\to G(\mK)$ a continuous homomorphism.
  A {\em $(G,\epsilon)$-representation} is a $\mK$-inner product space $W$
  endowed with a continuous $\mR$-linear $G$-action such that
  \[ 
    g\cdot (x \lambda)\ = \ (g x)\cdot \epsilon(g)(\lambda)  \text{\qquad and\qquad} 
    [g x, g y]\ = \ \epsilon(g)( [x,y] )
  \]
  for all  $g\in G$, $\lambda\in \mK$ and $x,y\in W$. 
\end{defn}

\begin{rk}[Extended isometry group]
  The {\em extended isometry group} $\tilde I(W)$ of a $\mK$-inner product space $W$
  is the group of pairs $(A,\tau)$ consisting of an $\mR$-linear automorphism
  $A:W\to W$ and a Galois automorphism $\tau\in G(\mK)$ such that
  \begin{itemize}
  \item the morphism $A$ is $\tau$-semilinear, i.e.,
    $A(x\lambda)=A(x)\cdot\tau(\lambda)$ for all $x\in W$ and $\lambda\in\mK$, and
  \item the relation
    \begin{equation}\label{eq:twisted isometry}
      [A x, A y]\ = \ \tau[x,y] 
    \end{equation}
    holds for all $x,y\in W$.
  \end{itemize}
  If $W$ is non-zero, then the Galois automorphism $\tau$ is actually determined by $A$;
  nevertheless, it is convenient to explicitly keep track of $\tau$.
  Composition in $\tilde I(W)$ is componentwise, i.e,
  $(A,\tau)\cdot (A',\tau')=(A A',\tau \tau')$.
  The relation \eqref{eq:twisted isometry}
  in particular implies that for $(A,\tau)\in \tilde I(W)$,
  the map $A$ is an isometry of the  underlying euclidean vector space $u W$,
  i.e., the underlying $\mR$-vector space of $W$ equipped with the 
  euclidean inner product $\td{x,y}=\Re[x,y]$.
  So $\tilde I(W)$ is a closed subgroup of $O(u W)\times G(\mK)$,
  and hence a compact Lie group.
  The projection to the second factor $p:\tilde I(W)\to G(\mK)$
  is a continuous epimorphism whose kernel is isomorphic to $I(W)$.
  Moreover, $\tilde I(W)$ is isomorphic to a semidirect product $I(W)\rtimes G(\mK)$,
  but not in a natural way.
    
  The reader might want to convince themself that a $(G,\epsilon)$-representation
  with underlying inner product space $W$ can equivalently be specified by
  a continuous homomorphism $\rho: G\to \tilde I(W)$
  to the extended isometry group that covers the augmentations,
  i.e., such that $p\circ \rho=\epsilon$.
\end{rk}

For $\mK=\mC$, every $(G,\epsilon)$-representation
in particular has an underlying orthogonal representation
of the compact Lie group $G$;
and it has an underlying unitary representation of the kernel of $\epsilon:G\to G(\mC)$.
Compact Lie groups augmented to the Galois group of $\mC$ over $\mR$
and their twisted representations
are studied as `augmented compact Lie groups' by Karoubi \cite{karoubi:clifford},
and prior to that (in special cases) by Atiyah and Segal \cite[Section 6]{atiyah-segal:completion}.
For $\mK=\mH$, every $(G,\epsilon)$-representation
in particular has an underlying orthogonal representation
of the compact Lie group $G$;
and it has an underlying symplectic representation of the kernel of $\epsilon:G\to G(\mH)$.
I am not aware of a reference where compact Lie groups augmented to $G(\mH)$ and their
`twisted' representations are systematically investigated.

\medskip

We write $V_\mK=V\tensor_\mR\mK$ for the scalar extension from $\mR$ to $\mK$
of a euclidean inner product space $V$, with $\mK$-inner product $[-,-]$
obtained from the euclidean inner product $\td{-,-}$ on $V$ by
\[ [x\tensor\lambda,y\tensor\mu]\ = \ \bar\lambda \cdot \td{x,y}\cdot \mu      \]
for $x,y\in V$ and $\lambda,\mu\in\mK$.
The {\em underlying euclidean inner product space} $u W$
of a $\mK$-inner product space $W$ is the underlying $\mR$-vector space endowed with the
euclidean inner product
\[ \td{x,y}\ = \ \Re[x,y]\ , \]
the real part of the $\mK$-valued inner product.

\begin{construction}
  We let $\mK$ be $\mC$ or $\mH$, and we let $W$ be a $\mK$-inner product space.
  For several arguments we shall need a specific natural $\mK$-linear isometric embedding
  \begin{align}\label{eq:define zeta}
    \zeta \ &: \ W \ \to \ (u W)_\mK \ = \ (u W)\tensor_\mR\mK \ .
  \end{align}
  To define it, we separate the two cases.
  For $\mK=\mC$, the map $\zeta$ is defined as
  \[ 
    \zeta \ : \ W \ \to \ (u W)_\mC \ = \ (u W)\tensor_\mR\mC  \text{\qquad by\qquad}
    \zeta(x) \ = \ 1/\sqrt{2}\cdot(x\tensor 1 - x i \tensor i) \ .
  \]
  For $\mK=\mH$, it is given by
  \begin{align*}
    \zeta \ : \ W \ &\to \ (u W)_\mH \ = \ (u W)\tensor_\mR\mH\ ,\\
    \zeta(x) \ &= \ 1/2\cdot(x\tensor 1 - x i \tensor i - x j\tensor j -x k\tensor k)\ .
  \end{align*}
\end{construction}

\begin{prop}
  Let $\epsilon:G\to G(\mK)$ be a continuous homomorphism,
  where $\mK$ is $\mC$ or $\mH$. Let $W$ be a $(G,\epsilon)$-representation.
  Then the map $\zeta$ is a $\mK$-linear isometric embedding
  that satisfies
  \begin{equation} \label{eq:j_twisted_equivariant}
    \zeta\circ l_g \ = \ (l_g\tensor \epsilon(g))\circ\zeta     
  \end{equation}
  for all $g\in G$, where $l_g:W\to W$ is left multiplication by $g$.
\end{prop}
\begin{proof}
  We give the argument for $\mK=\mH$. The proof in the complex case is similar, but easier.
  The map $\zeta$ is clearly $\mR$-linear, and it commutes with right multiplication
  by the quaternion scalars $i$ and $j$; so $\zeta$ is in fact $\mH$-linear.
  To show that the map $\zeta$ is isometric we observe that  
  \[ 
    \left[x \lambda\tensor\bar\lambda,x \nu \tensor\bar\nu\right]
    = \ \lambda\cdot \td{x \lambda,x \nu}\cdot\bar\nu\
    = \ \lambda\cdot \Re\left(\bar\lambda [x,x]\nu\right)\cdot\bar\nu\
    = \ [x,x]\cdot \Re\left(\bar\lambda \nu\right)\cdot\lambda \bar\nu
  \]
  for all $x\in W$ and $\lambda,\nu\in\mH$.
  For $\lambda,\nu\in\{1,i,j,k\}$ with $\lambda\ne\nu$ we have $\Re\left(\bar\lambda \nu\right)=0$.
  So
  \begin{align*}
    [\zeta(x),\zeta(x)]\
    = \  1/4\cdot [x,x]\cdot \sum_{\lambda\in\{1,i,j,k\}} \Re(\bar\lambda \lambda)\cdot\lambda\bar\lambda\
    = \ [x,x]\ .
  \end{align*}
  It remains to prove the relation \eqref{eq:j_twisted_equivariant}.
  The tautological action of $G(\mH)$ on $\mH$ is isometric for the euclidean inner product
  $\td{x,y}=\Re(\bar x y)$.
  Since $(1,i,j,k)$ is an orthonormal $\mR$-basis for this inner product,
  the element $1\tensor 1 + i\tensor i + j\tensor j + k\tensor k$ is fixed
  under the diagonal $G(\mH)$-action (and could have been specified by any other orthonormal basis).
  Since $1\tensor 1$ is evidently $G(\mH)$-fixed, too, the element
  \[ 1\tensor 1 - i\tensor i - j\tensor j - k\tensor k \]
  is fixed under the diagonal $G(\mH)$-action. So we conclude that
  \begin{align*}
    \zeta(g x)\
    &= \  1/2\cdot(g x)\cdot (1\tensor 1 - i \tensor i -  j\tensor j - k\tensor k) \\
    &= \  1/2\cdot(g x)\cdot (1\tensor 1 - \epsilon(g)(i)\tensor \epsilon(g)(i) - \epsilon(g)(j)\tensor\epsilon(g)(j)- \epsilon(g)(k)\tensor\epsilon(g)(k)) \\
    &= \  1/2\cdot((g x)\tensor 1 - g (x i) \tensor \epsilon(g)(i) - g(x j)\tensor\epsilon(g)(j)- g(x k)\tensor\epsilon(g)(k)) \
    = \  (l_g\tensor\epsilon(g))(\zeta(x))
  \end{align*}
  for all $g\in G$ and $x\in W$.
\end{proof}

A basic fact in the global homotopy theory of orthogonal spaces is that
for two compact Lie groups $K$ and $G$, and for any faithful $G$-representation $V$,
the infinite Stiefel manifold $\bL(V,\Uc_K)$ is a universal $(K\times G)$-space
for the family of graph subgroups, compare \cite[Proposition 1.1.26]{schwede:global}.
This property relies on the fact that for every finite-dimensional $K$-representation $W$,
the space $\bL^K(W,\Uc_K)$ of $K$-equivariant $\mR$-linear isometric embeddings is contractible, compare
\cite[Propositions 1.1.21]{schwede:global}.
Both statement have suitable analogs in the complex and quaternionic situations,
with easily adapted proofs.
Since I do not know of references,
I provide proper statements and proofs now.

\begin{prop}\label{prop:quaternion contractible}
  Let  $\Uc_K$ be a complete universe of a compact Lie group $K$.
  Let $\beta:K\to G(\mK)$ be a continuous homomorphism,
  where $\mK$ is either $\mC$ or $\mH$.
  Let $K$ act on $\Uc_K^\mK=\Uc_K\tensor_\mR\mK$ by
  \[ k\cdot(x\tensor\lambda)\ = \ (k x)\tensor\beta(k)(\lambda)\ . \]
  Then for every $(K,\beta)$-representation $W$, the space
  $\bL^{\mK,K}(W,\Uc_K^\mK)$ of $\mK$-linear $K$-equivariant isometric embeddings
  is weakly contractible.
\end{prop}
\begin{proof}
  We let $U$ be any $(K,\beta)$-representation, possibly of countably infinite dimension.
  Then the homotopy
  \[ H \ : \ [0,1]\times \bL^{\mK,K}(W,U)\ \to \ \bL^{\mK,K}(W,U\oplus W) \ ,\quad
  H(t,\varphi)(w)\ = \ (\sqrt{1-t^2} \cdot \varphi(w),\, t\cdot w) \]
  witnesses the fact that the map
  \[ i_1\circ - \ : \ \bL^{\mK,K}(W,U)\ \to \ \bL^{\mK,K}(W,U\oplus W) \]
  (post-composition with $i_1:U\to U\oplus W)$ is homotopic to a constant map.
  The space
  \[ \bL^{\mK,K}(W,U\oplus W^\infty) \ = \ \colim_{n\geq 0}\, \bL^{\mK,K}(W, U\oplus W^n)  \]
  is the colimit along the post-composition maps with the direct sum embeddings
  $U\oplus W^n\to U\oplus W^{n+1}$.
  Every map in the colimit system is homotopic
  to a constant map, by the previous paragraph. 
  Since the maps are also closed embeddings, the colimit is weakly contractible.

  Now we can prove the proposition.
  Because $\Uc_K$ is a complete $K$-universe, there is a $K$-equivariant $\mR$-linear isometry
  $\Uc_K\iso \Vc\oplus (u W)^\infty$ for some orthogonal $K$-representation $\Vc$
  (typically infinite dimensional).
  The $\mK$-linear $K$-equivariant isometric embedding $\zeta:W\to (u W)\tensor_\mR\mK$
  provides an isomorphism of $(K,\beta)$-representations 
  \[ (u W)\tensor_\mR \mK \ \iso \ C \oplus W\ , \]
  where $C$ is the orthogonal complement of the image of $\zeta$.
  So $\Uc_K^\mK$ is isomorphic to
  \[ \left( \Vc\oplus (u W)^\infty\right)\tensor_\mR \mK \ \iso \
     (\Vc\tensor_\mR\mK) \oplus C^\infty \oplus W^\infty\ .
  \]
  The space $\bL^{\mK,K}(W,\Uc_K^\mK)$ is thus weakly contractible by the first paragraph.
\end{proof}

\begin{construction}
  We let $K$ and $G$ be compact Lie groups,
  and we let $\beta:K\to G(\mK)$ and $\epsilon:G\to G(\mK)$ be two continuous homomorphisms.
  We write
  \[ K\times_{G(\mK)} G \ = \
    \{(k,g)\in K\times G\ : \ \beta(k)=\epsilon(g) \} \]
  for the fiber product over $G(\mK)$.
  We let $W$ be a $(G,\epsilon)$-representation, and we let $\Uc_K$ be a complete $K$-universe.
  As before we write $\Uc_K^\mK =\Uc_K\tensor_\mR \mK$
  for the scalar extension. The group $K\times_{G(\mK)} G$ then acts on the space
  $\bL^\mK(W,\Uc_K^\mK)$ of $\mK$-linear isometric embeddings by
  \begin{equation}\label{eq:k,g_act}
    (k,g)\cdot \varphi \ = \ (l_k\tensor\beta(k))\circ\varphi\circ l_g^{-1}\ ,   
  \end{equation}
  where $l_k:\Uc_K\to\Uc_K$ and $l_g:W\to W$ are the translation maps.
  We exploit here that for $(k,g)\in K\times_{G(\mK)} G$, both maps 
  $l_k\tensor\beta(k)$ and $l_g$ are semilinear for the same $\mR$-algebra automorphism of $\mK$,
  so the composite \eqref{eq:k,g_act} is again $\mK$-linear.
  So altogether the assignment \eqref{eq:k,g_act} defines a continuous action of the compact
  Lie group $K\times_{G(\mK)} G$ on the space $\bL^\mK(W,\Uc_K^\mK)$.
\end{construction}

\begin{prop}\label{prop:universal space}
  Let $\mK$ be $\mC$ or $\mH$.
  Let $\epsilon:G\to G(\mK)$ be a continuous epimorphism from a compact Lie group
  with kernel $G_0$.
  Let $W$ be a faithful $(G,\epsilon)$-representation.
  Then for every continuous homomorphism $\beta:K\to G(\mK)$ from another compact Lie group
  and for every complete $K$-universe $\Uc_K$, the $(K\times_{G(\mK)} G)$-space
  $\bL^\mK(W,\Uc_K^\mK)$ is a universal $(K\times_{G(\mK)} G)$-space for the family of those
  closed subgroups that intersect $1\times G_0$ trivially.
\end{prop}
\begin{proof}
  Essentially the same argument as in the real situation in \cite[Proposition 1.1.19]{schwede:global},
  which relies on Illman's equivariant triangulation theorem
  for smooth actions of compact Lie groups \cite{illman},
  shows that $\bL^\mK(W,\Uc_K^\mK)$ is cofibrant as a $(K\times_{G(\mK)} G)$-space.
  Since $G_0$ acts faithfully on $W$, the action on $\bL^\mK(W,\Uc_K^\mK)$
  by precomposition is free.
  So every subgroup of $K\times_{G(\mK)} G$
  with fixed points on $\bL^\mK(W,\Uc_K^\mK)$
  must intersect the group $1\times G_0$ trivially.
  Conversely, we now consider a closed subgroup $\Delta$ of
  $K\times_{G(\mK)} G$ such that $\Delta\cap(1\times G_0)=e$.
  Then $\Delta$ is the graph of some continuous homomorphism
  \[ \alpha \ : \ L \ \to \ G\]
  defined on some closed subgroup $L$ of $K$, such that $\epsilon\circ\alpha=\beta|_L$.
  Hence 
   \[ \left( \bL^\mK(W,\Uc_K^\mK) \right)^{\Delta} \ = \
   \bL^{\mK,L}(\alpha^*(W),\Uc_K^\mK) \ , \]
  the space of $\mK$-linear and $L$-equivariant isometric embeddings from $\alpha^*(W)$
  into $\Uc_K^\mK$,
  with $L$-action on the target by
  \[ l\cdot(x\tensor\lambda)\ = \ (l x)\tensor\beta(l)(\lambda)\ . \]
  Because $\Uc_K$ is a complete $K$-universe,
  the underlying $L$-representation is a complete $L$-universe. So the space
  $\bL^{\mK,L}(\alpha^*(W),\Uc_K^\mK)$ is weakly contractible by Proposition \ref{prop:quaternion contractible}.
  This completes the proof.
\end{proof}

\begin{construction}
  Let $\epsilon:G\to G(\mK)$ be a continuous epimorphism from a compact Lie group.
  Let $W$ be a $(G,\epsilon)$-representation.
  For a euclidean inner product space $V$,
  we define a $G$-action on the Stiefel manifold $\bL^\mK(W,V_\mK)$  by
  \[  {^g \varphi} \ = \ (V\tensor\epsilon(g))\circ \varphi\circ l_g^{-1} \ , \]
    where $l_g:W\to W$ is translation by $g\in G$.
  The orbit space
  \[ \bL^\mK(W,V_\mK)/G_0  \]
  by the action of the closed normal subgroup $G_0=\ker(\epsilon)$
  inherits a residual $G(\mK)$-action; for varying $V$, this defines an
  orthogonal $G(\mK)$-space  $\bL^\mK(W,(-)_\mK)/G_0$.
  For example, for the augmentation $\epsilon_k:\tilde I(k)=I(k)\rtimes G(\mK)\to G(\mK)$
  of the extended isometry group of $\mK^k$
  and the tautological $(\tilde I(k),\epsilon_k)$-representation on $\mK^k$,
  the construction specializes to the Grassmannian $\bGr_k^\mK$ from Example \ref{eg:global Thom}.
\end{construction}

We let $\epsilon:G\to G(\mK)$ be a continuous epimorphism,
and we continue to write $G_0=\ker(\epsilon)$.
We will now argue that for every {\em faithful} $(G,\epsilon)$-representation $W$,
the orthogonal $G(\mK)$-space $\bL^\mK(W,(-)_\mK)/G_0$
is a $G(\mK)$-global classifying space,
in the sense of Construction \ref{con:Bgl of alpha},
of the epimorphism $\epsilon:G\to G(\mK)$.
Because $G$ acts faithfully on $W$, we can use the underlying orthogonal $G$-representation $u W$
to construct a $G(\mK)$-global classifying space
\[ B_{\gl}\epsilon\ = \ G(\mK)\times_{\epsilon}\bL(u W,-) \ = \ \bL(u W,-)/G_0 \ .\]
The issue now is to compare this orthogonal $G(\mK)$-space to $\bL^\mK(W,(-)_\mK)/G_0$.
To this end we recall from \eqref{eq:define zeta}
the $\mK$-linear isometric embedding
\[  \zeta \ : \ W \ \to \ (u W)\tensor_\mR\mK \ = \ (u W)_\mK  \ .\]
The enriched Yoneda lemma provides a unique morphism of orthogonal spaces
\[  \zeta^\flat\ : \  \bL(u W,-)\ \to \ \bL^\mK(W,(-)_\mK) \]
whose value at $u W$ takes the identity to $\zeta$.
Since the map $\zeta$ is $G_0$-equivariant, so is the morphism $\zeta^\flat$.
For any $G$-space $Z$, we can thus form the morphism
of orthogonal spaces
\[  \zeta^\flat\times_{G_0}Z \ : \  \bL(u W ,-)\times_{G_0}Z\ \to \
  \bL^\mK(W,(-)_\mK)\times_{G_0}Z\ . \]
The relation \eqref{eq:j_twisted_equivariant} implies that this morphism
is $G(\mK)$-equivariant, and hence a morphism of orthogonal $G(\mK)$-spaces.

\begin{theorem}\label{thm:zeta_is_A-global}
  Let $\epsilon:G\to G(\mK)$ be a continuous epimorphism from a compact Lie group,
  where $\mK$ is either $\mC$ or $\mH$.
  Let $W$ be a faithful $(G,\epsilon)$-representation, and let $Z$ be a $G$-space.
  \begin{enumerate}[\em (i)]
  \item
    The morphism of orthogonal $G(\mK)$-spaces
    $\zeta^\flat\times_{G_0}Z$  is a $G(\mK)$-global equivalence.
  \item 
    The morphism of orthogonal $G(\mK)$-spectra
    \[  \Sigma^\infty \zeta^\flat_+\sm_{G_0} Z^\diamond \ : \
      \Sigma^\infty \bL(u W ,-)_+\sm_{G_0} Z^\diamond\ \to \
      \Sigma^\infty \bL^\mK(W,(-)_\mK)_+\sm_{G_0}Z^\diamond \]
    is a $G(\mK)$-global equivalence, where $Z^\diamond$ denotes the unreduced suspension.    
  \end{enumerate}
\end{theorem}
\begin{proof}
  (i) We consider a continuous homomorphism $\beta:K\to G(\mK)$ from another compact Lie group.
  Then  $\bL^\mK(W,\Uc_K^\mK)$ is a universal $(K\times_{G(\mK)} G)$-space
  for the family of those  closed subgroups that intersect $1\times G_0$ trivially,
  by Proposition \ref{prop:universal space}.
  Also, the $(K\times G)$-space $\bL(u W, \Uc_K)$
  is a universal space for the family of graph subgroups,
  i.e., those closed subgroups $\Delta$ of $K\times G$
  such that $\Delta\cap(1\times G)=e$, compare \cite[Proposition 1.1.26]{schwede:global}.
  For a subgroup of $K\times_{G(\mK)} G$,
  the intersections with $1\times G$ and $1\times G_0$
  coincide; so the underlying $(K\times_{G(\mK)} G)$-space of
  $\bL(u W, \Uc_K)$ is a universal space for
  the family of those closed  subgroups $\Delta$ such that $\Delta\cap(1\times G_0)=e$.
  The continuous map
  \[ \zeta^\flat(\Uc_K)\ : \ \bL(u W,\Uc_K)\ \to \ \bL^\mK(W,\Uc_K^\mK)\ , \quad
    \varphi \ \longmapsto \ (\varphi\tensor\mK)\circ \zeta\]
  is $(K\times_{G(\mK)} G)$-equivariant;
  since source and target are universal spaces for the same family of subgroups,
  the map is a $(K\times_{G(\mK)} G)$-equivariant homotopy equivalence.
  The induced map on $G_0$-orbit spaces
  \[ \zeta^\flat(\Uc_K)\times_{G_0}Z\ : \  \bL(u W,\Uc_K)\times_{G_0}Z\ \to \
    \bL^\mK(W, \Uc_K^\mK)\times_{G_0} Z \]
  is thus an equivariant homotopy equivalence for the action
  of the group $(K\times_{G(\mK)} G)/G_0$.
  The projection to the first factor identifies the group
  $(K\times_{G(\mK)} G)/G_0$ with $K$;
  and the  $(K\times_{G(\mK)} G)/G_0$-fixed points coincide with the fixed points
  of the graph of $\beta$ on $(\zeta^\flat\times_{G_0}Z)(\Uc_K)$.
  So we have verified that the map $((\zeta^\flat\times_{G_0}Z)(\Uc_K))^{\Gamma(\beta)}$
  is a weak equivalence for every continuous homomorphism $\beta:K\to G(\mK)$.
  Since the orthogonal spaces underlying source and target of
  $\zeta^\flat\times_{G_0}Z$ are closed,
  the morphism $\zeta^\flat\times_{G_0}Z$ is a $G(\mK)$-global equivalence
  by Proposition \ref{prop:global eq for closed}.

  (ii) The same argument to compare the long exact homotopy group sequences of
  mapping cones as in Theorem \ref{thm:rho2C-global} applies here, but now the role of
  Proposition \ref{prop:Bgl independence} is played by the first part of this theorem.  
\end{proof}

We record a special case of Theorem \ref{thm:zeta_is_A-global} that is
particularly relevant for the application to global stable splittings of Stiefel manifolds.
We write
\[ \epsilon_k\ :\ \tilde I(k) =  I(k)\rtimes G(\mK) = \bL^\mK(\mK^k,\mK^k)\rtimes G(\mK)\ \to\  G(\mK) \]
for the augmentation of the extended isometry group of $\mK^k$,
i.e., the projection to the second factor.
The tautological action makes $\mK^k$  a faithful $(\tilde I(k),\epsilon_k)$-representation
in the sense of Definition \ref{def:G,epsilon-rep}.
So for every orthogonal $\tilde I(k)$-representation $U$,
Theorem \ref{thm:zeta_is_A-global} (ii) provides a $G(\mK)$-global equivalence
of orthogonal $G(\mK)$-spectra
\[  \Sigma^\infty \zeta^\flat_+\sm_{I(k)}S^U \ : \    \Sigma^\infty (B_{\gl}\epsilon_k)^U\
  \xra{\ \simeq\ } \   \Sigma^\infty (\bGr_k^\mK)^U \ . \]
We write 
\begin{equation}\label{eq:define_e_k,U}
  e_{k,U}\ = \ (\Sigma^\infty \zeta^\flat_+\sm_{I(k)}S^U)_*(e_{\epsilon_k,U,u\mK^k})
  \ \in \ \pi_U^{\tilde I(k)}(\epsilon_k^*(\Sigma^\infty (\bGr_k^\mK)^U))  
\end{equation}
for the image under this $G(\mK)$-global equivalence
of the tautological class defined in \eqref{eq:tautological_class}.
Because $\Sigma^\infty \zeta^\flat_+\sm_{I(k)}S^U$ is a $G(\mK)$-global equivalence,
Theorem \ref{thm:Bgl alpha represents} (i) implies the following result.

\begin{cor}\label{cor:J_is_A-global}
  Let $\mK$ be $\mC$ or $\mH$.
  Let $U$ be an orthogonal representation of the extended isometry group $\tilde I(k)$.
  Then for every orthogonal $G(\mK)$-spectrum $Y$, the evaluation map
  \[  \GH_{G(\mK)}(\Sigma^\infty (\bGr_k^\mK)^U ,Y) \   \to  \
    \pi_U^{\tilde I(k)}(\epsilon_k^*(Y)) \ , \quad
      f \ \longmapsto \ f_*(e_{k,U})\]
    is an isomorphism.
\end{cor}

\section{Some linear algebra}
\label{sec:linear algebra}

The purpose of this appendix is to provide detailed proofs of the linear algebra facts used in
the main part of this paper.
We set things up so that everything works simultaneously over the fields $\mR$ and $\mC$ and over
the skew-field $\mH$ of quaternions.
Linear algebra over the quaternions comes with some additional caveats,
many due to the non-commutativity of the multiplication.
So I felt the need to justify that the relevant arguments can indeed be adapted
to the quaternion situation.
I make absolutely no claim to originality for anything in this appendix.

\medskip

Throughout this appendix, we will let $\mK$ denote one of the three skew-fields $\mR$, $\mC$ or $\mH$.
A $\mK$-vector space is a right $\mK$-module.
Given two  $\mK$-vector spaces $V$ and $W$, we write
$\Hom_\mK(V,W)$ for the $\mR$-vector space of right $\mK$-linear maps.
For $\mK=\mC$ one can make $\Hom_\mC(V,W)$ a $\mC$-vector space
by pointwise scalar multiplication, but we will not use this structure.
Due to the non-commutativity of the quaternions, for $\mK=\mH$
there is no natural way to endow $\Hom_\mH(V,W)$ with an $\mH$-action.

A {\em $\mK$-inner product space} is a finite-dimensional $\mK$-vector space
equipped with a sesquilinear, hermitian and positive-definite $\mK$-valued inner product $[-,-]$,
see Definition \ref {def:mK-inner}.
An example of an inner product space is $\mK^k$ with the standard inner product 
\[ [x,y]\ = \  \bar x_1\cdot y_1 + \dots + \bar x_k\cdot y_k \ . \]
Every $\mK$-inner product space $W$ admits
an orthonormal basis $(w_1,\dots,w_k)$, i.e., such that
\[ [w_i,w_j]\ = \
  \begin{cases}
    1 & \text{ for $i=j$, and}\\
    0 & \text{ for $i\ne j$.}
  \end{cases}
\]
The familiar argument from real and complex linear algebra also works over the
quaternions: we choose a non-zero vector $w\in W$
and normalize it to $w_1= w/|w|$; the orthogonal complement
$W^\perp=\{v\in W \ : \ [v,w]=0\}$ is then a $\mK$-subspace that has an orthonormal basis by induction
over the dimension.
A choice of orthonormal basis of $W$ provides a $\mK$-linear isometry $\mK^k\iso W$,
where $k=\dim_\mK(W)$.
So up to $\mK$-linear isometry, the standard examples $\mK^k$ are the only examples
of $\mK$-inner product spaces.

\begin{rk}[Adjoints]\label{rk:adjoint}
  Because inner products are positive-definite, they are in particular
  non-degenerate, i.e., they provide an identification with the dual vector space.
  Given a $\mK$-vector space, we write
  \[ W^\vee\ = \ \Hom_{\mK}(W,\mK)\]
  for the $\mR$-vector space of right $\mK$-linear maps.
  We give $W^\vee$ a right $\mK$-action by
  \[ (f\cdot \lambda)(w)\ = \ \bar\lambda \cdot f(w) \]
  for $f\in W^\vee$, $\lambda\in\mK$ and $w\in W$. If $W$ is finite-dimensional and
  $[-,-]$ is a $\mK$-inner product on $W$, then the map
  \[ W \ \to \ W^\vee \ , \quad w \ \longmapsto \ [w,-] \]
  is a $\mK$-linear isomorphism.

  Now we let $X:W\to V$ be a $\mK$-linear map between $\mK$-inner product spaces.
  The {\em adjoint} of $X$ is the $\mK$-linear map $X^*:V\to W$ that makes the following square commute:
  \[ \xymatrix{
      V \ar[r]^-{X^*}\ar[d]_{v\mapsto [v,-]}^\iso &
      W\ar[d]^{w\mapsto [w,-]}_\iso \\
      V^\vee \ar[r]_-{X^\vee} & W^\vee
    } \]
  So the adjoint is characterized by the relation
  \[ [X^*v,w]\ = \ [v, X w] \]
  for all $v\in V$ and $w\in W$. Passage to the adjoint is $\mR$-linear, contravariantly functorial,
  and involutive.
\end{rk}

Given two $\mK$-inner product spaces $V$ and $W$, we write $\bL^\mK(V,W)$
for the Stiefel manifold of $\mK$-linear isometric embeddings,
i.e., right $\mK$-linear maps $A:V\to W$ that satisfy
$[A v,A v'] = [v,v']$
for all $v,v'\in V$. An equivalent condition is to demand that $A^*\cdot A=\Id_V$.
In the special case $V=W$, we also write
\[ I(W)\ = \ \bL^\mK(W,W) \]
and refer to this as the {\em isometry group} of $W$.
The traditional names for the isometry group are of course
the orthogonal group $O(W)$ in the case $\mK=\mR$,
the unitary group $U(W)$ in the case $\mK=\mC$,
and the symplectic group $Sp(W)$ in the case $\mK=\mH$.

\begin{rk}[Adjoint representation]
  We let $W$ be a finite-dimensional $\mK$-vector space. The {\em exponential map} 
  \begin{equation}\label{eq:exponential}
 \exp\ : \ \End_\mK(W) \ \to \ G L_\mK(W)
    \text{\qquad is given by\qquad}
    \exp(X)\ = \ \sum_{k\geq 0} X^k/k!\ .    
  \end{equation}
  If the endomorphisms $X$ and $Y$ commute, then $\exp(X+Y)= \exp(X)\cdot \exp(Y)$.
  For a $\mK$-inner product space $W$, we write 
  \[  \mathfrak{ad}(W)\ = \ \{ X\in \End_\mK(W)\ : \ X^*=-X\} \]
  for the $\mR$-vector space of its skew-adjoint endomorphisms.
  Since $(X^*)^k=(X^k)^*$ for all $k\geq 0$, also $\exp(X^*)=\exp(X)^*$.
  So if $X\in \mathfrak{ad}(W)$ is skew-adjoint, then
  \[\exp(X)^*\cdot \exp(X)\ = \ \exp(X^*)\cdot \exp(X)\ = \
    \exp(-X)\cdot \exp(X)\ = \ \exp(0)\ = \ \Id_W\ .  \]
  In other words, the exponential map restricts to a smooth map
  \begin{equation}\label{eq:exp_on_ad}
    \exp\ : \ \mathfrak{ad}(W) \ \to \ I(W)     
  \end{equation}
  from the skew-adjoint endomorphisms to the isometry group.
  This map is in fact a local diffeomorphism around the origin,
  and it exhibits $\mathfrak{ad}(W)$
  as the adjoint representation of the compact Lie group $I(W)$,
  whence the notation.
\end{rk}

\Danger  Due to the non-commutativity of the quaternions, the notions of `eigenvalues' and `eigenvectors'
  of $\mH$-linear endomorphisms are somewhat problematic.
  Indeed, if $X:W\to W$ is an $\mH$-linear endomorphism
  and $\lambda\in\mH$ a scalar, then the `eigenspace'
  \[ \{ w\in W \ : \ X w = w \lambda \} \]
  is an $\mR$-subspace of $W$, but typically {\em not} closed under multiplication by scalars from $\mH$.
  If the scalar $\lambda$ is real, and hence central in $\mH$, then the issue disappears and
  the above is an $\mH$-subspace of $W$.
  So in order to deal with eigenspaces in a uniform way for $\mR$, $\mC$ and $\mH$,
  we should -- and will -- restrict to {\em real} eigenvalues.

\begin{prop}\label{prop:diagonalizable}
  Let $X$ be a self-adjoint endomorphism of a $\mK$-inner product space $W$,
  i.e., $X^*=X$.  Then $X$ is diagonalizable with real eigenvalues and pairwise orthogonal eigenspaces.
\end{prop}
\begin{proof}
  The argument over the real and complex numbers can be found in many text books on linear algebra.
  So we restrict to the less common case $\mK=\mH$ of the quaternions.
  We argue over the dimension of $W$; the induction starts with $W=0$, where there is nothing to show.
  Now we suppose that $W\ne 0$.
  We treat $X$ as a $\mC$-linear endomorphism of the underlying $\mC$-vector space
  of $W$. There is then an eigenvector $w\in W\setminus\{0\}$ and a complex scalar $\lambda\in\mC$ such that
  $X w= w\lambda$.
  Because $X$ is self-adjoint we deduce that
  \[
    [w,w]\lambda \ = \
    [w, w\lambda]\ = \  [w, X w]\ = \ [X w, w]\ = \ [w \lambda, w]
    \ = \ \bar\lambda [w, w]\ .
  \]
  Since $w\ne 0$, the inner product $[w,w]$ is a non-zero real number,
  so we must have $\lambda=\bar\lambda$,  i.e., the complex scalar $\lambda$ is in fact real.
  Now we observe that the orthogonal complement $W^\perp$
  of $w$ is invariant under $X$. Indeed, for $v\in W^\perp$ we have
  \[ [X v,w] \ = \  [v, X w]\ = \ [v, w\lambda]\ = \ [v,w]\lambda \ = \ 0 \ .\]
  We can thus apply the inductive hypothesis to the restricted endomorphism  $X|_{W^\perp}:W^\perp\to W^\perp$.
  Since this restriction is diagonalizable 
  with real eigenvalues and pairwise orthogonal eigenspaces,
  the same is true for the original endomorphism $W$.
\end{proof}

We write
\[  \mathfrak{s a}(W)\ = \ \{ X\in \End_\mK(W)\ : \ X^*=X\} \]
for the $\mR$-vector space of self-adjoint endomorphisms of a $\mK$-inner product space $W$.
As mentioned above, the exponential map \eqref{eq:exponential}
commutes with the passage to adjoints;
so if $X$ is self-adjoint, then so is $\exp(X)$.
Hence for $X\in \mathfrak{s a}(X)$, both $X$ and $\exp(X)$ are
diagonalizable with real eigenvalues and pairwise orthogonal eigenspaces,
by Proposition \ref{prop:diagonalizable}.
Moreover, if $X w=w\lambda$, then
\[ \exp(X)w \ = \ \sum_{k\geq 0} X^k w\ = \ \sum_{k\geq 0} w\cdot \lambda^k \ = \ w\cdot \exp(\lambda)\ . \]
So $\exp(X)$ has the same eigenspaces as $X$, but the corresponding eigenvalues are
exponentiated. This process can be reversed as long as the eigenvalues are positive,
so the restricted exponential map
\begin{equation}\label{eq:exp_on_hW}
 \exp \  : \ \mathfrak{s a}(W) \ \xra{\ \iso \ } \ \mathfrak{s a}^+(W) \ = \
  \{ X\in \mathfrak{s a}(W)\ : \ \text{$X$ is positive-definite}\}  
\end{equation}
is a homeomorphism onto the subspace of positive-definite self-adjoint endomorphisms,
i.e, the ones with positive real eigenvalues.

\begin{prop}\label{prop: inverse emb mod m over H}
  Let $V$ and $W$ be $\mK$-inner product spaces. The smooth map
  \[ \bL^\mK(W,V)\times \mathfrak{s a}(W) \ \to \ \Hom_\mK(W,V)\ , \quad
    (A,Z)\ \longmapsto \ A\cdot\exp(-Z) \]
  is an open embedding with image the subspace of $\mK$-linear monomorphisms.
\end{prop}
\begin{proof}
  Since the restricted exponential map \eqref{eq:exp_on_hW} is a homeomorphism, it suffices to show 
  that the composition map
  \[  \circ \ : \  \bL^\mK(W,V)\times \mathfrak{s a}^+(W)  \ \to \ \Hom^{\text{inj}}_\mK(W,V)  \]
  is a homeomorphism.
  If $B:W\to V$ is any $\mK$-linear map, then $B^*\cdot B$ is self-adjoint
  and positive semi-definite. If $B$ is a moreover injective, then 
  $B^*\cdot B$ is even positive-definite.  We write
  \[ \sqrt{\quad}\ : \ \mathfrak{s a}^+(W)\ \to \ \mathfrak{s a}^+(W)\]
  for the homeomorphism that sends a positive-definite self-adjoint endomorphism $X$
  to the unique  positive-definite self-adjoint endomorphism such that $\sqrt{X}\cdot\sqrt{X}=X$.
  Then the continuous map
  \[ \ \Hom_\mK^{\text{inj}}(W,V)\ \to\  \bL^{\mK}(W,V)\times\mathfrak{s a}^+(W)\ , \quad
    B \longmapsto \
    \left( B\cdot\sqrt{ B^*\cdot B}^{-1},\sqrt{B^*\cdot B}\right) \]
  is inverse to the composition map.
\end{proof}

The complex version of the next proposition is \cite[Lemma 1.13]{crabb:U(n) and Omega};
Crabb leaves the proof as an exercise in linear algebra, and we will do the exercise.
Crabb indicates how one can arrive at the formula for the open embedding $\mathfrak c$
in the paragraph after \cite[Lemma 1.15]{crabb:U(n) and Omega}.
Propositions \ref{prop:Cayley invariant} and \ref{prop:Stiefel relative homeo}
together recover Miller's homeomorphism \cite[Theorem A]{miller}
that exhibits the filtration stratum $\bF_k^\mK(W;m)\setminus \bF_{k-1}^\mK(W;m)$
as the total space of the vector bundle over $\bGr_k^\mK(W)$
associated to the $I(k)$-representation $\nu(k,m)\oplus\mathfrak{ad}(k)$.

\begin{prop}\label{prop:Cayley invariant}
  Let $V$ and $W$ be $\mK$-inner product spaces. Then the map
  \[  \mathfrak c\ : \ \Hom_{\mK}(W,V)\oplus \mathfrak{ad}(W) \ \to \ \bL^\mK(W,W\oplus V)  \ ,\quad
    \mathfrak c(Y,X)\  = \ (g,h)      \]
  with
  \begin{alignat*}{3}
    g \ &= \ (X/2 + Y^*Y/4 -1)(X/2 + Y^*Y/4 +1)^{-1} &&: \ W \ \to \ W\\
    h \ &= \ Y (1-g)/2 &&: \ W\ \to \ V
  \end{alignat*}
  is an open embedding onto the subspace of those $f\in\bL^\mK(W,W\oplus V)$
  such that $\ker(f-i_1)=0$.
\end{prop}
\begin{proof}
  It will be convenient to abbreviate
  \[  C \ = \ X/2 + Y^* Y/4  \ ;\]
  we check that the endomorphism $C+1$ has a trivial kernel,
  and is thus invertible.
  Indeed, suppose that $C w=- w$  for some $w\in W$. Then
  \[  - X w/2 \ = \ - (C-Y^*Y/4) w \ = \ w + Y^* Y w/4 \ ,\]
  and hence
  \begin{align*}
    -[w,w]\ &= \ [w, C w] \ = \ [w, (X/2+ Y^* Y/4)w] \ = \ 1/2 [X^*w,w] + 1/4 [Y w, Y w]\\
            &= \  - 1/2[X w,w] + 1/4 [Y w, Y w]\\
    &= \  [w+Y^* Y w/4 ,w] + 1/4 [Y w, Y w]
              \ = \  [w,w] + 1/2 [Y w, Y w]          \geq  0 \ .
  \end{align*}
  So we must have $w=0$. Since $C+1$ is invertible, the definitions of
    \begin{equation}\label{eq:g_in_C}
    g \ = \ (C-1)(C+1)^{-1}
  \end{equation}
  and $h$ make sense.

  Now we show that $g^* g +h^* h=1$, so that $(g,h):W\to W\oplus V$ is indeed a linear isometric embedding.
  Because $X$ is skew-adjoint and $Y^*Y$ is self-adjoint, we have $C+C^*=Y^*Y/2$.
  Relation \eqref{eq:g_in_C} implies
  \begin{equation}\label{eq:1-g}
    1-g \ =\ 2(C+1)^{-1}\ ,      
  \end{equation}
  so
  \[  h^* h \ =  \ (1-g^*) Y^* Y (1-g)/4 \quad  _\eqref{eq:1-g} = \ 2 (C^*+1)^{-1} (C^*+C) (C+1)^{-1}\ . \]
  Hence
  \begin{align*}
    g^* g + h^* h  \  &= \ (C^*+1)^{-1} \left[(C^*-1)(C-1)+ 2 (C^* + C)\right](C+1)^{-1} \\
                     &= \ (C^*+1)^{-1} (C^*+1)(C+1)(C+1)^{-1} \ = \ 1\ .
  \end{align*}
  The morphism $1-g$ is invertible by \eqref{eq:1-g}, and thus has trivial kernel.
  So also the kernel of $(g,h)-i_1=(g-1,h):W\to W\oplus V$ is trivial.
  This concludes the verification that $\mathfrak c$ is a well-defined map with image in the subspace
  of those linear isometric embeddings $f:W\to W\oplus V$ such that $f-i_1$ has trivial kernel.

  Now we exhibit a continuous inverse to $\mathfrak c$. We define 
  \[ \lambda \ : \ \{ (g,h)\in \bL^\mK(W,W\oplus V)\ : \ \ker(1-g)=0\}\ \to \
    \Hom_\mK(W,V)\oplus\mathfrak{ad}(W)  \text{\quad by \quad} \lambda(g,h)\ = \ (Y,X)\ ,\]
  where
  \begin{align*}
    Y\ &= \ 2 h (1-g)^{-1} \\
    X\ &= \   2(1-g^*)^{-1}(g-g^*)(1-g)^{-1}\ .
  \end{align*}
We show that $\lambda$ is indeed inverse to $\mathfrak c$.
To verify the relation $\mathfrak c(\lambda(g,h))=(g,h)$
in the first component we set
\begin{align*}
    C \ &= \ X/2 + Y^* Y/4  \\
        &= \  (1-g^*)^{-1}(g-g^*+h^* h)(1-g)^{-1}\\
        &= \  (1-g^*)^{-1}\left( g-g^* + (1-g^* g) \right) (1-g)^{-1} \\
        &= \  (1-g^*)^{-1}( 1-g^*)(1 + g) (1-g)^{-1} \ = \ (1 + g) (1-g)^{-1} \ = \ 1 + 2 g (1-g)^{-1} \ .
\end{align*}
Then
\begin{align*}
  (C-1)\cdot (C+1)^{-1} \  
  &= \  2 g (1-g)^{-1}\cdot \left( 2 ( 1+g(1-g)^{-1})\right)^{-1}\\
  &= \   g (1-g)^{-1}\cdot \left(  1+g(1-g)^{-1}\right)^{-1}\
  = \   g (1-g)^{-1}\cdot (  1-g)\ = \  g \ .
\end{align*}
In the second component, the desired equality is simply
\[  Y (1-g)/2 \ = \   2 h (1-g)^{-1} (1-g) /2 \ = \ h \ .  \]
The final check is the relation $\lambda(\mathfrak c(Y,X))=(Y,X)$.
In the second component we use that
\begin{align*}
  1-g \ &= \ 1-(C-1)(C+1)^{-1}\ = \ 2 (C+1)^{-1}\\
  g-g^*  &= \ 2(C^*+1)^{-1} - 2 (C+1)^{-1}\ = \ 2 (C^*+1)^{-1}(C-C^*)(C+1)^{-1} \ ,
\end{align*}
and hence
\begin{align*}
       2(1-g^*)^{-1}(g-g^*)(1-g)^{-1}\
  &= \ C-C^*\  \ = \ 1/2(X-X^*) \ =  \ X \ .
\end{align*}
In the first component, the desired equality is simply
\[  2 h(1-g)^{-1} \ = \   2  \left( Y (1-g) /2\right)  (1-g)^{-1} \ = \ Y \ . \qedhere \]
\end{proof}

We let $W$ be a $\mK$-inner product space. As in the body of this paper,
we denote the $k$-th term of the eigenspace filtration of the Stiefel manifold $\bL^\mK(W,W\oplus\mK^m)$
by 
\[ \bF_k(W;m)\ = \ \{f\in \ \bL^\mK(W,W\oplus\mK^m)\ : \ \dim_\mK(\ker(f-i_1)^\perp)\leq k\} \ ,\]
where $i_1:W\to W\oplus\mK^m$ is the embedding of the first summand.
We write $I(k)=\bL^\mK(\mK^k,\mK^k)$ for the isometry group of $\mK^k$.
Conjugation ${^\psi(-)}$ by a linear isometric embedding $\psi$ was defined in \eqref{eq:conjugate by varphi}.

\begin{prop}\label{prop:Stiefel relative homeo}
  For every $\mK$-inner product space $W$ and all $k,m\geq 0$, the map
  \[  \bL^\mK(\mK^k,W)\times_{I(k)} \bL^\mK(\mK^k,\mK^{k+m})  \ \to \ \bF_k(W;m) \ ,\quad
    [\psi,f]\ \longmapsto \ {^\psi f} \]
  is a relative homeomorphism from the pair
  \[ \left(
      \bL^\mK(\mK^k,W)\times_{I(k)} \bL^\mK(\mK^k,\mK^{k+m}),\
      \bL^\mK(\mK^k,W)\times_{I(k)}  \bF_{k-1}(\mK^k;m)  \right)\]
  to the pair $(\bF_k(W;m),\bF_{k-1}(W;m))$.
\end{prop}
\begin{proof}
  Since all spaces involved are compact, we only need to show that
  the map in question restricts to a homeomorphism from
\[  \bL^\mK(\mK^k,W)\times_{I(k)} ( \bL^\mK(\mK^k,\mK^{k+m}) \setminus \bF_{k-1}(\mK^k;m)) \]
  to the complement of $\bF_{k-1}(W;m)$ in $\bF_k(W;m)$.
  We consider a linear isometric embedding $g:W\to W\oplus \mK^m$
  in $\bF_k(W;m) \setminus \bF_{k-1}(W;m)$, i.e., so that 
  \[ \dim_\mK(\ker(g-i_1)^\perp)\ =\ k\ . \]
  We choose a $\mK$-linear isometry
  \[ \psi \ : \ \mK^k \ \xra{\ \iso \ } \ \ker(g-i_1)^\perp\ . \]
  Then $g$ restricts to a linear isometric embedding
  \[ \bar g \ : \ \ker(g-i_1)^\perp \ \to \  \ker(g-i_1)^\perp\oplus\mK^m\ ,  \]
  and we can define the linear isometric embedding  $f:\mK^k\to \mK^{k+m}$ by
  \[ f \ = \  (\psi^{-1}\oplus\mK^m) \circ \bar g \circ\psi \ .\]
  Then $g={^\psi f}$, and $f$ satisfies $\ker(f-i_1)=0$,
  so that $f$ belongs to $\bL^\mK(\mK^k,\mK^{k+m})\setminus \bF_{k-1}(\mK^k;m)$.
  Since the only choice in this construction was the $\mK$-linear isometry $\psi$,
  and any two choices differ by precomposition with an element of $I(k)$, this proves the claim.
\end{proof}

\end{appendix}


\begin{thebibliography}{99}

\bibitem{atiyah-segal:completion}
M F Atiyah, G B Segal, {\em Equivariant $K$-theory and completion.}
J.\ Differential Geom.\ 3 (1969), 1--18.

\bibitem{barrero:operads}
M Barrero, {\em Operads in unstable global homotopy theory.} \verb!arXiv:2110.01674v1!

\bibitem{boardman-vogt:homotopy everything}
J M Boardman, R M Vogt, {\em Homotopy-everything H-spaces.} 
Bull.\ Amer.\ Math.\ Soc.\ 74 (1968), 1117--1122. 

\bibitem{brown:abstract}
  K S Brown, {\em Abstract homotopy theory and generalized sheaf cohomology.}
  Trans. Amer. Math. Soc. 186 (1973), 419--458. 

\bibitem{cisinski:categories_derivables}
  D-C Cisinski, {\em Cat{\'e}gories d{\'e}rivables.}
  Bull. Soc. Math. France 138 (2010), no. 3, 317--393.
  
\bibitem{crabb:U(n) and Omega}
  M C Crabb,
  {\em On the stable splitting of $U(n)$ and $\Omega U(n)$.}
  Algebraic topology, Barcelona, 1986, 35--53,
  Lecture Notes in Math., 1298, Springer, Berlin, 1987.

\bibitem{dhlps}
D Degrijse, M Hausmann, W L{\"u}ck, I Patchkoria, S Schwede, 
{\em Proper equivariant stable homotopy theory.} To appear in {\em Mem. Amer. Math. Soc.}
\verb!arxiv:1908.00779v2!

\bibitem{frankel}
  T Frankel, {\em Critical submanifolds of the classical groups and Stiefel manifolds.}
  1965 Differential and Combinatorial Topology (A Symposium in Honor of Marston Morse) pp. 37--53.
  Princeton Univ. Press, Princeton, N.J.

\bibitem{gepner-henriques}
D Gepner, A Henriques, {\em Homotopy theory of orbispaces}.
\verb!arXiv:math.AT/0701916!

\bibitem{illman}
  S Illman, 
  {\em The equivariant triangulation theorem for actions of compact Lie groups.}
  Math.\ Ann.\  262 (1983), 487--501.

\bibitem{juran}
  B Juran, {\em Orbifolds, orbispaces and global homotopy theory.}
  \verb!arXiv:2006.12374!

\bibitem{karoubi:clifford}
  M Karoubi, {\em Alg{\`e}bres de Clifford et K-th{\'e}orie.}
  Ann. Sci. {\'E}cole Norm. Sup. (4) 1 (1968), 161--270. 

\bibitem{lenz:G-global}
  T Lenz, {\em $G$-global homotopy theory and algebraic K-theory.}
  \verb!arXiv:2012.12676!
  
\bibitem{lind:diagram}
J A Lind, {\em Diagram spaces, diagram spectra and spectra of units.}
Algeb.\ Geom.\ Topol. 13 (2013), 1857--1935.

\bibitem{lurie:higher_algebra}
 J Lurie, {\em Higher algebra.} 
 Version dated September 18, 2017.
 Available from the author's homepage.

\bibitem{may-quinn-ray}
J P May, {\em $E\sb{\infty }$ ring spaces and $E\sb{\infty }$ ring spectra.}
With contributions by F Quinn, N Ray, and J Tornehave. 
Lecture Notes in Mathematics, Vol.\ 77.
Springer-Verlag, Berlin-New York, 1977. 268 pp.

\bibitem{mandell-may}
M A Mandell, J P May, {\em Equivariant orthogonal spectra and $S$-modules.} 
Mem.\ Amer.\ Math.\ Soc. 159 (2002), no.\,755, x+108 pp.

\bibitem{mccord}
M C McCord, {\em Classifying spaces and infinite symmetric products.}
Trans.\ Amer.\ Math.\ Soc.\ {\bf 146} (1969), 273--298. 

\bibitem{miller}
  H Miller, {\em Stable splittings of Stiefel manifolds.}
  Topology 24 (1985), no.\,4, 411--419.

\bibitem{schwede:topological}
S Schwede, {\em The $p$-order of topological triangulated categories.}
J.\ Topol.\ {\bf 6} (2013), no.\ 4, 868--914.

\bibitem{schwede:global}
  S Schwede, {\em Global homotopy theory.}
  New Mathematical Monographs 34.
  Cambridge University Press, Cambridge, 2018. xviii+828 pp.

\bibitem{schwede:universal_Lie}
 S Schwede, {\em Orbispaces, orthogonal spaces, and the universal compact Lie group.}  
 Math. Z. 294 (2020), 71--107.
 
\bibitem{szumilo:frames}
  K Szumi{\l}o, {\em Frames in cofibration categories.}
J. Homotopy Relat. Struct. 12 (2017), no. 3, 577--616.

\bibitem{szumilo:cofibration cats}
  K Szumi{\l}o, {\em Homotopy theory of cofibration categories.}
  Homology Homotopy Appl. 18 (2016), no. 2, 345--357. 

  \bibitem{ullman:equivariant_miller}
    H E Ullman, {\em An equivariant generalization of the Miller splitting theorem.}
    Algebr. Geom. Topol. 12 (2012), no. 2, 643--684. 
  \end{thebibliography}
\end{document}